\documentclass[letterpaper,11pt]{amsart}
\usepackage{amsmath}
\usepackage{amssymb}
\usepackage{amscd}
\usepackage{amsthm}
\usepackage{enumerate}
\usepackage{fullpage}
\usepackage{mathrsfs}
\usepackage{setspace}
\usepackage{mathtools}
\usepackage{xcolor}
\usepackage[all]{xy}
\usepackage{url}
\usepackage[colorlinks=true,linkcolor=blue,breaklinks,pagebackref]{hyperref}
\usepackage[lite,abbrev,alphabetic]{amsrefs}
\usepackage[poorman]{cleveref}
\usepackage{bbm}

\theoremstyle{definition}
\newtheorem{definition}{Definition}[section]

\theoremstyle{plain}
\newtheorem{theorem}[definition]{Theorem}
\newtheorem{proposition}[definition]{Proposition}
\newtheorem{lemma}[definition]{Lemma}
\newtheorem{corollary}[definition]{Corollary}

\newtheorem{condition}[definition]{Condition}
\newtheorem*{theorem*}{Theorem}

\theoremstyle{remark}
\newtheorem{remark}[definition]{Remark}

\crefname{theorem}{Theorem}{Theorems}
\crefname{proposition}{Proposition}{Propositions}
\crefname{lemma}{Lemma}{Lemmas}
\crefname{corollary}{Corollary}{Corollaries}
\crefname{conjecture}{Conjecture}{Conjectures}
\crefname{hypothesis}{Hypothesis}{Hypotheses}
\crefname{remark}{Remark}{Remarks}
\crefname{condition}{Condition}{Conditions}
\crefname{example}{Example}{Examples}

\def\N{\mathbb{N}}
\def\C{\mathbb{C}}
\def\R{\mathbb{R}}
\def\Q{\mathbb{Q}}
\def\Z{\mathbb{Z}}
\def\A{\mathbb{A}}
\def\F{\mathbb{F}}

\def\GL{\mathrm{GL}}
\def\PGL{\mathrm{PGL}}
\def\SL{\mathrm{SL}}

\def\O{\textnormal{O}}
\def\SO{\mathrm{SO}}

\def\U{\textnormal{U}}

\def\Sp{\mathrm{Sp}}

\def\M{\mathrm{M}}

\def\t{\,{}^t}

\def\rd{\,\mathrm{d}}

\def\inf{\infty}

\def\Cc{C_c^{\infty}}

\def\bs{\backslash}

\def\t{\,{}^t}

\def\cf{\mathbf{1}}
\def\AL{{\mathrm{AL}}}
\def\HE{{\mathrm{HE}}}
\def\SM{{\mathrm{SM}}}
\def\Pr{{\mathrm{Pr}}}

\def\1{\mathbf{1}}
\def\0{\mathbf{0}}

\DeclareMathOperator{\diag}{diag} 
\DeclareMathOperator{\ord}{ord}
\DeclareMathOperator{\vol}{vol}

\DeclareMathOperator{\Gal}{Gal}
\DeclareMathOperator{\Res}{Res}
\DeclareMathOperator{\Ind}{Ind}
\DeclareMathOperator{\ind}{ind}
\DeclareMathOperator{\End}{End}
\DeclareMathOperator{\Aut}{Aut}

\DeclareMathOperator{\Irr}{Irr}
\DeclareMathOperator{\re}{Re}
\DeclareMathOperator{\Ad}{Ad}
\DeclareMathOperator{\Nm}{Nm}
\DeclareMathOperator{\tr}{tr}
\DeclareMathOperator{\id}{id}
\DeclareMathOperator{\Sym}{Sym}
\DeclareMathOperator{\As}{As}

\def\trep{{\mathbbm{1}}}

\newcommand{\fa}{\mathfrak{a}}

\newcommand{\fc}{\mathfrak{c}}

\newcommand{\fm}{\mathfrak{m}}
\newcommand{\fn}{\mathfrak{n}}
\newcommand{\fo}{\mathfrak{o}}
\newcommand{\fp}{\mathfrak{p}}

\newcommand{\fS}{\mathfrak{S}}

\newcommand{\fu}{\mathfrak{u}}

\newcommand{\fP}{\mathfrak{P}}
\newcommand{\fT}{\mathfrak{T}}

\newcommand{\fM}{\mathfrak{M}}

\newcommand{\bG}{\mathbb{G}}

\newcommand{\bJ}{\mathbb{J}}

\newcommand{\bM}{\mathbb{M}}
\newcommand{\bS}{\mathbb{S}}

\newcommand{\cA}{\mathcal{A}}

\newcommand{\cG}{\mathcal{G}}
\newcommand{\cH}{\mathcal{H}}

\newcommand{\cL}{\mathcal{L}}

\newcommand{\cO}{\mathcal{O}}

\newcommand{\cQ}{\mathcal{Q}}

\newcommand{\cS}{\mathcal{S}}

\newcommand{\cU}{\mathcal{U}}
\newcommand{\cV}{\mathcal{V}}

\numberwithin{equation}{section}
\newcommand{\disc}{\mathrm{disc}}
\newcommand{\St}{\mathrm{St}}
\newcommand{\geom}{\mathrm{geom}}
\newcommand{\spec}{\mathrm{spec}}

\newcommand{\temp}{\mathrm{temp}}

\newcommand{\el}{\mathrm{el}}

\newcommand{\Temp}{\mathrm{Temp}}
\newcommand{\gen}{\mathrm{gen}}
\newcommand{\ru}{\mathrm{u}}
\newcommand{\EL}{\mathrm{EL}}

\newcommand{\main}{\mathrm{main}}
\newcommand{\negl}{\mathrm{negl}}
\newcommand{\rss}{\mathrm{ss}}
\newcommand{\rns}{\mathrm{ns}}

\title{Asymptotic behavior for twisted traces of self-dual and conjugate self-dual representations of $\GL_n$}

\author{Yugo Takanashi}
\author{Satoshi Wakatsuki}

\address{Yugo Takanashi, Graduate School of Mathematical Sciences, the University of Tokyo, 3-8-1 Komaba, Meguro-Ku, Tokyo 153-8014, Japan}
\email{tknashi@ms.u-tokyo.ac.jp}

\address{Satoshi Wakatsuki, Faculty of Mathematics and Physics, Institute of Science and Engineering, Kanazawa University, Kakumamachi, Kanazawa, Ishikawa, 920-1192, Japan}
\email{wakatsuk@staff.kanazawa-u.ac.jp}

\begin{document}

\begin{abstract}
In this paper, we study the asymptotic behavior of the sum of twisted traces of self-dual or conjugate self-dual discrete automorphic representations of $\GL_n$ for the level aspect of principal congruence subgroups under some conditions. 
Our asymptotic formula is derived from the Arthur twisted trace formula, and it is regarded as a twisted version of limit multiplicity formula on Lie groups.  
We determine the main terms for the asymptotic behavior under different conditions, and also obtain explicit forms of their Fourier transforms, which correspond to endoscopic lifts from classical groups. 
Its main application is the self-dual (resp. conjugate self-dual) globalization of local self-dual (resp. conjugate self-dual) representations of $\GL_n$. 
We further derive an automorphic density theorem for conjugate self-dual representations of $\GL_n$. 
\end{abstract}

\maketitle

\setcounter{tocdepth}{1}

\tableofcontents

\section{Introduction and main theorems}

In this paper, we study the asymptotic behavior of the sum of twisted traces of self-dual or conjugate self-dual discrete automorphic representations of $\GL_n$ for the level aspect of principal congruence subgroups under some conditions. 
This problem is regarded as a twisted version of the limit multiplicity problem, which concerns the asymptotic behaviors of the spectra of the lattices in semisimple Lie groups. 
It originated from the work of DeGeorge and Wallach \cite{dGW78}, and various studies now exist. 
Finis, Lapid and M\"uller have established a limit multiplicity formula without any conditions, which covers $\GL_n$ and a wide class of connected reductive algebraic groups of higher rank, see \cite{FLM15} and \cite{FL19}. 
For the non-twisted case, they studied the asymptotic behavior of the sum of the trace of discrete automorphic representations by using the Arthur trace formula, and they proved that the main term is the contribution of central elements and that the normalized limit coincides with the Plancherel measure. 
As for the twisted case this time, we use the twisted trace formula. 
We determine the main terms under different conditions and also obtain explicit forms of their Fourier transforms, which correspond to endoscopic lifts from classical groups. 
In addition, we prove an automorphic density theorem for conjugate self-dual representations of $\GL_n$.

Note that we do not use the endoscopic classification \cites{Art13,Mok15} in this paper, and our results are obtained directly from the theory of the twisted trace formulas and the intertwining operators. 
In more detail, we will not use the twisted weighted fundamental lemma to stabilize twisted trace formulas, but we stabilize the main terms without the fundamental lemma to obtain their measures. 
The local Langlands correspondence for classical groups will be used only to show that the measures of the main terms agree with the Plancherel measures of twisted endoscopic groups, and this fact is independent of our main theorems and their proofs.

\subsection{Overview of limit multiplicity problem and our results}

Let $\mathbf{G}$ be a semisimple real Lie group with finite center and $\Gamma$ a cocompact discrete subgroup of $\mathbf{G}$. 
Denote by $\Irr_\ru(\mathbf{G})$ the unitary dual of $\mathbf{G}$ with the Fell topology. 
The right regular representation $R_\Gamma$ on the $L^2$-space $L^2(\Gamma \bs \mathbf{G})$ decomposes into the Hilbert direct sum of irreducible unitary representations of $\mathbf{G}$. 
The multiplicity $m(\pi,\Gamma)$ of each element $\pi \in \Irr_\ru(\mathbf{G})$ in the decomposition is finite and we have $m(\pi,\Gamma)=\dim\mathrm{Hom}(\pi,R_\Gamma)$. 
Then, define the discrete spectral measure $\mu_\Gamma$ on $\Irr_\ru(\mathbf{G})$ by 
\begin{equation}\label{eq:intro0}
\mu_\Gamma\coloneqq  \frac{1}{\vol(\Gamma\bs\mathbf{G})} \sum_{\pi\in \Irr_\ru(\mathbf{G})}m(\pi,\Gamma)\cdot \delta_\pi ,    
\end{equation}
where $\delta_\pi$ is the Dirac measure on $\Irr_\ru(\mathbf{G})$. 
Let $\{\Gamma_j\}_{j\in\N}$ be a sequence of normal subgroups of $\Gamma$. 
We suppose that the index of $\Gamma_j$ in $\Gamma$ is finite and the intersection of $\Gamma_j$ for all $j\in\N$ is trivial. 
Then, it is conjectured in \cite{dGW78} that the limit of measures $\mu_{\Gamma_j}$ converges to the Plancherel measure $\mu_\mathrm{pl}$ on $\Irr_\ru(\mathbf{G})$. 
More precisely, the conjecture is that 
\[
\lim_{j\to\inf} \mu_{\Gamma_j}(A)=\mu_{\mathrm{pl}}(A) 
\]
holds for any Jordan measurable subset $A$ of $\Irr_\ru(\mathbf{G})$. 
This limit formula (automorphic Plancherel density theorem) was proved positively in the real rank $1$ case in \cite{dGW79} and in general in \cite{Del86}. 
This formula has been generalized to the ad\`elic setting by \cite{Sau97}, and similar formulas for connected reductive algebraic groups of higher rank have been shown for their discrete spectrum in \cite{Shi12} and \cite{FL19}.

This paper aims to give a twisted version of the above-mentioned limit formula, and the main interest is in what measures appear in analogous limit formulas. 
Since the ad\`elization is essential to explain our discrete measures and their limits, let us first review the limit multiplicity formula of $\GL_n$ for the discrete spectrum by \cite{FLM15}. 
For the sake of simplicity, we restrict our discussion to the rational number field $\Q$, not the algebraic number field. 
Let $\A$ be the ad\`ele ring of $\Q$. 
For the locally compact group $\GL_n(\A)$, we denote by $\GL_n(\A)^1$ the subgroup consisting elements of $\GL_n(\A)$ whose determinant is the id\'ele norm $1$. 
The right regular representation $R_\disc$ on the discrete spectrum of $L^2(\GL_n(\Q) \bs \GL_n(\A)^1)$ decomposes into the Hilbert direct sum of irreducible unitary representations of $\GL_n(\A)^1$. 
Let $\Q_p$ denote the $p$-adic number field for each prime $p$, and let $\Q_\inf$ denote the real number field $\R$ for the infinite place $\inf$. 
Fix a finite set $S$ consisting of $\inf$ and primes. 
Hereafter, $v$ is taken to mean $\inf$ or a prime, $\Q_S$ is the product of $\Q_v$ over $v\in S$, and $\A^S$ is the restricted product of $\Q_v$ for all primes $v$ not in $S$. 
Putting $\GL_n(\Q_S)^1\coloneqq \GL_n(\Q_S)\cap\GL_n(\A)^1$, we obtain the isomorphism $\GL_n(\A)^1\simeq \GL_n(\Q_S)^1\times \GL_n(\A^S)$. 
Take a sequence $\{\mathbf{n}_j\}_{j\in\N}$ of natural numbers such that $\mathbf{n}_j\mid \mathbf{n}_{j+1}$ and $\lim_{j\to\inf}\mathbf{n}_j=\inf$. 
Instead of $\Irr_\ru(\mathbf{G})$ and $\Gamma_j$ in the definition \eqref{eq:intro0} of $\mu_{\Gamma_j}$, we consider the unitary dual $\Irr_\ru(\GL_n(\Q_S)^1)$ of $\GL_n(\Q_S)^1$ and the principal congruence subgroup $K^S(\mathbf{n}_j)$ of $\GL_n(\A^S)$. 
Then, a discrete spectral measure $\mu_{K^S(\mathbf{n}_j)}$ is defined by
\begin{equation}\label{eq:intro1}
\mu_{K^S(\mathbf{n}_j)}\coloneqq \frac{\mathrm{vol}(K^S(\mathbf{n}_j))}{\mathrm{vol}(\GL_n(\Q)\bs \GL_n(\A)^1)} \sum_{\pi\simeq \pi_S\otimes\pi^S} m(\pi)\, \dim(V_{\pi^S}(\mathbf{n}_j))\, \delta_{\pi_S},
\end{equation}
where we set $m(\pi)\coloneqq\dim \mathrm{Hom}(\pi,R_\disc)$ and write $V_{\pi^S}(\mathbf{n}_j)$ for the vector subspace consisting of $K^S(\mathbf{n}_j)$-fixed vectors in the representation space of $\pi^S$. 
Assume that every $\mathbf{n}_j$ does not have any primes in $S$ as prime factors. 
For a Jordan measurable subset $A$ in $\Irr_\ru(\GL_n(\Q_S)^1)$, the limit formula
\begin{equation}\label{eq:intro2}
\lim_{j\to\inf} \mu_{K^S(\mathbf{n}_j)}(A)=\mu_{\GL_n(\Q_S)^1}(A) 
\end{equation}
was proved in \cite{FLM15}, where $\mu_{\GL_n(\Q_S)^1}$ is the Phancherel meausre $\mu_{\GL_n(\Q_S)^1}$ on $\Irr_\ru(\GL_n(\Q_S)^1)$.  
For a compactly supported smooth function $f$ on $G(\Q_S)^1$, 
define a function $\hat{f}$ on $\Irr_\ru(\GL_n(\Q_S)^1)$ by $\hat{f}(\pi_S)\coloneqq \tr(\pi_S(f^\vee))$, $f^\vee(g)\coloneqq f(g^{-1})$. 
The main work of \cites{FLM15} is to prove 
\begin{equation}\label{eq:intro3}
\lim_{j\to\inf} \mu_{K^S(\mathbf{n}_j)}(\hat{f})=\mu_{\GL_n(\Q_S)^1}(\hat{f})
\end{equation}
using the Arthur trace formula. 
Applying the argument of \cite{Sau97} to this formula \eqref{eq:intro3}, the automorphic density theorem \eqref{eq:intro2} follows.

For the case of restricting discrete automorphic representations in the sum of \eqref{eq:intro1} to self-dual representations, or restricting them to conjugate self-daul representations under an extension of the base field, we aim to obtain limit formulas such as \eqref{eq:intro2} and \eqref{eq:intro3}. 
Therefore, the twisted trace formula is used to extract such restricted sums. 
In the following, for the sake of simplicity, we leave the base field as $\Q$, abbreviate the conjugate self-dual representations, and give a rough description of our limit formula for the self-dual representations. 
The element $J_n$ of $\GL_n(\Q)$ is given as \eqref{eq:J_n}, and an involution $\theta$ on $\GL_n$ is defined by $\theta(g)\coloneqq J_n{}^t\! g^{-1} J_n^{-1}$ $(g\in G)$. 
Let $\Irr_\ru^\theta(\GL_n(\Q_S)^1)$ denote the subset of elements $\pi_S$ in $\Irr_\ru(G(\Q_S)^1)$ such that $\pi_S\circ\theta\simeq\pi_S$, that is, $\pi_S$ is self-dual. 
The element $\delta_n$ of $\GL_n(\Q)$ is given as \eqref{eq:delta_n}.
For each self-dual discrete automorphic representation $\pi\simeq\pi_S\otimes\pi^S$, we obtain a twisted trace $\tr(\pi(\delta_n)\circ\pi^S(\theta)\mid_{V_{\pi^S}(\mathbf{n}_j)})$ and we can define a discrete spectral measure $\mu_{K^S(\mathbf{n}_j)}^\theta$ as
\begin{equation}\label{eq:intro4}
\mu_{K^S(\mathbf{n}_j)}^\theta\coloneqq \mathrm{vol}(K^S(\mathbf{n}_j)) \sum_{\pi\simeq \pi_S\otimes\pi^S, \; \pi\circ\theta\simeq \pi} m(\pi)\, \tr(\pi(\delta_n)\circ\pi^S(\theta)\mid_{V_{\pi^S}(\mathbf{n}_j)})\, \delta_{\pi_S}.
\end{equation}
A function $\hat{f}^\theta$ on $\Irr_\ru^\theta(G(\Q_S))$ is defined by $\hat{f}^\theta(\pi_S)=\tr(\pi_S(\theta)\circ\pi_S(f^\vee))$. 
From here, we fix a prime $v_0$, set $S_0\coloneqq S\setminus \{v_0\}$, and suppose $f=f_{v_0}\otimes f_{S_0}$. 
In our main result (\cref{thm:maintheorem1}), one of the conditions {\it (A1)} and {\it (A2)} is assumed, but here for simplicity we assume {\it (A2)} $f_{v_0}$ is a matrix coefficient of a supercuspidal representation $\sigma_{v_0}$ of $\GL_n(\Q_{v_0})$. 
When $\sigma_{v_0}$ satisfies one of the conditions (ii), (iii), (iv) (see \S\ref{sec:intro1}), there exist a subdomain $\EL(\GL_n(\Q_{S_0}))$ of $\Irr_\ru^\theta(G(\Q_S))$, an a.e. non-zero Radon measure $\mu_{S_0}$ on $\EL(\GL_n(\Q_{S_0}))$, and a rational polynomial $\mathbb{M}^\theta(\mathbf{n}_j)$ of prime factors of $\fn_j$ such that  
\begin{equation}\label{eq:intro5}
\lim_{j\to\inf}\frac{1}{\mathbb{M}^\theta(\mathbf{n}_j)}\mu_{K^S(\mathbf{n}_j)}^\theta(\hat{f}^\theta)= \mu_{S_0}(\hat{f}_{S_0}^\theta)  .   
\end{equation}
Note that the condition satisfied $\sigma_{v_0}$ determines $\EL(\GL_n(\Q_{S_0}))$, $\mu_{S_0}$, and $\mathbb{M}^\theta(\mathbf{n}_j)$. 
If we suppose the local Langlands correspondence, $\mu_{S_0}$ can be interpreted as the Plancherel measure on classical groups. 
The limit formula \eqref{eq:intro5} is part of \cref{thm:maintheorem1}, which is our first main result. 
\cref{thm:maintheorem1} also deals with the conjugate self-dual representations. 
Our second main result is \cref{thm:globalization}, in which we obtain self-dual and conjugate self-dual globalizations as an application of \cref{thm:maintheorem1}. 
That is, if we take a subset $\mathbf{S}$ of $S$ and suppose that, for each $v\in \mathbf{S}$, a self-dual (resp. conjugate self-dual) $\theta$-discrete representation $\sigma_v$ of $\GL_n(\Q_v)$ satisfies the appropriate conditions, 
then there exists a self-dual (resp. conjugate self-dual) cuspidal automorphic representation whose local component at $v\in \mathbf{S}$ is isomorphic to $\sigma_v$. 
Our third main result is the automorphic density theorem (\cref{thm:density}) for conjugate self-dual representations, which is a twisted analogy of \eqref{eq:intro5}. 
The case of self-dual representations could not be dealt with because of the non-negativity of twisted traces required to prove it by \eqref{eq:intro3}.

\subsection{Limit formula}\label{sec:intro1}

Let us explain our main results more precisely. 
Let $F$ be a number field, and let $E$ be an extension of $F$ with $[E:F]\leq 2$.
Write $\iota$ for the Galois conjugation when $E\neq F$ and the identity when $E=F$.
Set $G\coloneqq\Res_{E/F}\GL_n$, where $\Res_{E/F}$ means the restriction of scalars from $E$ to $F$. 
Hence we consider the two cases: $G(F)=\GL_n(F)$ and $G(F)=\GL_n(E)$. 
Define the involution $\theta$ of $G$ over $F$ as 
\begin{equation}\label{eq:theta}
    \theta(g) = J_n\, \iota\!\left({}^t\! g^{-1}\right) \, J_n^{-1},  \qquad g\in G,    
\end{equation}
where $J_n\in G$ is defined recursively as 
\begin{equation}\label{eq:J_n}
    J_{n+1}=
        \begin{pmatrix}
                           & J_n  \\
        (-1)^n &  
        \end{pmatrix},  \quad 
    J_2=
        \begin{pmatrix}
           0  & 1  \\
        -1 & 0 
        \end{pmatrix}.    
\end{equation}
We also set 
\begin{equation}\label{eq:delta_n}
    \delta_n\coloneqq
    \begin{pmatrix}
        (-1)^{\frac{n-1}{2}}1_{\frac{n+1}{2}} &  \\
         & (-1)^{\frac{n+1}{2}}1_{\frac{n-1}{2}}
    \end{pmatrix}   \;\; \text{if $E=F$ and $n$ is odd},
    \qquad   \delta_n\coloneqq  1_n \;\; \text{otherwise}.
\end{equation}
where $1_n$ denotes the unit matrix of degree $n$. 
For each place $v$ of $F$, denote by $F_v$ the completion of $F$ at $v$, and choose a $\theta$-stable maximal compact subgroup $K_v$ of $G(F_v)$ as in \eqref{eq:K_v}. 
Write $\fo_F$ for the ring of integers of $F$. 
For an ideal $\fn$ of $\fo_F$ and a finite place $v$, define a principal congruence subgroup $K_v(\fn)$ of $K_v$ as in \eqref{eq:K_v()1} and \eqref{eq:K_v()2}.

Let $\A$ denote the ring of ad\`eles of $F$, $|\;|$ the id\`ele norm on $\A^\times$, and 
\[
G(\A)^1\coloneqq \{ g\in G(\A) \mid |\det(g\, \iota(g))|=1\}. 
\]
Fix an infinite place $v_{\infty,1}$ of $F$ and an embedding $\R_{>0}\hookrightarrow F_{v_{\inf,1}}^\times$. 
Then, the group $\R_{>0}$ is identified with the subgroup $\{a 1_n \mid a\in \R_{>0} \}$ of $\GL_n(F_{v_{\infty,1}})$, and we have an isomorphism $G(\A)^1\times \R_{>0}\simeq G(\A)$. 
Denote by $\Irr_\ru^\theta(G(\A))$ (resp. $\Irr_\ru^\theta(G(\A)^1)$) the set of equivalence classes of irreducible $\theta$-stable (i.e., self-dual or conjugate self-dual) unitary representations of $G(\A)$ (resp. $G(\A)^1$). 
Since any representation $\pi$ of $\Irr_\ru^\theta(G(\A))$ satisfies that the restriction of $\pi$ to $\R_{>0}$ is trivial by the $\theta$-stability, $\Irr_\ru^\theta(G(\A))$ is identified with $\Irr_\ru^\theta(G(\A)^1)$. 
For each $\pi\in \Irr_\ru^\theta(G(\A)^1)$, write $m(\pi)$ for the multiplicity of $\pi$ in the discrete spectrum of $L^2(G(F)\bs G(\A)^1)$. 
For each $\pi\simeq \otimes_v\pi_v\in\Irr_\ru^\theta(G(\A))$, we set
\[
\pi_S\coloneqq\bigotimes_{v\in S}\pi_v \quad \text{and} \quad \pi^S\coloneqq \bigotimes_{v\notin S}\pi_v .    
\]
Let $V_{\pi_v}$ denote the representation space of $\pi_v$, set $V_{\pi^S}\coloneqq\otimes_{v\notin S} V_{\pi_v}$ and let $V_{\pi^S}(\fn)$ denote the subspace of $K^S(\fn)$-fixed vectors of $V_{\pi^S}$, where $K^S(\fn)\coloneqq \prod_{v\notin S} K_v(\fn)$.

From now on, choose a finite set $S$ of places of $F$, and suppose that $S$ contains all infinite places and finite places such that $E_v$ is ramified over $F_v$, where $E_v\coloneqq \prod_{w\mid v} E_w$. 
Set $F_S\coloneqq\prod_{v\in S}F_v$ and $G(F_S)^1\coloneqq G(F_S)\cap G(\A)^1$. 
Then, $G(F_S)^1\times \R_{>0}\simeq G(F_S)$. 
Denote by $\Irr_\ru^\theta(G(F_S))$ the set of equivalence classes of irreducible $\theta$-stable unitary representations of $G(F_S)$ with the Fell topology. 
We identify $\Irr_\ru^\theta(G(F_S))$ with $\Irr_\ru^\theta(G(F_S)^1)$ by the same reason as above. 
Let $\cH(G(F_S))$ denote the space of compactly supported smooth $K_S$-finite functions on $G(F_S)$, where $K_S\coloneqq \prod_{v\in S} K_v$.  
For each $f\in\cH(G(F_S))$, we define a function $\hat{f}^\theta$ on $\Irr_\ru^\theta(G(F_S))$ by $\hat{f}^\theta(\tau_S)=\tr(\tau_S(\theta)\circ\tau_S(f^\vee))$ $(\tau_S\in \Irr_\ru^\theta(G(F_S)))$, where $f^\vee(g)\coloneqq f(g^{-1})$. 
Take a decreasing sequence $\fn_1 \supset \fn_2 \supset \cdots$ of ideals of $\fo_F$ which satisfy 
\begin{equation}\label{eq:fn_j}
    \text{$\fn_j$ is prime to } \prod_{\substack{ v\in S \\ v<\infty}} \fp_v \text{ for any $j\in\N$} \quad \text{and} \quad      \lim_{j\to\infty}\Nm(\fn_j)=\infty,
\end{equation}
where $\fp_v$ means the prime ideal of $\fo_F$ corresponding to $v$, $\N\coloneqq \{1,2,3,\dots\}$, and $\Nm(\fn)$ denotes the norm of $\fn$ in $\fo_F$.  
Define a discrete spectral measure $\mu_{K^S(\fn_j)}^\theta$ on $\Irr_\ru^\theta(G(F_S))$ as 
\begin{equation}\label{eq:measure}
\mu_{K^S(\fn_j)}^\theta = \vol(K^S(\fn_j)) \sum_{\pi\simeq\pi_S\otimes\pi^S\in\Irr_\ru^\theta(G(\A)^1)} m(\pi)\,  \tr\left( \pi^S(\delta_n)\circ\pi^S(\theta)|_{V_{\pi^S}(\fn_j)}\right) \, \delta_{\pi_S}    
\end{equation}
where $\delta_{\pi_S}$ is the Dirac measure of $\pi_S$ on $\Irr_\ru^\theta(G(F_S))$.  
Note that $\pi_v(\theta)$ is normalized as in \cite{Art13}*{\S 2.2}, and $\pi(\theta)$ is compatible with it (cf. \cite{Art13}*{Lemma 4.2.3}).

Fix a non-trivial additive character $\psi=\otimes_v\psi_v$ on $F\bs\A$. 
We obtain the $\gamma$-factors $\gamma(s,\pi_v,r,\psi_v)$ below by Shahidi's result \cite{Sha90}*{Theorem 3.5} (see \S\ref{sec:measure}). 
Here, $r$ means a representation of a $L$-group. 
See \cite{CST17} and \cite{Sha20} for the compatibility between the $L$-parameters and the $\gamma$-factors. 
Let $n_{\pi_v, r}$ denote the order of the zero of $\gamma(s, \pi_v, r,\psi_v)$ at $s=0$ and set
\begin{equation}\label{eq:gamma*}
    \gamma^\ast(0, \pi_v, r, \psi_v) \coloneqq \lim_{s\to0}\zeta_{F_v}(s)^{n_{\pi_v, r}}\gamma(s, \pi_v, r, \psi_v),     
\end{equation}   
where $\zeta_{F_v}(s)$ is the local factor of the Dedekind zeta function, see \eqref{eq:localfactor}.

The value $\gamma^\ast(0, \pi_v, r,\psi_v)$ is used to describe measures appearing in the limits of our discrete measures \eqref{eq:measure}. 
In fact, it is known that the Plancherel measure on $\GL_n(\Q_S)^1$ can be described by $\gamma^\ast(0, \pi_v, \mathrm{Ad },\psi_v)$. 
Let us review this fact, since the measures of our limit formulas are given as its analogue. 
Let $\Temp(\GL_n(F_v))$ denote the set of tempered elements of $\Irr_\ru(\GL_n(F_v))$. 
The measure $\rd_{\Temp(\GL_n(F_v))}(\pi_v)$ on $\Temp(\GL_n(F_v))$ is defined by \eqref{eq:meagl}, and a finite group $S_{\pi_v}$ is defined by \eqref{eq:Sgroup}. 
Following \cite{BP21a}*{Proof of Proposition 2.132}, the Harish-Chandra Plancherel measure $\mu_{\GL_n(F_v)}$ on $\Temp(\GL_n(F_v))$ is described as
\begin{equation}\label{eq:plangl}
\mu_{\GL_n(F_v)}(\hat{f})=c_v\, \int_{\Temp(\GL_n(F_v))} \hat{f}(\pi_v) \, \frac{\chi_{\pi_v}(-1)^{n-1} \gamma^*(0,\pi_v,\mathrm{Ad},\psi_v)}{|S_{\pi_v}|} \, \rd_{\Temp(\GL_n(F_v))}(\pi_v)    
\end{equation}
for some constant $c_v\in\C^\times$, where $\hat{f}$ is a test function on $\Temp(\GL_n(F_v))$ and $\chi_{\pi_v}$ is the central character of $\pi_v$.

We now give the conditions necessary to state the main theorem.
Take a {\it finite} place $v_0\in S$ and an element $\sigma_{v_0}\in\Irr^\theta_\ru(G(F_{v_0}))$. 
Set $S_0\coloneqq S\setminus \{v_0\}$ and $S_j\coloneqq S\sqcup \{v<\infty \mid \fp_v$ divides $\fn_j\}$. 
Assume that either {\it (A1)} or {\it (A2)} holds:
\begin{itemize}  
\item[{\it (A1)}] ${\displaystyle \lim_{j\to\inf}|S_j|<\infty}$ and $\sigma_{v_0}$ is a character twist of the Steinberg representation.  
\item[{\it (A2)}] $\sigma_{v_0}$ is supercuspidal. 
\end{itemize}
Let $f_{v_0}$ denote the specified pseudo-coefficient (resp. matrix coefficient) of $\sigma_{v_0}$ when {\it (A1)} (resp. {\it (A2)}) is assumed, see \S\ref{sec:thetaelliptic} for the definition. 
Take a function $f_{S_0}\in \cH(G(F_{S_0}))$ and set 
\begin{equation}\label{eq:testfunct}
f_S\coloneqq f_{v_0}\otimes f_{S_0}\in \cH(G(F_S)).    
\end{equation}
Write $\cQ_S$ for the set of quadratic characters on $F_S^\times$, and $\trep_S$ the trivial character on $F_S^\times$. 
Let $\cH(G(F_S), \chi_S)$ denote the subset of functions $f_S\in \cH(G(F_S))$ such that $\int_{F_S^\times} f_S(a x)\, \omega_S(a)\, \rd^\times a=0$ $(x\in G(F_S))$ for every $\omega_S\in\cQ_S\setminus\{\chi_S\}$. 
Consider one of the following four cases:
\begin{itemize}
\item[(i)] {\bf Conjugate self-dual.} Suppose that $E\neq F$ and $E_{v_0}$ is not split over $F_{v_0}$. 
Assume one of the following two cases:
\begin{itemize}
\item[$(+)$] $\gamma(0,\sigma_{v_0},\As^+,\psi_{v_0})=0$ and $\gamma(0,\sigma_{v_0},\As^-,\psi_{v_0})\neq 0$. 
\item[$(-)$] $\gamma(0,\sigma_{v_0},\As^-,\psi_{v_0})=0$ and $\gamma(0,\sigma_{v_0},\As^+,\psi_{v_0})\neq 0$. 
\end{itemize}
Let $v\in S_0$. 
When $E_v = F_v \times F_v$, we set $r_v=\mathrm{Std}\otimes\mathrm{Std}^\vee$.  
When $E_v \neq F_v \times F_v$, we set $r_v=\As^+$ for $(+)$, and $r_v=\As^-$ for $(-)$.  

\item[(ii)] {\bf Self-dual $\SO(2m+1)$-type.} Suppose that $E=F$, $n=2m$, and $f_S\in\cH(G(F_S),\trep_S)$. 
We suppose $\gamma(0, \sigma_{v_0}, \wedge^2, \psi_{v_0})=0$.
For every $v\in S_0$, we set $r_v=\wedge^2$. 

\item[(iii)] {\bf Self-dual $\Sp(2m)$-type.}  Suppose that $E=F$ and $n=2m+1$. Take a quadratic character $\chi_S\in \cQ_S$. Suppose that $f\in \cH(G(F_S),\chi_S)$ and $\gamma(0, \sigma_{v_0}, \Sym^2, \psi_{v_0})=0$. 
For every $v\in S_0$, we set $r_v=\Sym^2$. 

\item[(iv)] {\bf Self-dual $\SO(2m)$-type.} Suppose that $E=F$ and $n=2m$. Take a quadratic character $\chi_S\in \cQ_S$. Suppose that $f_S\in \cH(G(F_S),\chi_S)$ and $\gamma(0, \sigma_{v_0}, \Sym^2, \psi_{v_0})=0$. When $\chi_S=\trep_S$, we also suppose that $\gamma(0, \sigma_{v_0}, \wedge^2, \psi_{v_0})\neq 0$. 
For every $v\in S_0$, we set $r_v=\Sym^2$.  
\end{itemize}
\begin{remark}\label{rem:gamma}
    Suppose that $r$ is one of $\As^\pm$, $\wedge^2$, $\Sym^2$ as appropriate. 
    The condition $\gamma(0,\sigma_{v_0},r,\psi_{v_0})=0$ implies that $\sigma_{v_0}$ comes by twisted endoscopic transfer from a twisted endoscopic group, see \cites{Gol94,Sha92}. 
    For Case (i), it is the quasi-split unitary group of degree $n$ over $E_{v_0}/F_{v_0}$. 
    Whether the transfer (base change lift) is standard or non-standard follows from the choice between $(+)$ and $(-)$.
    For Cases (ii), (iii), and (iv), it is the odd special orthogonal group, the symplectic group, and the even special orthogonal group respectively. 
    Hence, the type refers to that twisted endoscopic group.
    
    For Cases (i), (ii), (iv), we can construct a supercuspidal $\theta$-stable representation $\sigma_{v_0}$ which satisfies the above assumptions. See \cite{PR12}*{Proposition 4.1} and \cite{Mie20}*{Proposition 4.5}. 
    Regarding Case (iii), when the residue characteristic of $F_{v_0}$ is not $2$, there are no supercuspidal $\theta$-stable representations, see \cite{Pra99}*{Propositon 5}, but when the residue characteristic of $F_{v_0}$ is $2$, it is possible to construct such representations, see \cite{Pra99}*{Remark 5}. 
    In all case except (iv), there is a character twist $\sigma_{v_0}$ of the Steinberg representation so that $\sigma_{v_0}$ satisfies the above assumptions. 
\end{remark}
For each $v\in S_0$, we will define a subspace $\EL(G(F_v),r_v)$ of $\Irr_{\ru}^\theta(G(F_v))$ in \S\ref{sec:measure2}. 
Regarding {\rm (i)} $E_v=F_v\times F_v$, it is identified with $\Temp(\GL_n(F_v))$. 
If we suppose the local Langlands correspondence for classical groups, it agrees with the image of the tempered dual of a classical group via an endoscopic lift, see \S\ref{sec:Langlands}. 
In each case, the measure $\mu_{\EL(G(F_v),r_v)}$ on $\EL(G(F_v),r_v)$ is defined in \S\ref{sec:Fourier}, see
\begin{itemize}
    \item[\eqref{eq:measureGL}] for {\rm (i)} $E_v=F_v\times F_v$, 
    \item[\eqref{eq:measureU}] $r_\chi=\As^\pm$ for {\rm (i)} $E_v\neq F_v\times F_v$ $(\pm)$ (double-sign corresponds),
    \item[\eqref{eq:measureSOodd}] for {\rm (ii)}, 
    \item[\eqref{eq:measureSpSOeven}] for {\rm (iii)} and {\rm (iv)}.
\end{itemize}
Note that $\mu_{\EL(G(F_v),r_v)}$ is described by $\gamma^*(0,\pi_v,\Ad,\psi_v)/\gamma^*(0,\pi_v,r_v,\psi_v)$ in a similar way to \eqref{eq:plangl}. 
If we suppose the local Langlands correspondence for classical groups, it agrees with the Plancherel measure on the corresponding twisted endoscopic group, see \cites{BP21b,BP21a,HII08} and \S\ref{sec:Langlands}. 
Let $\mu_0$ denote the direct product of the measures $\mu_{\EL(G(F_v),r_v)}$ over $v\in S_0$. 
We set
\[
\bM^S(\fn_j)\coloneqq \begin{cases}
    \fm^S(\fn_j) & \text{when {\rm (i)},} \\
    \tilde\fm^S(\fn_j) & \text{when {\rm (ii)},} \\
    \fm^S(\fn_j)\, \#\cQ(\chi_S,j) & \text{when {\rm (iii)},} \\
    \fm^S(\fn_j) \, c(\chi_S,j) & \text{when {\rm (iv)},}
\end{cases}
\]
where the notations $\fm^S(\fn_j)$, $\tilde\fm^S(\fn_j)$, $\cQ(\chi_S,j)$, $c(\chi_S,j)$ are defined in \eqref{eq:fm^S}, \eqref{eq:tildefm^S}, \eqref{eq:cQ(omega_S,j)}, \eqref{eq:c(omega_S,j)} respectively. 
This factor $\bM^S(\fn_j)$ is a function of $\fn_j$ in order for the limit of discrete measures $\mu^\theta_{K^S(\fn_j)}$ to converge properly.

\begin{theorem}\label{thm:maintheorem1}
Assume the above conditions.
In particular, fix a place $v_0\in S$, assume either (A1) or (A2), and consider one of Cases {\rm (i)}, {\rm (ii)}, {\rm (iii)}, {\rm (iv)}. 
Under Case {\rm (iii)} (resp. {\rm (iv)}), we also suppose $\cQ(\chi_S,j)\neq\emptyset$ (resp. $c(\chi_S,j)\neq 0$) for any large $j$. 
Then there exists a constant $c\in\C^\times$ such that
\begin{equation}\label{eq:maintheorem1}
\lim_{j\to\infty}\frac{1}{\bM^S(\fn_j)} \mu_{K^S(\fn_j)}^\theta\! \left(\hat{h}_S^\theta\right)=c \times \mu_0\!\left(\hat{f}_{S_0}^\theta\right),    
\end{equation}
where $h_S\in\cH(G(F_S))$ is defined by $h_S(g)= f_S(\theta(g^{-1}))$. 
\end{theorem}
\begin{proof}
    Since $\mu_{K^S(\fn_j)}^\theta(\hat{h}_S)=\tr(R_\disc(h_j)\circ R_\disc(\theta))$, see \eqref{eq:R_disc}, 
    an asymptotic formula for $\mu_{K^S(\fn_j)}^\theta(\hat{h}_S)$ is given in \cref{thm:asym}. 

        Case (i) $E\neq F$. By \eqref{eq:testfunct} and \cref{thm:asym} {\it (3)}, the LHS of \eqref{eq:maintheorem1} equals
        \[
        c' \times \left(I^\theta(h_{v_0},\trep_{v_0}) I^\theta(h_{S_0},\trep_{S_0}) +  I^\theta(h_{v_0},\eta_{v_0})\, I^\theta(h_{S_0},\eta_{S_0}) \right)
        \]
        for some constant $c'>0$.
        See \S\ref{sec:toihermi} for the definition of a twisted orbital integral $I^\theta(h_S,\chi_S)$.  
        Here, $\trep_S$ is the trivial representation of $F_S^\times$ and $\eta_S$ is the quadratic character on $F_S^\times$ corresponding to $E_S\coloneqq\prod_{v\in S}E_v$.  
        When $E_v=F_v\times F_v$, the Fourier transform of $I^\theta(h_v,\trep_v)$ is computed in \cref{cor:measureGL}. 
        When $E_v$ is not split over $F_v$, the Fourier transform of $I^\theta(h_v,\chi_v)$ $(\chi_v=\trep_v$ or $\eta_v)$ is computed in \cref{cor:measureU}. 
        Since $E_{v_0}$ is not split over $F_{v_0}$, it follows from the assumptions and \cref{cor:measureU} that
        \[
        \begin{cases}
        I^\theta(h_{v_0},\trep_{v_0})\neq 0 \quad \text{and} \quad I^\theta(h_{v_0},\eta_{v_0})=0 & \text{when $(+)$ is assumed}, \\          
        I^\theta(h_{v_0},\trep_{v_0})= 0 \quad \text{and} \quad I^\theta(h_{v_0},\eta_{v_0})\neq 0 & \text{when $(-)$ is assumed}.          
        \end{cases}
        \]
        Hence, the LHS of \eqref{eq:maintheorem1} equals $c'\times I^\theta(h_{S_0},\chi_{S_0})$, where $\chi_{S_0}=\trep_{S_0}$ when $(+)$, and $\chi_{S_0}=\eta_{S_0}$ when $(-)$. 
        Thus, we obtain the assertion by Corollaries \ref{cor:measureGL} and \ref{cor:measureU}.
        
        Case (ii) $E=F$ and $n$ is even. 
        By \eqref{eq:testfunct} and \cref{thm:asym} {\it (2)}, the LHS of \eqref{eq:maintheorem1} equals $c'\times\tilde{I}^\theta(h_{v_0})\times \tilde{I}^\theta(h_{S_0})$ for some constant $c'>0$.
        See \S\ref{sec:twistedAL} for the definition of a twisted orbital integral $\tilde{I}^\theta(h_S)$.  
        The Fourier transform of $\tilde{I}^\theta(h_v)$, given in \cref{cor:measureSOodd}, completes the proof.
        
        Cases (iii) and (iv) $E=F$. 
        We use \cref{thm:asym} {\it (1)} if $n$ is odd, and \cref{thm:asym} {\it (2)} if $n$ is even. 
        When $n$ is even, we have $\tilde{I}^\theta(h_{v_0})=0$ by the assumptions. 
        Hence, by \eqref{eq:testfunct} and the assumptions, the LHS of \eqref{eq:maintheorem1} equals $c'\times I^\theta(h_{v_0},\chi_{v_0})\times I^\theta(h_{S_0},\chi_{S_0})$ for some constant $c'>0$.
        See \S\ref{sec:vecsp} for the definition of a twisted orbital integral $I^\theta(h_S,\chi_S)$.  
        The Fourier transform of $I^\theta(h_v,\chi_v)$, given in \cref{cor:measureSpSOeven}, completes the proof.
\end{proof}
\begin{remark}
    One can choose a sequence $\{\fn_j\}_{j\in\N}$ for which the assumption for $\cQ(\chi_S,j)$ is satisfied. One can also choose a finite set $S$ and a sequence $\{\fn_j\}_{j\in\N}$ for which the assumption for $c(\chi_S,j)$ is satisfied. 
\end{remark}

\subsection{$\theta$-Elliptic representations, pseudo-coefficients, and matrix coefficients}\label{sec:thetaelliptic}

Before discussing applications of \cref{thm:maintheorem1}, it is necessary to explain $\theta$-elliptic representations, pseudo-coefficients, and matrix coefficients. 
We refer to \cite{MW18}*{\S2.12 in Ch.1} for the definition of $\theta$-elliptic representations.
By the same token, such representations are referred to as $\varepsilon$-discrete representations in \cite{Rog88}*{\S9}.

Let $v$ be a place of $F$. 
Write $\Temp^\theta(G(F_v))$ for the subset of tempered elements in $\Irr^\theta_\ru(G(F_v))$.
Take an element $\sigma_v\in \Temp^\theta(G(F_v))$ and suppose that $\sigma_v$ is $\theta$-elliptic. 
When $E_v=F_v\times F_v$, $\sigma_v$ is square-integrable.
When $E_v\neq F_v\times F_v$, using the fact that every irreducible tempered representation of $G(F_v)$ is induced from an essentially square-integrable representation of a Levi subgroup of $G(F_v)$, cf. \cref{lem:BZ}, we can express $\sigma_v$ as
\begin{equation}\label{eq:sigma}
\sigma_v=\sigma_{v,1}\times \sigma_{v,2} \times \cdots \times \sigma_{v,t},     
\end{equation}
where $n_1+n_2+\cdots+n_t=n$, $\sigma_{v,i}$ is a square-integrable $\theta$-stable representaton of $\GL_{n_i}(E_{v_0})$, and $\sigma_{v,i}$ is inequivalent to $\sigma_{v,j}$ if $i\neq j$. 
Note that $t=1$ if and only if $\sigma_v$ is square-integrable. 
In addition, for each $\pi_v\in\Temp^\theta(G(F_v))$, $\pi_v$ is $\theta$-elliptic if and only if $\pi_v$ belongs to a topologically discrete subspace of $\Temp^\theta(G(F_v))$.

Let $v$ be a finite place.  
By the twisted invariant Paley-Wiener theorem \cites{Rog88,DM08} (see also \cite{MW18}*{\S7 in Ch.1}), for a $\theta$-elliptic representation $\sigma_v$, there exists a pseudo-coefficient $f_v$ of $\sigma_v$, that is, $\hat{f}^\theta_v(\sigma_v)=\delta_{\sigma_v}(\pi_v)$ holds for any $\pi_{v}\in\Temp^\theta(G(F_v))$.  
If $h_{v}(g)\coloneqq f_{v}(\theta(g^{-1}))$, then $\hat{h}^\theta_{v}$ is a pseudo-coefficient of the contragredient representation of $\sigma_{v}$.

When $v$ is finite and $\sigma_v$ is a character twist of the Steinberg representation, we choose a specific pseudo-coefficient of $\sigma_v$. 
For an element $\pi_v\in\Irr_\ru(G(F_v))$ and a character $\omega_v$ on $E_v^\times$, we set $\omega_v\pi_v\coloneqq (\omega_v\circ\det)\otimes\pi_v$, which is called the character twist of $\pi_v$ by $\omega_v$.  
Let $\St_v$ denote the Steinberg representation of $\PGL_n(E_v)$ and $f_{EP,v}$ the pseudo-coefficient of $\St_v$ which is defined in \cite{CC09}*{Proposition 3.8 (i)}. 
Suppose that $\omega_v \St_v$ is $\theta$-stable. 
We obtain a pseudo-coefficient $f_v$ of $\omega_v \St_v$ by $f_v(g)\coloneqq \omega_v(\det(g))f_{EP,v}(g)$ $(g\in G(F_v))$, because $\hat{f}_v^\theta(\pi_v)=\hat{f}_{EP,v}^\theta(\omega_v^{-1}\pi_v)$ $(\pi_v\in\Irr_\ru^\theta(G(F_v)))$. 

Suppose that $v$ is finite and $\sigma_v$ is a $\theta$-stable supercuspidal representation of $G(F_v)$. 
Throughout this paper, a matrix coefficient $\tilde{f}_v$ of $\sigma_v$ must be chosen as $\tilde{f}_v(g)\coloneqq \langle \sigma_v(g)\phi , \phi \rangle$ for a $K_v$-finite $\sigma_v(\theta)$-fixed vector $\phi$ in $V_{\sigma_v}$ where $\langle \; ,\; \rangle$ is an inner product on $V_{\sigma_v}$. 
We can construct a function $f_v\in\cH(G(F_v))$ such that, when $E=F$ (resp. $E\neq F$), we have $\int_{(F_v^\times)^2} f_v(ax) \, \rd a=\tilde{f}_v(x)$ (resp. $\int_{N_{E_v/F_v}(E_v^\times)}f_v(ax) \, \rd a=\tilde{f}_v(x)$) for any $x\in G(F_v)$. 
Here, we write $N_{E_v/F_v}$ for the norm of $E_v$ over $F_v$ when $E\neq F$.  
Then, by multiplying $\phi$ by a constant, we may assume that $f_v$ satisfies the same property as the pseudo-coefficient of $\sigma_v$. 
Note that $f_v$ also satisfies the cuspidal condition \eqref{eq:cuspform}.

Let $v$ be an infinite place. 
If $F_v=\C$ or $E_v=\R\times \R \, (\neq F_v)$, then $G(F_v)$ has no $\theta$-elliptic representations. 
Hence, we suppose $F_v=\R$, and in case of $E\neq F$, we also suppose $E_v=\C$. 
Then, in any of Cases (i)--(iv), $G(F_v)$ always has $\theta$-elliptic representations satisfying the assumption of \cref{thm:globalization}. 
As an example, for Case (ii), we refer to \cite{CC09}*{\S2} for an explicit description of $\theta$-elliptic representations and their $L$-parameters. 
Regarding the real reductive groups, the relation between the $L$-parameters and the $\gamma$-factors are clear, see \cite{Sha85}.

Let $\Irr^\theta_{\gen,\ru}(G(F_v))$ denote the set of equivalence classes of irreducible generic unitary $\theta$-stable representations of $G(F_v)$. 
\begin{lemma}\label{lem:stdm}
    Take a place $v$ of $F$ and a $\theta$-elliptic representation $\sigma_v$ of $G(F_v)$.
    Let $f_v$ be a pseudo-coefficient of $\sigma_v$. 
    For each $\pi_v\in \Irr^\theta_{\gen,\ru}(G(F_v))$, we have $\hat{f}^\theta_v(\pi_v)=1$ if $\pi_v$ is isomorphic to $\sigma_v$, and $\hat{f}^\theta_v(\pi_v)=0$ otherwise. 
\end{lemma}
\begin{proof}
    By \cref{lem:BZ}, the function $\hat{f}^\theta_v$ is identically zero on the topologically non-discrete subspace of $\Irr^\theta_{\gen,\ru}(G(F_v))$. 
    Hence, we obtain the assertion. 
\end{proof}

\subsection{Self-dual and conjugate self-dual globalizations}

The globalization method of supercuspidal representations for the level aspect was established by Clozel \cite{Clo86}, and various methods of globalization exist today (see e.g., \cite{Art13}*{\S6.2} and \cite{Shi12}). 
The self-dual globalization problem for the $\SO(2m+1)$-type and the weight aspect was solved by Chenevier and Clozel in \cite{CC09} by using the twisted trace formula on $\GL_{2m}$ under several conditions. 
Recently, Takanashi proved the conjugate self-dual globalization for the level aspect in \cite{Tak25}*{Theorem 3.11} combining the results in \cites{Clo86,Shi12} with the endoscopic classification \cite{Mok15}.
In this paper, we introduce the self-dual and the conjugate self-dual globalization for the level aspect as the main application of Theorem \ref{thm:maintheorem1}. Our theorem covers the assertions in \cite{Tak25}*{Theorem 3.11} by a different method.

\begin{lemma}\label{lem:1031}
     Assume the same conditions as in Theorem \ref{thm:maintheorem1}. 
    Then,
    \begin{multline}\label{eq:mutheta}
    \mu_{K^S(\fn_j)}^\theta\! \left(\hat{h}_S^\theta\right)=\vol(K^S(\fn_j)) \sum_{\pi} m(\pi)\,  \tr\left( \pi^S(\delta_n)\circ\pi^S(\theta)|_{V_{\pi^S}(\fn_j)}\right) \, \hat{h}_{S_0}^\theta(\pi_{S_0}) \\
    + O\left(\vol(K^S(\fn_j)) \, \Nm(\fn_j)^{[E:F]} \right)  
    \end{multline}
    where $h_{S_0}(g)\coloneqq f_{S_0}(\theta(g^{-1}))$ and $\pi\simeq\pi_{v_0}\otimes\pi_{S_0}\otimes\pi^S$ runs through  cuspidal automorphic representations of $G(\A)$ satisfying the conditions that $\pi_{S_0}\in \Irr^\theta_{\gen,\ru}(G(F_{S_0}))$ and  $\pi_{v_0}\simeq\tilde\sigma_{v_0}$. 
    We remark on 
    \[
    \lim_{j\to\inf}\bM^S(\fn_j)^{-1}\vol(K^S(\fn_j)) \, \Nm(\fn_j)^{[E:F]}=0.
    \]
\end{lemma}
\begin{proof}
    Under {\it(A1)} $\sigma_{v_0}=\omega_{v_0}\St_{v_0}$, by \cite{CC09}*{Propositon 3.8 (i)'} we have $\hat{f}^\theta_{v_0}(\pi_{v_0})=0$ for any $\pi_{v_0}\in\Irr_\ru(G(F_{v_0}))\setminus \{\omega_v\circ\det$, $\sigma_{v_0} \}$. 
    Under {\it(A2)} $\sigma_{v_0}$ is supercuspidal, we have $\hat{f}^\theta_{v_0}(\pi_{v_0})=0$ for any $\pi_{v_0}\in\Irr_\ru(G(F_{v_0}))\setminus\{\sigma_{v_0}\}$. 
    Hence, by \cite{MW89}, we have $\hat{h}_{S}^\theta(\pi_{S})=0$ for every residual automorphic representation $\pi=\pi_S\otimes\pi^S$ of $G(\A)$ except $1$-dimensional representations. 
    Therefore, it is enough to consider cuspidal and $1$-dimensional representations in the sum of $\mu_{K^S(\fn_j)}\! \left(\hat{h}_S^\theta\right)$. 
    
    It is known that all cuspidal autormophic representations $\pi=\otimes_v\pi_v$ of $G(\A)$ are globally generic, that is, all the local representation $\pi_v$ are also generic. 
    Hence, we obtain the first term in RHS of \eqref{eq:mutheta} by \cref{lem:stdm}.

    Let $\A_E$ denote the ad\`ele ring of $E$. 
    Any $1$-dimensional representation is expressed as $\chi\circ\det$ where $\chi=\otimes_v\chi_v$ is a character on $E^\times\bs \A_E^\times$. 
    Suppose that $\chi\circ\det$ is $\theta$-stable and $\otimes_{v\in S}\chi_v$ belongs to a finite set of characters on $E_S^\times$, which is determined by $h_S$.  
    When $E=F$, $\chi$ is quadratic. When $E\neq F$, $\chi$ is constant on the image of $\A_E^\times$ via the norm. 
    In addition, there are only characters that are constant on $\det(K^S(\fn_j))$ in the sum of $\mu_{K^S(\fn_j)}\! \left(\hat{h}_S^\theta\right)$. 
    Hence, the number of such characters is bounded by a constant multiple of $\Nm(\fn_j)^{[E:F]}$.
    This completes the proof. 
\end{proof}

\begin{theorem}\label{thm:globalization}
    Assume the same conditions as in Theorem \ref{thm:maintheorem1}. 
    Take a subset $\mathbf{S}$ of $S_0$. 
    For each $v\in \mathbf{S}$, we also take a $\theta$-elliptic representation $\sigma_v$ of $G(F_v)$ . 
    Let $\tilde\sigma_v$ be the the contragredient representation of $\sigma_v$ for every $v\in\mathbf{S}$. 
    Suppose that $\sigma_v$ is expressed as $\sigma_v=\sigma_{v,1}\times \sigma_{v,2} \times \cdots \times \sigma_{v,t_v}$ in the same way with \eqref{eq:sigma}.
    The following conditions (i)--(iv) correspond to the four cases {\rm(i)--(iv)} of Theorem \ref{thm:maintheorem1} respectively. 
    We assume the corresponding condition for each case.
\begin{itemize}
\item[{\it(i)}] {\bf Conjugate self-dual.} Suppose that $E_v$ is not split over $F_v$ for every $v\in \mathbf{S}$. When $\sigma_{v_0}$ satisfies $(+)$ (resp. $(-)$), we suppose $\gamma(0,\sigma_{v,i},\As^+,\psi_v)=0$ (resp. $\gamma(0,\sigma_{v,i},\As^-,\psi_v)=0$) for every $v\in \mathbf{S}$ and every $i$. 

\item[{\it(ii)}] {\bf Self-dual $\SO(2m+1)$-type.} For every $v\in \mathbf{S}$ and every $i$, we suppose that the central character of $\sigma_{v,i}$ is trivial and we have $\gamma(0, \sigma_{v,i}, \wedge^2, \psi_v)=0$. 

\item[{\it(iii)}] {\bf Self-dual $\Sp(2m)$-type.} For every $v\in \mathbf{S}$ and every $i$, we suppose that the central character of $\sigma_v$ is $\chi_v$ and we have $\gamma(0, \sigma_{v,i}, \Sym^2, \psi_v)=0$. 

\item[{\it(iv)}] {\bf Self-dual $\SO(2m)$-type.} Suppose the same conditions as in {\rm (iii)}. 

\end{itemize}
    Then, there exist a sufficiently large $j\in \N$ and a $\theta$-stable cuspidal automorphic representation $\pi=\otimes_v \pi_v$ of $G(\A)$ such that we have $\pi_v\simeq\tilde\sigma_v$ for every $v\in \mathbf{S}\sqcup\{v_0\}$ and $\tr\left( \pi^S(\delta_n)\circ\pi^S(\theta)|_{V_{\pi^S}(\fn_j)}\right)\neq 0$. 
\end{theorem}
\begin{proof}
For each $v\in \mathbf{S}$, we take a pseudo-coefficient $f_v$ of $\sigma_v$, and set $h_\mathbf{S}\coloneqq\otimes_{v\in \mathbf{S}} h_v$ where $h_v(g)\coloneqq f_v(\theta(g^{-1}))$. 
By \cref{lem:stdm}, for each $\pi_\mathbf{S}\in \Irr^\theta_{\gen,\ru}(G(F_\mathbf{S}))$, we have $\hat{h}_\mathbf{S}^\theta(\pi_\mathbf{S})=1$ if $\pi_\mathbf{S}\simeq \otimes_{v\in\mathbf{S}}\sigma_v$, and $\hat{h}_\mathbf{S}^\theta(\pi_\mathbf{S})=0$ otherwise. 
By the twisted invariant Paley-Wiener theorem \cites{Rog88,DM08}, for $v\in S_0\setminus \mathbf{S}$, there are many functions $f_v\in \cH(G(F_v))$ so that $\mu_{\EL(G(F_v),r_v)}(\hat{f}^\theta_v)\neq 0$, where $h_v(g)\coloneqq f_v(\theta(g^{-1}))$.  
Then, setting $f_S=\otimes_{v\in S} f_v$, we obtain $\mu_0(f_S)\neq 0$. 
Hence, for large $j\in\N$, we have $\mu_{K^S(\fn_j)}\! \left(\hat{h}_S^\theta\right)\neq 0$ by Theorem \ref{thm:maintheorem1}, where $h_S\coloneqq \otimes_{v\in S}h_v$. 
By Lemma \ref{lem:1031}, if $j$ is sufficiently large, $\mu_{K^S(\fn_j)}\! \left(\hat{h}_S^\theta\right)$ is almost equal to a sum of twisted traces of cuspidal representations whose local component at $v\in\mathbf{S}$ is $\tilde{\sigma}_v$.  
Hence, this completes the proof. 
\end{proof}

\subsection{Automorphic density theorem for conjugate self-dual representations of $\GL_n$}

The automorphic Plancherel density theorem for cocompact discrete subgroups was established by Delorme \cite{Del86} and Sauvageot \cite{Sau97}. 
It is known that we can obtain the automorphic Plancherel density theorem by combining Sauvageot's density theorem with the asymptotic formula for traces of Hecke operators. See \cite{Shi12} and \cite{FLM15}*{\S2} for the higher rank case. 
In \S\ref{sec:sau} we will prove a twisted version of Sauvageot's density theorem, and here we present an automorphic density theorem for the conjugate self-dual case. 
Throughout this subsection, we treat only Case (i) $E\neq F$.

\begin{lemma}\label{lem:positivetrace}
    Consider Case {\rm (i)} $E\neq F$. 
    Let $\pi_v$ be an irreducible $\theta$-stable unitary representation of $G(F_v)$. 
    Take a finite place $v$ of $F$ and suppose that $E_v=F_v\times F_v$ and $v$ is not dyadic. 
    Then, $\tr\left(\pi_v(\theta)|_{V_{\pi_v}\left(\fp_v^m\right)}\right)$ is non-negative for any $m\in\Z_{\ge 0}$.
\end{lemma}
\begin{proof}
    Under the assumption we have $G(F_v)= \GL_n(F_v)\times \GL_n(F_v)$ and $\theta(g_1,g_2)=(\theta'(g_2),\theta'(g_1))$ where $\theta'(g)=J_n\t g^{-1}J_n^{-1}$, and there exist $\sigma$, $\sigma'\in\Irr_\ru(\GL_n(F_v))$ such that $\pi_v=\sigma\otimes\sigma'$. 
    Since $\pi_v$ is $\theta$-stable, we may suppose $\sigma'= \sigma\circ\theta'$. Hence the representation space of $\pi_v$ is given by $V_{\sigma}\otimes V_{\sigma}$, and we have $\pi_v(g_1,g_2)\phi_1 \otimes \phi_2=\sigma(g_1)\phi_1\otimes \sigma(\theta'(g_2))\phi_2$. 
    In addition, there exists an inner product $\langle \;, \; \rangle_\sigma$ on $V_\sigma$ such that $\sigma\otimes(\sigma\circ\theta')$ is unitary with respect to the inner product 
    \[
    \langle \phi_1\otimes\phi_2, \phi_1'\otimes\phi_2' \rangle_{\pi_v} \coloneqq \langle \phi_1,\phi_1'\rangle_\sigma \times \langle \phi_2,\phi_2'\rangle_\sigma.
    \]
    Finally, we obtain $\pi_v(\theta)\phi_1\otimes \phi_2=\phi_2\otimes \phi_1$ by the normalization of \cite{Art13}*{\S2.2}, since $\pi_v(\theta)$ is a unitary intertwing operator from $\sigma\otimes(\sigma\circ\theta')$ to $(\sigma\circ\theta')\otimes\sigma$.  
    Hence, for any orthonormal basis $\{\psi_j\}_{j\in \N}$ of $V_\sigma$, it follows that $\{\psi_i\otimes \psi_j\}_{i,j\in \N}$ forms an orthonormal basis of $V_{\pi_v}$. 

    Let $K_v'(\fp_v^m)$ denote the principal congruence subgroup of level $\fp_v^m$ in $\GL_n(\fo_{F,v})$. 
    Since $v$ is not dyadic, we have $K_v(\fp_v^m)=K_v'(\fp_v^m)\times K_v'(\fp_v^m)$. 
    We can choose an orthonormal basis $\{\psi_j\}_{j\in\N}$ extended from an orthonarmal basis of $V_{\sigma}(\fp_v^m)$.  
    Therefore, 
    \[
    \tr(\pi_v(\theta)|_{V_{\pi_v}(\fp_v^m)})=\sum_{\psi_i,\psi_j\in V_{\sigma}(\fp_v^m)}  \langle \psi_j ,\psi_i \rangle_\sigma\times \langle \psi_i ,\psi_j \rangle_\sigma \ge 0. \qedhere
    \]
\end{proof}

We write the prime decomposition of $\fn_j$ as
\begin{equation}\label{eq:level}
    \fn_j=\prod_{v<\inf} \fp_v^{r_{v,j}} ,\quad r_{v,j}\in\Z_{\ge 0}.    
\end{equation} 
\begin{theorem}\label{thm:density}
    Assume the same conditions as in Theorem \ref{thm:maintheorem1} and Case {\rm (i)}. 
    Suppose that for each finite place $v\notin S$, we have $E_v=F_v\times F_v$ if there exists an element $j\in\N$ such that $r_{v,j}>0$. 
    Then, we obtain
    \begin{itemize}
        \item There exists a constant $c \in \C^\times$ such that, for any Jordan measurable subset $A$ in $\Temp^\theta(G(F_{S_0}))$ we have
        \[
        \lim_{j\to\inf} \frac{1}{\bM^S(\fn_j)} \mu_{K^S(\fn_j)}^\theta (A\times \{\sigma_{v_0}\})=c \times \mu_0\left( \bar{A} \right) ,
        \]
        where $\bar{A}$ denotes the set of contragredient representations of elements of $A$. 
        
        \item For any relatively compact subset $B$ in $\Irr_\ru^\theta(G(F_{S_0}))\setminus \Temp^\theta(G(F_{S_0}))$ we have
        \[
        \lim_{j\to\inf} \frac{1}{\bM^S(\fn_j)} \mu_{K^S(\fn_j)}^\theta (B\times \{\sigma_{v_0}\})=0 .
        \]
        
    \end{itemize}
\end{theorem}
\begin{proof}
In this proof, we use some notations given in \S\ref{sec:sau}. 
Take a small positive real number $\varepsilon$. 
Suppose that $f$ is a Riemann integrable function on $\Omega^\theta_\ru(G(F_{S_0}))$, that is, $f$ is bounded and compactly supported and $f$ is continuous almost everywhere with respect to $\mu_0$.    
By \cref{prop:sav1} there exist functions $h_1$, $h_2\in\cH(G(F_{S_0}))$ such that 
\begin{align}
& |f\circ\nu_{S_0}(\pi)-\hat{h}_1^\theta(\pi)|\le \hat{h}_2^\theta(\pi) \quad \text{(a.e. $\pi\in\Irr_{\gen,\ru}^\theta(G(F_{S_0}))$)} , \label{eq:bound16}\\
& |\mu_0|(\hat{u}_2^\theta)<\varepsilon, \label{eq:bound17} 
\end{align}
where $|\mu_0|$ denotes the the total variation of $\mu_0$ and $u_2(g)\coloneqq h_2(\theta(g^{-1}))$.  
By using \cref{lem:1031}, the assumption for $r_{v,j}$, \cref{lem:positivetrace} and \eqref{eq:bound16}, there exists a number $N_1\in \N$ such that that
\begin{equation}\label{eq:bound18}
\left| \mu_{K^S(\fn_j)}^\theta(f\circ\nu_{S_0}\otimes \hat{f}_{v_0}^\theta ) - \mu_{K^S(\fn_j)}^\theta(\hat{h}_1^\theta\otimes \hat{f}_{v_0}^\theta) \right|\le \mu_{K^S(\fn_j)}^\theta(\hat{h}_2^\theta\otimes \hat{f}_{v_0}^\theta)+\varepsilon   
\end{equation}
if $j>N_1$. 
By \cref{thm:maintheorem1} and \eqref{eq:bound17}, there exists a number $N_2\in \N$ so that 
\begin{align}
& \left|\mathbb{M}^S(\fn_j)^{-1}\mu_{K^S(\fn_j)}^\theta(\hat{h}_1^\theta\otimes \hat{f}_{v_0}^\theta)-c\, \mu_0(\hat{u}_1^\theta) \right|<\varepsilon, \label{eq:bound19} \\
& \left|\mathbb{M}^S(\fn_j)^{-1}\mu_{K^S(\fn_j)}^\theta(\hat{h}_2^\theta\otimes \hat{f}_{v_0}^\theta) \right|<(1+c)\, \varepsilon \label{eq:bound20}
\end{align}
if $j>N_2$, where $u_1(g)\coloneqq h_1(\theta(g^{-1}))$. 
We set $N\coloneqq \max\{N_1$, $N_2\}$.

Let us prove the first equation. 
Take a Jordan measurable subset $A$ in $\Temp^\theta(G(F_{S_0}))$. 
By \cref{prop:sav2} (1), there exists a Riemann integrable function $f$ on $\Omega^\theta_\ru(G(F_{S_0}))$ such that $f\circ \nu_{S_0}$ is the characteristic function of $A$ on $\Irr_{\ru}^\theta(G(F_{S_0}))$. 
Denote by $\chi_{\bar{A}}$ the characteristic function of $\bar{A}$. 
By \eqref{eq:bound18} we obtain
\begin{multline}\label{eq:boundtemp}
    \left|\mathbb{M}^S(\fn_j)^{-1}\mu_{K^S(\fn_j)}^\theta(f\circ\nu_{S_0}\otimes \hat{f}_{v_0}^\theta) - c\,\mu_0(\chi_{\bar{A}}) \right| \le \\
    \mathbb{M}^S(\fn_j)^{-1}\left| \mu_{K^S(\fn_j)}^\theta(\hat{h}_2^\theta\otimes \hat{f}_{v_0}^\theta)  \right| + \left|\mathbb{M}^S(\fn_j)^{-1}\mu_{K^S(\fn_j)}^\theta(\hat{h}_1^\theta\otimes \hat{f}_{v_0}^\theta) - c\,\mu_0(\hat{u}_1^\theta) \right| + c\, \left|\mu_0(\hat{u}_1^\theta - \chi_{\bar{A}}) \right|  + \varepsilon
\end{multline}
if $j>N$. 
The inequalities \eqref{eq:bound16} and \eqref{eq:bound17} imply that $|\mu_0|(\hat{u}_1^\theta - \chi_{\bar{A}})\le |\mu_0|(\hat{u}^\theta_2)<\varepsilon$.  
Applying this inequality, \eqref{eq:bound19}, and \eqref{eq:bound20} to \eqref{eq:boundtemp}, we obtain 
\begin{equation*}\label{eq:bound21}
\left|\mathbb{M}^S(\fn_j)^{-1}\mu_{K^S(\fn_j)}^\theta(f\circ\nu_{S_0}\otimes \hat{f}_{v_0}^\theta) - c\,\mu_0(\chi_{\bar{A}})  \right|<(3+2c) \varepsilon    
\end{equation*}
when $j>N$. 
This completes the proof of the first equation together with \cref{lem:1031}.

Next, we will prove the second equation. 
Take a relatively compact subset $B$ in $\Irr_\ru^\theta(G(F_{S_0}))\setminus \Temp^\theta(G(F_{S_0}))$. 
By \cref{prop:sav2} (2), there exists a continuous compactly supported function $f$ on $\Omega^\theta_\ru(G(F_{S_0}))$ such that 
\[
\text{$f\circ \nu_{S_0}\equiv 0$ on $\Temp^\theta(G(F_{S_0}))$,\hspace{2mm} $f\circ \nu_{S_0} \ge 0$ on $\Irr_{\ru}^\theta(G(F_{S_0}))$,\hspace{2mm} and $f\circ \nu_{S_0} \ge 1$ on $B$.}
\]
By using \cref{lem:1031}, the assumption for $r_{v,j}$ and \cref{lem:positivetrace}, there exists a number $N_3\in\N$ so that
\[
\mu_{K^S(\fn_j)}^\theta(B\times \{\sigma_{v_0}\})\le \mu_{K^S(\fn_j)}^\theta(f\circ\nu_{S_0}\otimes \hat{f}_{v_0}^\theta) +\varepsilon 
\]
when $j>N_3$.
Furthermore, by \eqref{eq:bound18}
\begin{align*}
&\mathbb{M}^S(\fn_j)^{-1}\mu_{K^S(\fn_j)}^\theta(f\circ\nu_{S_0})  \\
& \le \, \left|\mathbb{M}^S(\fn_j)^{-1} \mu_{K^S(\fn_j)}^\theta(\hat{h}_2^\theta) \right|  +     \left| \mathbb{M}^S(\fn_j)^{-1}\mu_{K^S(\fn_j)}^\theta(\hat{h}_1^\theta) - c\, \mu_0(\hat{u}_1^\theta) \right|  + c \, \left| \mu_0(\hat{u}_1^\theta) \right|+\varepsilon 
\end{align*}
if $j>N$. 
From \eqref{eq:bound16}, \eqref{eq:bound17}, and $f\circ \nu_{S_0}\equiv 0$ on $\Temp^\theta(G(F_S))$, we see $|\mu_0|(\hat{u}_1^\theta)\le |\mu_0|(\hat{u}_2^\theta)<\varepsilon$. 
Therefore, by \eqref{eq:bound19} and \eqref{eq:bound20}, if $j>\max\{N$, $N_3\}$ we obtain
\[
\mathbb{M}^S(\fn_j)^{-1}\mu_{K^S(\fn_j)}^\theta(B)\le \mathbb{M}^S(\fn_j)^{-1}\mu_{K^S(\fn_j)}^\theta(f\circ\nu_{S_0})+\varepsilon \le (4+2c)\, \varepsilon .
\]
This completes the proof of the second equation. 
\end{proof}
\begin{remark}
\begin{itemize}
    \item If we can prove the non-negativity of twisted traces like \cref{lem:positivetrace}, then we obtain the automorphic density theorems for the other cases (ii)--(iv) in \cref{thm:maintheorem1}. Without such the non-negativity, 
    the following proof does not work, hence we can not prove the assertions above. 
    \item Let $A$ be a relatively compact subset in $\Temp^\theta(G(F_{S_0}))$.  
    It is known that $A$ is Jordan measurable if and only if the boundary $\partial A = A^-\setminus A^\circ$ is a measure zero set, where $A^\circ$ (resp. $A^-$) is the interior (resp. closure) of $A$.  
    This condition of $A$ in \cref{thm:density} is essentially required for automorphic density theorems, see \cite{Sug15}*{Appendix A}.
    \item A gap in Sauvageot's result \cite{Sau97} was pointed out in \cite{NV21}*{p.111}, but the gap was filled by \cite{Tak25+}.
    The gap stems from the fact that the Stone-Weierstrass theorem is only valid in Hausdorf spaces. 
    In the setting of this paper, such a problem does not occur because $\Irr_{\gen,\ru}^\theta(G(F_{S_0}))$ is Hausdorff, see \eqref{eq:inj}. 
\end{itemize}
\end{remark}

\medskip
\noindent
\textbf{Acknowledgments.} 
We would like to express our heartfelt gratitude to Miyu Suzuki for her invaluable support throughout the course of this research. 
In particular, from her, we learned the method in \cite{HII08} and the results in \cite{Sha90} for intertwining operators. 
We would like to thank Atsushi Ichino, Masao Oi, and Shingo Sugiyama for kind advice and helpful discussions. 
We are also grateful to the anonymous referee for many helpful comments. 
The first author is partially supported by JSPS Grant-in-Aid for JSPS Fellows 23KJ0403.
The second author is partially supported by JSPS Grant-in-Aid for Scientific Research (C) No.20K03565, (B) No.21H00972.

\newpage

\part{Asymptotic formula for twisted traces}\label{part:1}

In this part, we prove an asymptotic formula (\cref{thm:asym}) for $\mu_{K^S(\fn_j)}(\hat{h}_S)=\tr(R_\disc(h_j)\circ R_\disc(\theta))$ with respect to $j\to\infty$. In the formula, $\tr(R_\disc(h_j)\circ R_\disc(\theta))$ is described by the reminder term and the sum of the twisted orbit integrals, which become the main terms in the formula of \cref{thm:maintheorem1} under some conditions for each case. 

\section{Setup}

\subsection{Notations}
Let $F$ be a number field. 
Consider the two cases $E=F$ and $E$ is a quadratic extension of $F$. 
Set $G\coloneqq \Res_{E/F}\GL_n$, that is, $G(F)=\GL_n(F)$ if $E=F$, and $G(F)=\GL_n(E)$ if $E\neq F$. 
We write the ring of integers of $F$ and $E$ by $\fo_F$ and $\fo_E$ respectively.
Let $\iota$ denote the Galois conjugation when $E\neq F$ and the identity when $E=F$.
Let $v$ be a place of $F$, and denote by $F_v$ the completion of $F$ at $v$. 
Write $\fo_{F,v}$ for the integer ring of $F_v$.
Set $E_v\coloneqq \prod_{w\mid v} E_w$ and $\fo_{E,v}\coloneqq \prod_{w|v} \fo_{E,w}$. 
A $\theta$-stable maximal compact subgroup $K_v$ of $G(F_v)=\GL_n(E_v)$ is defined by
\begin{equation}\label{eq:K_v}
    K_v=
        \begin{cases}
        \GL_n(\fo_{E,  v}) & \text{if $v<\infty$},  \\
        \text{$\O(n)$ (resp. $\O(n)\times \O(n)$)} & \text{if $E_v=\R$ (resp. $\R\times \R$)},  \\
        \text{$\U(n)$ (resp. $\U(n)\times \U(n)$)} & \text{if $E_v=\C$ (resp. $\C\times \C$)}.
        \end{cases}    
\end{equation}
When $E\neq F$,  we fix an embedding $\nu\,\colon \fo_E\to \M_2(\fo_F)$, which induces an embedding $\GL_n(\fo_E)\to \GL_{2n}(\fo_F)$. 
Let $1_n$ denote the unit matrix of degree $n$. 
For each $v<\infty$ and each ideal $\fn$ of $\fo_F$,  we define the open compact subgroup $K_v(\fn)$ as 
\begin{equation}\label{eq:K_v()1}
    K_v(\fn)=\{ k\in\GL_n(\fo_{F,v}) \mid  k\equiv 1_n \mod \fn\otimes\fo_{F,v} \}    
\end{equation}
when $E=F$, and 
\begin{equation}\label{eq:K_v()2}
    K_v(\fn)=\{ k\in \GL_n(\fo_{E, v}) \mid  \nu(k)\equiv 1_{2n} \mod \fn\otimes\fo_{F,v} \}    
\end{equation}
when $E\neq F$. 


\subsection{Twisted orbital integrals}

Set 
\[
\AL_n\coloneqq\{x\in\M_n \mid \t x=-x\} \quad  \text{and} \quad \HE_n\coloneqq \{x\in\Res_{E/F}\M_n \mid \t\iota(x)=(-1)^{n-1} x\}.
\]
The notation $\AL_n$ (resp. $\HE_n$) means the space of alternating (resp. hermitian) matrices.
Recall the notation $J_n$ defined by \eqref{eq:J_n}. 
Note that we have $\t J_n=(-1)^{n-1}J_n$ and $\det(J_n)=1$.
Let $S$ be a finite set of places of $F$ and set $F_S\coloneqq \prod_{v\in S}F_v$. 


\subsubsection{\texorpdfstring{$E=F$}{} and \texorpdfstring{$n$}{} is even}\label{sec:twistedAL}
The group $G=\GL_n$ acts on $\AL_n$ by
    \[
    x\cdot g= \t g \, x g,   \qquad x\in\AL_n,\ g\in G.
    \]
Then $\AL_n^0:=\{x\in\AL_n \mid \det(x)\neq 0\}$ is a Zariski open orbit and $\AL_n^0(F_S)$ is a single $G(F_S)$-orbit for all place $v$.
Fix a $G(F_S)$-invariant measure $\rd^\times x_S$ on $\AL_n^0(F_S)$,  which is unique up to constant. 
We define the twisted orbital integral of $h_S\in\cH(G(F_S))$ as
    \[
    \tilde{I}^\theta(h_S)= \int_{\AL_n^0(F_S)} 
    h_S(x_S J_n^{-1}) \, \rd^\times x_S. 
    \]

\subsubsection{\texorpdfstring{$E=F$}{}}\label{sec:vecsp}
Set $V\coloneqq\AL_n\oplus \M_{1\times n}$ and we define the action of $G\times \GL_1$ on $V$ by
    \[
    (x_1,x_2)\cdot (g,  a)=(a^2\t g\, x_1 \, g ,  ax_2 \, g),  \qquad 
    (x_1,x_2)\in V , \ g\in G,  a\in\GL_1.
    \]
There is a Zariski open orbit $V^0$.
Note that $V^0(F_S)$ is a single $G(F_S)$-orbit, see \cite{HKK88}.
Fix a $G(F_S)$-invariant measure $\rd^\times x_S$ on $F_S^\times\bs V^0(F_S)$,  which is unique up to constant. 

We define the map $H\,\colon V\to\M_n$ by $H(x_1,x_2)=x_1+\t x_2 x_2$ for $(x_1,  x_2)\in V$.
It is easy to check that
    \[
    H(x\cdot g)=\t g H(x)  g,  \qquad x\in V,\ g\in G.
    \]
For a test funciton $h_S\in\cH(G(F_S))$ and a quadratic character $\chi_S$ of $F_S^\times$, we define the twisted orbital integral as
    \[
    I^\theta(h_S,\chi_S)= \int_{F_S^\times}
    \int_{F_S^\times\bs V^0(F_S)} h_S(a_S H(x_S)J_n^{-1}) \, 
    \chi_S(a_S) \, \rd^\times x_S \, \rd^\times a_S
    \]
where $\rd^\times a_S$ denotes a Haar measure on $F_S^\times$. 

\subsubsection{\texorpdfstring{$E\neq F$}{}}\label{sec:toihermi}
The group $G=\Res_{E/F}\GL_n$ acts on $\HE_n$ by
    \[
    x\cdot g= \t g\, x\, \iota(g),  \qquad x\in\HE_n,\ g\in G.
    \]
Then $\HE_n^0:=\{x\in\HE_n \mid \det(x)\neq 0\}$ is a Zariski open orbit.
Fix a $G(F_S)$-invariant measure $\rd^\times x_S$ on $\HE_n^0(F_S)$,  which is unique up to constant. 
Let $N_{E_S/F_S}$ denote the norm of $E_S$ over $F_S$. 
For a character $\chi_S$ of $F_S^\times$ which is trivial on $N_{E_S/F_S}(E_S^\times)$, we define the twisted orbital integral as
    \[
    I^\theta(h_S,\chi_S)= \int_{\HE_n^0(F_S)} h_S(x_SJ_n^{-1}) \, 
    \chi_S\circ \det(x_S J_n^{-1}) \rd^\times x_S.
    \]

\section{Asymptotic formula}\label{sec:asymptotic}

Take a finite set $S$ of places of $F$.
Suppose that $S$ contains all archimedean places and finite places such that $E_v/F_v$ is a ramified extension.
Take a decreasing sequence $\fn_1 \supset \fn_2 \supset \cdots$ of ideals of $\fo_F$ satisfying \eqref{eq:fn_j}. 
Recall the notation $\delta_n$ defined in \eqref{eq:delta_n}. 
Note that we have $\det(\delta_n)=1$ in all cases.
Let $\cf^S_j\in\cH(G(\A^S))$ given as
    \[
    \cf^S_j= \bigotimes_{v\notin S} \cf_{\delta_nK_v(\fn_j)}
    \]
where $\cf_{D_v}$ denotes the characteristic function of the subset $D_v$ of $G(F_v)$. 
For each $j$,  let $S_j$ be the union of $S$ and the set of finite places of $F$ which divide $\fn_j$.
\begin{condition}\label{cond:asym}
Take a finite place $v_0\in S$ and a function $h_{v_0}\in\cH(G(F_{v_0}))$. 
At least one of {\rm (A1)} and {\rm (A2)} holds:
\begin{itemize}
\item[(A1)] $\lim_{j\to\inf}|S_j|<\infty$ and $h_{v_0}$ is the specified pseudo-coefficient of a character twist of the Steinberg representation of $G(F_{v_0})$.   
\item[(A2)] $h_{v_0}$ is the specified matrix coefficient of an irreducible $\theta$-stable supercuspidal representation of $G(F_{v_0})$. 
\end{itemize}
See \S\ref{sec:thetaelliptic} for the definition of $h_{
v_0}$.
\end{condition}
The conditions (A1) and (A2) correspond to {\it (A1)} and {\it (A2)} of Theorem \ref{thm:maintheorem1} respectively.

Set $S_0\coloneqq S\setminus \{v_0\}$, and take a function $h_{S_0}\in \cH(G(F_{S_0}))$.  
We set
\[
h_S\coloneqq h_{S_0}\otimes  h_{v_0} \in\cH(G(F_S)), \qquad \tilde{h}_j\coloneqq h_S  \otimes  \cf_j^S\in\cH(G(\A)),
\]
    \[
    h_j(g)\coloneqq \int_{\R_{>0}} \tilde{h}_j(ag)\, \rd^\times a \in\cH(G(\A)^1) \qquad (j\in\N).
    \]
Here,  fixing an archimedean place $v_{\infty,1}$ of $F$,  we embed $\R_{>0}$ into $G(\A)$ so that we obtain $G(\A)^1\simeq G(\A)/\R_{>0}$. 
Hereafter we fix Haar measures on $G(\A)$ and $\R_{>0}$ and we take the quotient measure on $G(\A)^1$.

For an ideal $\fn$ of $\fo_F$, we put $K^S(\fn) \coloneqq \prod_{v\notin S} K_v(\fn)$.
For $\pi\simeq \otimes_v\pi_v\in\Irr_\ru^\theta(G(\A))$, set $\pi_S=\otimes_{v\in S}\pi_v$ and $\pi^S=\otimes_{v\notin S}\pi_v$. 
Let $V_{\pi_v}$ denote the representation space of $\pi_v$,  $V_{\pi^S}\coloneqq\otimes_{v\notin S} V_{\pi_v}$ and let $V_{\pi^S}(\fn)$ denote the subspace of $K^S(\fn)$-fixed vectors of $V_{\pi^S}$.

Write $R_\disc$ for the right translation action of $G(\A)^1$ on the discrete spectrum $L_\disc^2(G(F)\bs G(\A)^1)$ of the $L^2$-space $L^2(G(F)\bs G(\A)^1)$. 
Define an operator $R_\disc(\theta)$ on $L_\disc^2(G(F)\bs G(\A)^1)$ by 
\[
(R_\disc(\theta)\phi)(g)= \phi(\theta(g)) , \qquad \phi\in L_\disc^2(G(F)\bs G(\A)^1),\quad g\in G(\A)^1.
\]
Since $R_\disc(\theta)$ is a bounded operator and $R_\disc(h_j)$ is of the trace class (cf. \cite{Mul89}), we see that $R_\disc(h_j)\circ R_\disc(\theta)$ is of the trace class. 
For each $j$,  we have the expansion
\begin{multline}\label{eq:R_disc}
\tr(R_\disc(h_j)\circ R_\disc(\theta))= \\
\vol(K^S(\fn_j)) \sum_{\pi \in \Irr_\ru^\theta(G(\A))} m(\pi)\,  \tr\left( \pi^S(\delta_n)\circ\pi^S(\theta)|_{V_{\pi^S}(\fn_j)}\right) \, \tr(\pi_S(h_S)\circ \pi_S(\theta)),
\end{multline}
where $m(\pi) \coloneqq \dim \mathrm{Hom}_{G(\A)^1} (\pi,L^2(G(F)\bs G(\A)^1))$ is the multiplicity of $\pi$ in $L^2(G(F)\bs G(\A)^1)$. 
We note that the volume of the open compact subgroup $K^S(\fn_j)$ is bounded as 
    \[
    \vol(K^S(\fn_j)) \ll 
        \begin{cases} 
        \Nm(\fn_j)^{-n^2} & \text{when $E=F$},  \\ 
        \Nm(\fn_j)^{-2n^2} & \text{when $E\neq F$}. 
        \end{cases}
    \]

Let $\cQ$ (resp.\,$\cQ_S$) be the set of quadratic characters of $F^\times\R_{>0}\bs \A^\times$ (resp.\,$F_S^\times$). 
For $\omega_S\in\cQ_S$ and $j\in\N$,  we set
\begin{equation}\label{eq:cQ(omega_S,j)}
    \cQ(\omega_S,  j)\coloneqq \left\{ \chi=\otimes_v \chi_v\in\cQ \, \middle| \, 
\begin{array}{l} 
\text{$\otimes_{v\in S}\chi_v=\omega_S$, $\chi_v$ is unramified for every $v\notin S_j$, and} \\
\text{$\chi_v$ is trivial on $\det(K_v(\fn_j))\subset F_v^\times$ for every $v\in S_j\setminus S$ above $2$} 
\end{array}
\right\}.
\end{equation}
For $\omega_S\in\cQ_S$,  let $\cQ'(\omega_S)$ be the set of $\chi=\otimes_v \chi_v\in\cQ$ which satisfies the following two conditions:
\begin{itemize}
\item $\otimes_{v\in S}\chi_v=\omega_S$.
\item $\chi_v$ is unramified for all $v\notin S$.
\end{itemize}
Note that each $\cQ(\omega_S,  j)$ is a finite set, but $\{\#\cQ(\omega_S,j)\}_j$ may not be bounded depending on $\{\fn_j\}_j$.

Set $q_v\coloneqq\Nm(\fp_v)$,  $c_v\coloneqq(1-q_v^{-1})^{-1}$,  and $\mathfrak{e}_v\coloneqq\mathrm{ord}_v(2)$ for $v<\inf$. 
Let $\varpi_v$ denote a prime element in $\fo_{F,v}$. 
In each case,  we set
\begin{equation}\label{eq:fm^S}
 \fm^S(\fn_j) \coloneqq   \Nm(\fn_j)^\fm\times 
        \begin{cases}
         \displaystyle\prod_{v\in S_j\setminus S} c_v^2 \, 
        2^{-2\mathfrak{e}_v-1}  \times \prod_{v\notin S_j} 
        \prod_{t=2}^{m+1} (1-q_v^{-2t+1})
        & \text{$E=F$ and $n=2m+1$},  \\
         \displaystyle\prod_{v\in S_j\setminus S}c_v 
        \times \prod_{v\notin S_j} \prod_{t=2}^{m} (1-q_v^{-2t+1})
        & \text{$E=F$ and $n=2m$},  \\
         \displaystyle\prod_{v\in S_j\setminus S}c_v 
        \times \prod_{v\notin S_j} \prod_{t=2}^n 
        (1-\eta_v(\varpi_v)^{t-1}q_v^{-t})
        & \text{$E\neq F$},  
        \end{cases}
\end{equation}
where $m\in\N$ and
\begin{equation}\label{eq:fmdef}
    \fm\coloneqq
        \begin{cases} 
        -2m^2-3m-1 & \text{$E=F$ and $n=2m+1$,} \\ 
        - 2m^2& \text{$E=F$ and $n=2m$,} \\ 
        -n^2 & \text{$E\neq F$.} 
        \end{cases}
\end{equation}
When $E=F$ and $n=2m$ is even,  we also set 
\begin{equation}\label{eq:tildefm^S}
    \tilde\fm^S(\fn_j)\coloneqq \Nm(\fn_j)^{-2m^2+m} 
    \prod_{v\in S_j\setminus S}c_v \times
    \prod_{v\notin S_j} \prod_{t=2}^{m} (1-q_v^{-2t+1}).    
\end{equation}

The next theorem is a global asymptotic formula for $\tr(R_\disc(h_j)\circ R_\disc(\theta))$.
\begin{theorem}\label{thm:asym}

\noindent (1) Suppose that $E=F$ and $n=2m+1$ is odd.
Then there is a constant $c>0$ which is independent of $j$ such that, for any small $\epsilon>0$,
    \[
    \tr(R_\disc(h_j)\circ R_\disc(\theta))= c \, \fm^S(\fn_j)\sum_{\omega_S \in \cQ_S} \#\cQ(\omega_S,j) \, I^\theta(h_S,\omega_S)  +O_\epsilon\left( \Nm(\fn_j)^{\fm-1+\epsilon} \right).
\]

\noindent (2) Suppose that $E=F$ and $n=2m$ is even.
Then there are constants $\tilde{c}$, $c>0$ which are independent of $j$ such that, for any small $\epsilon>0$,
\begin{multline*}
\tr(R_\disc(h_j)\circ R_\disc(\theta))= \tilde{c}\, \tilde\fm^S(\fn_j)\,  \tilde{I}^\theta(h_S)  \\ 
+c \, \fm^S(\fn_j) \sum_{\omega_S \in \cQ_S} c(\omega_S,j)\, I^\theta(h_S,\omega_S)  +O_\epsilon\left( \Nm(\fn_j)^{\fm-1+\epsilon} \right),
\end{multline*}
where 
\begin{equation}\label{eq:c(omega_S,j)}
    c(\omega_S,j)\coloneqq \sum_{\chi=\otimes_v \chi_v\in \cQ'(\omega_S)}L^S(m,\chi) \, 
    \prod_{v\in S_j\setminus S}\chi_v(\varpi_v^{r_{v,j}}) \qquad \text{(see \eqref{eq:level} for $r_{v,j}$)}. 
\end{equation}

\noindent (3) Suppose that $E\neq F$. 
Let $\eta=\otimes_v\eta_v$ denote the quadratic character on $F^\times\R_{>0}\bs \A^\times$ corresponding to $E$. 
Then there is a constant $c>0$ which is independent of $j$ such that, for any small $\epsilon>0$,
\[
\tr(R_\disc(h_j)\circ R_\disc(\theta))= c \, \fm^S(\fn_j) \left(  I^\theta(h_S,\trep_S) +  I^\theta(h_S,\eta_S) \right) +O_\epsilon\left(\Nm(\fn_j)^{\fm-1+\epsilon}\right), 
\]
where $\trep_S$ is the trivial character on $F_S^\times$ and $\eta_S\coloneqq \otimes_v\eta_v$ (the non-trivial quadratic character on $F_S^\times$).

In each case, the implicit constant that appears in the remainder term is independent of $j$.
\end{theorem}

\section{Proof of Theorem \ref{thm:asym}}

We use the Arthur invariant trace formula \cites{Art88a,Art88b}. 
Since \cref{cond:asym} allows us to exclude the contributions of the continuous spectra on the spectral side (cf. \cref{lem:gspec}), our main task is to estimate an upper bound for each term on the geometric side (cf. \cref{lem:geometricsidesimple}).  
An outline of the proof is as follows:
\begin{itemize}
    \item[\S\ref{sec:preproof}] We explain some notations and notions.
    \item[\S\ref{sec:invtr}] Applying the invariant trace formula to our setting, $\tr(R_\disc(h_j)\circ R_\disc(\theta))$ is expressed as a finite sum of orbital integrals. We explain which terms are negligible and which terms become the main terms. We prove Theorem \ref{thm:asym} by the results presented in the following subsections.  
    \item[\S\ref{sec:intersections}] We clearify which $\theta$-conjugacy classes have intersections with the support of $h_j$ when $j$ tends to $\infty$. 
    \item[\S\ref{sec:Some}] The lemmas presented in \S\ref{sec:Some} are necessary to calculate upper bounds of the negligible terms.   
    \item[\S\ref{sec:UnderA1}] Under Condition \ref{cond:asym} (A1), we study upper bounds of the negligible terms. 
    \item[\S\ref{sec:UnderA2}] Under Condition \ref{cond:asym} (A2), we study upper bounds of the negligible terms.
    \item[\S\ref{sec:Main}] We determine the explicit forms of the main terms.
\end{itemize}

\cref{lem:red} in \S\ref{sec:Some} is the key point of the proof of the upper bounds of the negligible terms. 
There, the integral of the characteristic function of $\delta_nK_v(\fp_v^r)$ over the $\theta$-conjugation of $K_v$ is expressed by the product of the characteristic function of a certain subset $K_v(\fp_v^r)U_v$ with a certain factor $\fM_v(r)$ which has the same magnitude as the leading term (cf. \cref{lem:fmbound}). 
Applying some projections in \S\ref{sec:proj} to the characteristic function of $K_v(\fp_v^r)U_v$, we can estimate upper bounds by arguments similar to the non-twisted case. 

To calculate the main terms, stabilization is required. In this study, we have solved this problem by using the method of the stabilization of the unipotent term in $\SL_2$ by Labesse and Langlands and Saito's formula, which is regarded as the stabilization of prehomogeneous zeta functions.

\subsection{Preliminaries for the proof}\label{sec:preproof}

\subsubsection{Notations}\label{sec:notationproof}

We fix the parabolic subgroup $P_0$ and the Levi subgroup $M_0$ as follows.
\begin{itemize}
    \item $P_0$ is the Borel subgroup of $G$, which consists of upper triangular matrices.
    \item $M_0$ is the Levi component of $P_0$, which consists of diagonal matrices. 
    \item $N_0$ is the unipotent radical of $P_0$. 
\end{itemize}
Note that $P_0$, $M_0$, and $N_0$ are $\theta$-stable. 
We use the following notations.
\begin{itemize}
    \item For any connected reductive algebraic group $\cG$ over $F$, we denote by $\bS_\cG$ the $F$-split part of the center of $\cG$. 
    \item Set $\bS_0\coloneqq \bS_{M_0}$. Then, $\bS_0= M_0$ when $E=F$, and $\bS_0= \{m\in M_0\mid  m=\iota(m) \}$ when $[E:F]=2$. Note that $\bS_0 \simeq \bG_m^{n-1}$ over $F$. 
    \item $X(M_0)$ is the abelian group consisting of $F$-rational characters of $M_0$.
    \item Set $\fa_0^*\coloneqq X(M_0)\otimes \R$. 
    \item $\fa_0$ is the dual space of $\fa_0^*$, that is, $\fa_0=X_*(\bS_0)\otimes\R$. 
    \item $W_0^M$ is the set of linear isomorphisms of $\fa_0$ induced by the $\theta$-conjugations of elements of $M$ which normalize $\bS_0$.
\end{itemize}

Let us introduce notations and concepts related to $F$-rational points in $G$, which are necessary for the description of the geometric side of the trace formula.
\begin{itemize}
    \item $\cL$ is the set of $\theta$-stable standard Levi subgroups in $G$.
    \item Take a Levi subgroup $M\in\cL$. Let $\gamma$, $\gamma'$ in $M(F)$.
    \begin{itemize}
        \item We say that $\gamma$ and $\gamma'$ are $\theta$-conjugate in $M(F)$ if there exists an element $\delta\in M(F)$ such that $\delta^{-1}\gamma\theta(\delta)=\gamma'$. 
        \item We write $M_{\gamma}$ for the connected component of $1$ in the $\theta$-centralizer $\{ g\in M \mid g^{-1}\gamma\theta(g)=\gamma\}$ of $\gamma$ in $M$.
        \item When $\gamma \in M(F)$ is semisimple, we say that $\gamma$ is $F$-elliptic in $M$ if $\bS_{M_\gamma}=\bS_M$.
    \end{itemize} 
    \item $G^+\coloneqq G\rtimes \langle \theta \rangle$ is the extension of $G$ by the group $\langle \theta \rangle \subset \Aut(G)(F)$ generated by $F$-involution $\theta$ on $G$. Note the relation $g\rtimes \theta= (g\rtimes 1)\, (1\rtimes \theta)=  (1\rtimes \theta) (\theta(g)\rtimes 1)$ for $g\in G$. 
    \item We denote the Jordan decomposition of $\gamma\rtimes \theta$ in $G^+$ by $\gamma\rtimes \theta= (\gamma_s\rtimes \theta) \, (\gamma_u\rtimes 1)$, where $\gamma_s\rtimes \theta$ is the semisimple part and $\gamma_u\rtimes 1$ is the unipotent part. 
    \item We say that $\gamma$ and $\gamma' \in G(F)$ are $\cO$-equivalent if $\gamma_s$ and $\gamma'_s$ are $\theta$-conjugate,  \emph{i.e.} $g^{-1}\gamma_s\theta(g)=\gamma_s'$ for some $g\in G(F)$.
    \item $\cO$ is the set of $\cO$-equivalence classes in $G(F)$. 
    \item Take a finite set $\mathscr{S}$ of places of $F$ and an element $\gamma$ in $M(F)$. The semisimple part of $\gamma\rtimes\theta$ is denoted by $\sigma\rtimes\theta$. We say that $\gamma'\in G(F)$ is $(M,\mathscr{S})$-equivalent to $\gamma$ if there exists an element $\delta$ in $M(F)$ so that $\sigma\rtimes\theta$ is the semisimple part of $\delta^{-1}\gamma'\theta(\delta)\rtimes\theta$, and the unipotent parts $\sigma^{-1}\gamma$ and $\sigma^{-1}\delta^{-1}\gamma'\theta(\delta)$ are $\theta$-conjugate in $M_\sigma(F_{\mathscr{S}})$. 
    \item We write $(M(F)\cap \fo)_{M,\mathscr{S}}$ for the $(M,\mathscr{S})$-equivalence classes in $M(F)\cap \fo$. 
\end{itemize}

We record the notations for finite sets of places as follows
\begin{itemize}
    \item We have fixed a finite set $S$ of places of $F$. 
    \item We have also fixed a place $v_0\in S$ and set $S_0=S\setminus \{v_0\}$.
    \item $\fn_j=\prod_{v<\inf} \fp_v^{r_{v,j}}$, $r_{v,j}\in\Z_{\ge 0}$, cf. \eqref{eq:level}.    
    \item We set $S_j=\{v \mid v\in S$ or $r_{v,j}>0 \}$ in the last section.
    \item $\tilde{S}$ is a finite set of places of $F$ which depends only on $h_S$ and is independent of $\fn_j$. 
    \item $S'_j\coloneqq \tilde{S}\cup S_j$. 
    \item In \S\ref{sec:UnderA1}, under \cref{cond:asym} (A1), we fix a finite set $S'$ such that $S_j'\subset S'$ for all $j$, since $S'_j$ does not vary for any large $j$.
    \item In \S\ref{sec:UnderA1}, under \cref{cond:asym} (A1),  $S''=\{ v\in S'\setminus S \mid \lim_{j\to\inf}r_{v,j}=\inf \}$.
\end{itemize}

We also record the notations related to test functions as follows
\begin{itemize}
    \item $\cH(G(F_\mathbf{S}))$ is the space of compactly supported smooth $K_\mathbf{S}$-finite functions on $G(F_\mathbf{S})$, where $\mathbf{S}$ is a finite set of places of $F$ and $K_\mathbf{S}=\prod_{v\in\mathbf{S}}K_v$. 
    \item We have fixed a test function $h_{v_0}\in \cH(G(F_{v_0}))$ which satisfies \cref{cond:asym}. 
    \item We have also fixed a test function $h_{S_0}\in \cH(G(F_{S_0}))$. 
    \item $h_S=h_{v_0}\otimes h_{S_0}\in \cH(G(F_S))$ is independent of $j$. 
    \item $\mathbf{1}_{D_v}$ is the characteristic function of a subset $D_v$ of $G(F_v)$. 
    \item $\mathbf{1}_j^S=\otimes_{v\notin S} \mathbf{1}_{\delta_nK_v(\fn_j)}$. 
    \item $\cH(G(\A)^1)$ (resp. $\cH(G(\A))$) is the space of compactly supported smooth $K$-finite functions on $G(\A)^1$ (resp. $G(\A)$). 
    \item $\tilde{h}_j=h_S\otimes \mathbf{1}_j^S\in\cH(G(\A))$. 
    \item $h_j(g)=\int_{\R_{>0}}\tilde{h}_S\otimes \mathbf{1}_j^S(ag)\,\rd^\times a \in\cH(G(\A)^1)$. 
    \item $U_v=\{ k_v \delta_n k_v^{-1} \mid k_v\in K_v \} \subset K_v$, see \S\ref{sec:Some}.
    \item See \eqref{eq:h_jK} for the notation $h_{j,K}$.
\end{itemize}

\subsubsection{Some projections}\label{sec:proj}

Let $\SM_n\coloneqq \{x\in\M_n \mid x=\t x \}$. 
When $E=F$, following the direct decomposition $\M_n=\AL_n \oplus \SM_n$, define $\Pr_1\colon \M_n \to \AL_n$ and $\Pr_2 \colon \M_n \to \SM_n$ by
\[
\Pr_1(x) = \frac{1}{2}\{ xJ_n-\t(xJ_n) \}, \qquad \Pr_2(x) = \frac{1}{2}\{ xJ_n+\t(xJ_n) \}.
\]
Set $\HE_n'=\{x\in \Res_{E/F}\M_n \mid x=(-1)^n\iota({}^t\! x) \}$. 
When $E\neq F$, following the direct decomposition $\Res_{E/F}\M_n=\HE_n \oplus \HE_n'$, define $\Pr_1\colon \Res_{E/F}\M_n \to \HE_n$ and $\Pr_2 \colon \Res_{E/F}\M_n \to \HE_n'$ by
\[
\Pr_1(x) = \frac{1}{2}\{ xJ_n +(-1)^{n-1} \iota(\t(xJ_n)) \}, \qquad \Pr_2(x) = \frac{1}{2}\{ xJ_n +(-1)^n \iota(\t(xJ_n)) \}.
\]
These projections are $G$-equivariant, that is,
\[
\Pr_j(g \, x \, \theta(g^{-1}))= g \, \Pr_j(x) \, \iota( \t g) ,\qquad j=1,2, \quad x\in \M_n \;\; \text{or} \;\; \Res_{E/F}\M_n ,\quad g\in G. 
\]

\subsection{Invariant trace formula and the proof of Theorem \ref{thm:asym}}\label{sec:invtr}

In this subsection, we briefly explain the invariant trace formula and prove Theorem \ref{thm:asym} using the results of the later subsections.

\subsubsection{References for twisted trace formula}

Arthur's invariant trace formula is established for any reductive algebraic groups over number fields.  
It was proved conditionally in the original papers \cites{Art88a,Art88b}, but now all those conditions have been removed by \cite{DM08} and \cite{KR00}. 
Note that the papers \cites{Art88a,Art88b} deal not only with connected but also with non-connected reductive groups.
The trace formula for non-connected reductive groups includes our case because we are dealing with the non-connected algebraic group $G^+=G\rtimes\langle\theta\rangle$ over $F$. 
We summarize the necessary literature on trace formulas for non-connected reductive groups. 
\begin{itemize}
    \item \cite{LW13} for the coarse expansions. 
    \item \cite{LW13} for the fine spectral expansion. 
    \item \cite{Art86} for the fine geometric expansion. 
    \item \cite{Artlocal} for the non-invariant version of weighted orbital integrals. 
    \item \cite{Art88a} for the invariant version of weighted orbital integrals. 
    \item \cite{Art88b} for the invariant trace formula. 
    \item \cite{Par19} for the continuity of the trace formula. 
\end{itemize}
Regarding \cite{Par19}, his theorem (Theorem 4.1) for the geometric side is conditional, but it covers our setting and a large class, see \cite{Par19}*{\S6}.

\subsubsection{Simple trace formula}
For each $f\in \cH(G(\A)^1)$, we write $I(f)$ for the invariant distribution defined in \cite{Art88b}*{(2.5)} in our setting. 
The invariant trace formula is obtained by giving $I(f)$ two different expansions, called the spectral side $I_\spec(f)$ and the geometric side $I_\geom(f)$, that is, we have
\[
I_\spec(f)=I(f)=I_\geom(f).
\]

We start with the spectral side $I_\spec(h_j)$. 
Take a $\theta$-stable standard Levi subgroup $M\in \cL$ and any place $v$ of $F$. 
Let $P$ be a standard parabolic subgroup of $G$ whose Levi subgroup is $M$. 
We write $N_P$ for the unipotent radical of $P$. 
Let $\rd n_v$ (resp. $\rd k_v$) denote a Haar measure on $N_P(F_v)$ (resp. $K_v$).
For any place $v$, we normalize $\rd k_v$ by $\int_{K_v}\rd k_v = 1$.
If $v$ is finite, we normalize $\rd n_v$ by $\int_{N_P(F_v)\cap K_v}\rd n_v = 1$.  
Let $\mathscr{S}$ be a finite set of places of $F$. 
Set $\rd n_\mathscr{S}\coloneqq \prod_{v\in \mathscr{S}} \rd n_v$ and  $\rd k_\mathscr{S}\coloneqq \prod_{v\in \mathscr{S}} \rd k_v$. 
For each test function $f\in \cH(G(F_\mathscr{S}))$, we set
\[
f_P(m_\mathscr{S})\coloneqq \delta_P(m)^{\frac12} \int_{K_\mathscr{S}}\int_{N_P(F_\mathscr{S})} f(k_\mathscr{S}^{-1}m_\mathscr{S}n_\mathscr{S}\theta(k_\mathscr{S})) \, \rd n_\mathscr{S} \, \rd k_\mathscr{S} \qquad (m_\mathscr{S}\in M(F_\mathscr{S})),
\]
where $K_\mathscr{S}\coloneqq\prod_{v\in \mathscr{S}}K_v$ and $\delta_P$ is the modular function of $P(F_\mathscr{S})\rtimes \theta$. 
The function $f_P$ on $M(F_\mathscr{S})$ belongs to $\cH(M(F_\mathscr{S}))$. 
Let $\Irr_\fu^\theta(M(F_\mathscr{S}))$ denote the subset of $\theta$-stable elements in the unitary dual $\Irr_\fu(M(F_\mathscr{S}))$ of $M(F_\mathscr{S})$. 
Define a function $\hat{f}_M^\theta$ on $\Irr_\fu^\theta(M(F_\mathscr{S}))$ by
\[
\hat{f}_M^\theta(\pi_\mathscr{S})\coloneqq \tr( \pi_\mathscr{S}(\theta)\circ \pi_\mathscr{S}( f_P^\vee)) ,\quad  \pi_\mathscr{S}\in\Irr_\fu(M(F_\mathscr{S})),
\]
where $f_P^\vee(g)\coloneqq f_P(g^{-1})$. 
Under \cref{cond:asym}, $h_{v_0}$ is $\theta$-cuspidal, that is, 
\begin{equation}\label{eq:thetacuspidal}
\hat{h}_{v_0,M}^\theta \equiv 0 \quad \text{on $\Temp^\theta(M(F_S))$} \quad \text{for any $M\in\cL$, $M\neq G$}.     
\end{equation}
Here, $\hat{h}_{v_0,M}^\theta$ means $\hat{f}^\theta_M$ for $f=h_{v_0}$. 
\begin{lemma}\label{lem:gspec}
\[
I_\spec(h_j)=\tr(R_\disc(h_j)\circ R_\disc(\theta)).    
\]
\end{lemma}
\begin{proof}
By \eqref{eq:thetacuspidal} and \cite{Art88b}*{Theorem 7.1 (a)},
\[
I_\spec(h_j)=\sum_{\pi\in\Irr_\fu^\theta(G(\A))} a_\disc^{G\rtimes\theta}(\pi) \, \tr(\pi(h_j)\circ \pi(\theta)),
\]
see \cite{Art88b}*{\S4} for the definition of $a_\disc^{G\rtimes\theta}(\pi)$. 
Since $\tr(\pi_S(h_S)\circ \pi_S(\theta))=\hat{h}_{S_0}^\theta(\pi_{S_0})\times \hat{h}^\theta_{v_0}(\pi_{v_0})$ and 
\[
\tr(\pi(h_j)\circ \pi(\theta))= \vol(K^S(\fn_j)) \,  \tr\left( \pi^S(\delta_n)\circ\pi^S(\theta)|_{V_{\pi^S}(\fn_j)}\right) \, \tr(\pi_S(h_S)\circ \pi_S(\theta)),
\]
it is sufficient to prove $\hat{h}^\theta_{v_0}(\pi_{v_0})=0$ if $a_\disc^{G\rtimes\theta}(\pi)\neq m(\pi)$. 
Assume $a_\disc^{G\rtimes\theta}(\pi)\neq m(\pi)$. 
Then, the irreducible $\theta$-stable unitary representation $\pi=\otimes_v \pi_v$ is an irreducible constituent of the induced representation of a proper Levi subgroup $L$ of $G$ and an element $\sigma\in\Irr_\ru(L(\A))$, see \cite{Art88b}*{p.517}. 
In our setting, it is known that such induced representation is irreducible, see \cite{Ber84}*{Corollary in p.51}. 
Hence, $\pi_{v_0}$ is not supercuspidal, Steinberg or $1$-dimensional, see \cite{CC09}*{Propostion 3.8(i)'}.  
Thus, we obtain $\hat{h}^\theta_{v_0}(\pi_{v_0})=0$, which completes the proof. 
\end{proof}
As this lemma makes it unnecessary to investigate the spectral side $I_\spec(h_j)$ further, the geometric side $I_\geom(h_j)$ will be discussed in more detail.

Since the geometric side is expanded by orbital integrals and global coefficients, we will review their notations and properties. 
Let $\mathscr{S}$ be a finite set of places of $F$, $f\in\cH(G(F_\mathscr{S}))$, and $M\in\cL$. 
\begin{itemize}
    \item According to the references cited below, our twisted weighted orbit integral should be described as $J^{G\rtimes\theta}_{M\rtimes\theta}(\gamma,f)$, but for the sake of simplicity, it is abbreviated as $J^{G}_{M}(\gamma,f)$ in this paper.
    In other cases, it is also abbreviated in the same way.
    Note that all of the weighted orbit integrals dealt with in this paper are twisted, and so writing ``twisted" is abbreviated throughout this section.
    \item $J_M^G(\gamma,f)$ is the weighted orbital integral of $\gamma\in M(F)$ and $f\in \cH(G(F_\mathscr{S}))$, see \cite{Artlocal}*{(2.1), (2.1*), (6.5)} for the definition. 
    When $M\neq G$, $J_M^G(\gamma,f)$ is non-invariant, that is, 
    \[
    J_M^G(\gamma,f)=J_M^G(\gamma,f^y), \quad f^y(x)=f(y^{-1}x\theta(y)), \quad y\in G(F_\mathscr{S})
    \]
    does not hold in general. 
    When $M=G$, $J_G^G(\gamma,f)$ is invariant and called the (non-weighted) orbital integral. 
    \item $I_M^G(\gamma,f)$ is the invariant version of the weighted orbital integral of $\gamma\in M(F)$ and $f\in \cH(G(F_\mathscr{S}))$, see \cite{Art88b}*{\S 2} for the definition. 
    For every $M$, $I_M^G(\gamma,f)$ is invariant. 
    In particular, we have $I_G^G(\gamma,f)=J_G^G(\gamma,f)$. 
    \item Consider the case $M=G$. As we will study upper bounds of $J_G^G(\gamma,f)$, let us recall its definition.
    Let $G_\gamma$ denote the connected component of $1$ in the $\theta$-centralizer $\{g\in G\mid g^{-1}\gamma\theta(g)=\gamma\}$. 
    Choose a Haar measure $\rd g_\gamma$ (resp. $\rd g$) of $G_\gamma(F_\mathscr{S})$ (resp. $G(F_\mathscr{S})$) and denote the quotient measure of $\rd g_\gamma$ and $\rd g$ on $G_\gamma(F_\mathscr{S})\bs G(F_\mathscr{S})$ by $\rd g_\gamma\bs \rd g$. 
    Then,
    \[
    J_G^G(\gamma,f)=I_G^G(\gamma,f)=|D^G(\gamma)|_\mathscr{S}^{\frac{1}{2}}  \int_{G_\gamma(F_\mathscr{S})\bs G(F_\mathscr{S})} f(g^{-1}\gamma \theta(g))\, \frac{\rd g}{\rd g_\gamma}.
    \]
    Here, $|D^G(\gamma)|_\mathscr{S}=\prod_{v\in\mathscr{S}} |D^G(\gamma)|_v$ and see \cite{Art86}*{p.202} for the definition of $D^G(\gamma)$. 
    Only two properties $|D^G(\gamma)|_\mathscr{S}^{\frac12}\ll_\gamma 1$ and $D^G(\gamma)\in F^\times$ will be needed later.     
    \item Suppose $\mathscr{S}=\mathscr{S}_1\sqcup \mathscr{S}_2$ and $f=f_1\otimes f_2$ $(f_j\in\cH(G(F_{\mathscr{S}_j})))$. 
    Then, we have $J_G^G(\gamma,f)=J_G^G(\gamma,f_1)J_G^G(\gamma,f_1)$ by the definition. 
    Regarding weighted orbital integrals, such a simple decomposition does not hold, but there are splitting formulas that expand $J_M^G(\gamma,f)$ and $I_M^G(\gamma,f)$ to sums of weighted orbital integrals of $f_1$ and $f_2$.
    Regarding $I_M^G(\gamma,f)$, we refer to \cite{Art88a}*{Proposition 9.1}.  
    Regarding $J_M^G(\gamma,f)$, in the connected case, the splitting formula is given in \cite{Art05}*{(18.7)}, and it is derived from \cite{Art88a}*{Corollary 7.4}. 
    For the non-connected case, it can be derived in the same way from \cite{Art88a}*{Corollary 7.4}.
    \item  $a^M(\mathscr{S},\gamma)$ is the global coefficient of $\gamma$ in $M$, and it satisfies the following properties, which are proved in \cite{Art86}. See also \cite{Art88b}*{p.509} for its summary. We must suppose that $\mathscr{S}$ contains all infinite places. 
         \begin{itemize}
         \item Suppose that, for some $v\notin \mathscr{S}$, the orbit $\{g^{-1}\gamma \theta(g) \mid g\in M(F_v)\}$ has no intersection with $K_v$. Then, we have $a^M(\mathscr{S},\gamma)=0$. 
         \item If $\gamma_s$ is not $F$-elliptic in $M$, then $a^M(\mathscr{S},\gamma)=0$. 
         \item When $\gamma_s$ is $F$-elliptic in $M$ and $\gamma_u$ is trivial, $a^M(\mathscr{S},\gamma)$ is expressed as a volume of $M_{\gamma}$, and it is independent of $\mathscr{S}$. 
         \item When $\gamma_s$ is $F$-elliptic in $M$ and $\gamma_u$ is non-trivial, $a^M(\mathscr{S},\gamma)$ is unknown in general, and it should depend on $\mathscr{S}$ in general.  
         \end{itemize}
\end{itemize}

There exists a finite subset $\tilde{S}$ of places of $F$ such that $S_j'\coloneqq S_j\cup \tilde{S}$ is large enough to give a fine expansion of $I_\geom(h_j)$ for all $j\in \N$, cf. \cite{Art86}*{Theorem 9.2}. 
Note that $\tilde{S}$ is independent of $j$. 
\begin{lemma}\label{lem:geometricsidesimple}
\begin{equation}\label{eq:ggeom}
I_\geom(h_j)=\sum_{\fo\in\cO}I_\fo(h_j), \qquad 
I_\fo(h_j)=\sum_{M\in\cL} \frac{|W^M_0|}{|W^G_0|} \sum_{\gamma\in (M(F)\cap\fo)_{M,S_j'}} I_{M,\gamma}(h_j),
\end{equation}
\begin{equation}\label{eq:geomfo}
I_{M,\gamma}(h_j)=  a^M(S_j',\gamma)\times I_M^G(\gamma,h_{v_0}) \times  J^M_M(\gamma,(h_{S_0})_P) \times \prod_{v\in S_j'\setminus S} J^M_M(\gamma,(\cf_{\delta_nK_v(\fn_j)})_P) ,
\end{equation}
where $(*)_P$ means $f_P$ for $f=*$.     
See \S\ref{sec:notationproof} for the notations $\cO$, $\cL$, $W_0^G$, and $(M(F)\cap\fo)_{M,S_j'}$.
Note that we have $I_\fo(h_j)=0$ except for finitely many classes $\fo$, and $\# (M(F)\cap\fo)_{M,S_j'}$ is finite, cf. \cite{Art88b}*{Theorem 3.3}.  
\end{lemma}
\begin{proof}
Since this lemma can be proved by the same argument as in \cite{Art89}*{p.269}, here we sketch the proof.  
By the fine expansion \cite{Art88b}*{Theorem 3.3}, 
\[
I_\geom(h_j)=\sum_{\fo\in\cO} \sum_{M\in\cL} \frac{|W^M_0|}{|W^G_0|}\sum_{\gamma\in(M(F)\cap\fo)_{M,S_j'}} a^M(S_j',\gamma) I_M(\gamma,h_{v_0}\otimes \mathbf{h}_j), \quad  \mathbf{h}_j=h_{S_0}\otimes \left( \otimes_{v\in S_j'\setminus S} \cf_{\delta_nK_v(\fn_j)} \right) .
\]
Applying the splitting formula \cite{Art88a}*{Proposition 9.1} to $I_M(\gamma,h_{v_0}\otimes \mathbf{h}_j)$,
\[
I_M(\gamma,h_j)=\sum_{L_1,L_2\in\cL(M)} d_M^G(L_1,L_2)\, \hat{I}_M^{L_1}(\gamma,h^\theta_{v_0,L_1})\, \hat{I}_M^{L_2}(\gamma,\hat{\mathbf{h}}^\theta_{j,L_2}).
\]
We do not explain the notations of this formula for simplicity, but it is sufficient to use \eqref{eq:thetacuspidal}, namely $\hat{I}_M^{L_1}(\gamma,h^\theta_{v_0,L_1})=0$ if $L_1\neq G$, and $d_M^G(G,L_2)=0$ if $L_2\neq M$. 
Hence,
\[
I_M(\gamma,h_j)=\hat{I}_M^G(\gamma,\hat{h}^\theta_{v_0,G})\, \hat{I}_M^M(\gamma,\hat{\mathbf{h}}^\theta_{j,L_2}).
\]
Since $\hat{I}_M^G(\gamma,\hat{h}^\theta_{v_0,G})=I_M^G(\gamma,h_{v_0})$ and $\hat{I}_M^M(\gamma,\hat{\mathbf{h}}^\theta_{j,L_2})=I_M^M(\gamma,(\mathbf{h}_j)_P)=J_M^M(\gamma,(\mathbf{h}_j)_P)$, 
we obtain the assertion. 
\end{proof}

\subsubsection{Proof of Theorem \ref{thm:asym}}
The following three cases are discussed.
\begin{itemize}
    \item[{\it (1)}] $E=F$ and $n$ is odd. 
    \item[{\it (2)}] $E=F$ and $n$ is even. 
    \item[{\it (3)}] $E\neq F$. 
\end{itemize}
These numbers correspond to {\it(1)}, {\it(2)}, {\it(3)} in Theorem \ref{thm:asym} respectively.
For Cases {\it(2)} and {\it(3)}, we set
\[
\tilde{\cO}\coloneqq \{ \fo \in \cO \mid \text{$\fo$ contains an element $\gamma$ such that $\gamma_s\theta(\gamma_s)=1_n$} \}.
\]
For Case {\it(1)}, we set
\[
\tilde{\cO}\coloneqq \{ \fo \in \cO \mid \text{$\fo$ contains an element $\gamma$ such that $(\gamma_s\theta(\gamma_s))^2=1_n$} \}.
\]
We will now begin to prove Theorem \ref{thm:asym} using some lemmas, equations, and propositions. Proving the lemmas, etc., will make up the remainder of Part \ref{part:1}. 
\begin{proposition}\label{prop:asym1}
    Suppose that $j$ is sufficiently large. Then, for any $\fo\in \cO\setminus \tilde{\cO}$, we have $I_\fo(h_j)=0$.  
\end{proposition}
\begin{proof}
    This follows from Lemmas \ref{lem:range1}, \ref{lem:g1}, and \ref{lem:g2}.
\end{proof}
Let $\fo\in\tilde{\cO}$. 
Recall the notations $H$ and $V^0$ given in \S\ref{sec:vecsp}. 
We set
\[
\fo_\main \coloneqq \begin{cases}
    \{ \gamma\in \fo \mid \gamma=\gamma_s \quad  \text{and} \quad \gamma\in H(V^0)J_n^{-1} \}  & \text{Case {\it(1)}}, \\
    \{ \gamma\in \fo \mid \gamma=\gamma_s  \quad  \text{or} \quad \gamma\in H(V^0)J_n^{-1} \} & \text{Case {\it(2)}}, \\
    \{ \gamma\in \fo \mid \gamma=\gamma_s  \} & \text{Case {\it(3)}}.
\end{cases} 
\]
Put $\fo_\negl\coloneqq \fo \setminus \fo_\main$ and
\[
\fo_{\negl,\rss}\coloneqq \{ \gamma\in\fo_\negl \mid \gamma=\gamma_s\} ,\qquad \fo_{\negl,\rns}\coloneqq \{ \gamma\in\fo_\negl \mid \gamma\neq \gamma_s\}.
\]
\begin{proposition}\label{prop:asym2ss}
    Suppose that $\fo\in\tilde{\cO}$ and $j$ is sufficiently large. Then, we have $I_{G,\gamma}(h_j)=0$ for every $\gamma\in \fo_{\negl,\rss}$.      
\end{proposition}
\begin{proof}
This is a consequence of the second assertion of \cref{lem:glodd} since we have $\fo_{\negl,\rss}\neq \emptyset$ only for Case {\it(1)} $E=F$ and $n$ is odd. 
\end{proof}
We set
\begin{equation*}\label{eq:negl}
I_{\fo_\negl}(h_j)\coloneqq \sum_{\gamma\in (G(F)\cap\fo_{\negl,\rns})_{G,S_j'}} I_{G,\gamma}(h_j) +\sum_{M\in\cL,M\neq G} \frac{|W^M_0|}{|W^G_0|} \sum_{\gamma\in (M(F)\cap\fo)_{M,S_j'}} I_{M,\gamma}(h_j),    
\end{equation*}
\[
I_{\fo_\main}(h_j)\coloneqq \sum_{\gamma\in (G(F)\cap\fo_\main)_{G,S_j'}} I_{G,\gamma}(h_j) .
\]
By \cref{prop:asym2ss} for $\fo\in\tilde\cO$ and large $j$ we have
\begin{equation}\label{eq:mainnegl}
I_\fo(h_j)=I_{\fo_\main}(h_j)+I_{\fo_\negl}(h_j) .
\end{equation}
\begin{proposition}\label{prop:asym2}
Take a small $\epsilon>0$. 
For each $\fo\in\tilde{\cO}$ we have $I_{\fo_\negl}(h_j)=O_\epsilon (N(\fn_j)^{\fm-1+\epsilon})$ as $j\to \inf$.  
\end{proposition}
\begin{proof}
Under Condition \ref{cond:asym} (A1), there exists a finite set $S'$ of places of $F$ such that $S_j'$ is included in $S'$ for any $j$. Then we can replace $S_j'$ by $S'$ in \eqref{eq:ggeom}, and then \eqref{eq:geomfo} is rewritten as \eqref{eq:A1-1}. 
Hence, $a^M(S',\gamma)$ does not vary for $j$, and $\gamma$ runs through only finitely many elements, which are independent of $j$.
Therefore, we obtain the assertion by Lemmas \ref{lem:a1b1} and \ref{lem:a1b2}.  

Under Condition \ref{cond:asym} (A2), we obtain the assertion by \eqref{eq:A2b1} and \eqref{eq:A2b2}, since for any $j\in\N$ we have $\sum_{\fo\in \tilde{\cO}}I_{\fo_\negl}(h_j)=I^{T_j}_{\negl,1}(h_j)+I^{T_j}_{\negl,2}(h_j)$ by \eqref{eq:A2negl}. 
\end{proof}

By Propositions \ref{prop:asym1}, \ref{prop:asym2ss}, and \ref{prop:asym2}, we have for sufficiently large $j$
\[
I_\geom(h_j)=\sum_{\fo\in \tilde{\cO}} I_{\fo_\main}(h_j) + O_\epsilon (N(\fn_j)^{\fm-1+\epsilon}).
\]
In addition, the term $\sum_{\fo\in \tilde{\cO}}I_{\fo_\main}(h_j)$ is explicitly determined in Propositions \ref{prop:mainterm1}, \ref{prop:mainterm2}, and \ref{prop:mainterm3}. 
Thus, we obtain Theorem \ref{thm:asym}.

\subsection{Intersections of \texorpdfstring{$\theta$}{}-conjugacy classes and small neighborhoods}\label{sec:intersections}

For each $v\in S\setminus \{v_0\}$, we denote by $C_v$ a compact subset in $M(F_v)$, which contains the support of $(h_v)_P$ in $M(F_v)$. There exists a compact subset $C_{v_0}$ in $M(F_{v_0})$ such that, if $I_M^G(\gamma,h_{v_0})\neq 0$, then we have $g^{-1}\gamma\theta(g)\in C_{v_0}$ for some $g\in M(F_{v_0})$, see \cite{Art88b}*{Lemma 3.2}. 
Set $C_S\coloneqq \prod_{v\in S} C_v$. 
Hence, we obtain
\begin{lemma}\label{lem:range1}
    There exists an element $g\in M(\A)$ such that $g^{-1}\gamma\theta(g)\in C_S\, \delta_n\, K^S(\fn_j)$ if $I_{M,\gamma}(h_j)\neq 0$.
\end{lemma}

\begin{lemma}\label{lem:g1}
Suppose Case (2) $E=F$ and $n$ is even or Case (3) $E\neq F$.
Fix a compact subset $C_S$ of $G(F_S)^1$. 
Let $\gamma\in G(F)$.
Suppose that $j$ is sufficiently large and we have $g^{-1}\gamma \theta(g)\in C_SK^S(\fn_j)$ for some $g\in G(\A)$. 
Then,  $\gamma_s\theta(\gamma_s)=1_n$. 
\end{lemma}

\begin{proof}
Recall that when $E\neq F$,  we fixed an embedding $\nu\,\colon \GL_n(\fo_E)\to\GL_{2n}(\fo)$.
Note that we can extend the involution $\theta$ to $\GL_{2n}(\fo)$ so that $\nu$ becomes equivariant.

Let $m=n$ when $E=F$ and $m=2n$ when $E\neq F$.
Define the embedding $i\,\colon \GL_m(\fo)\rtimes \langle \theta \rangle \to \GL_{2m}(\fo)$ by
    \[
    i(g\rtimes 1)= 
        \begin{pmatrix} 
        g &  \\ 
         & \theta(g) 
        \end{pmatrix} ,\qquad 
        i(1\rtimes\theta)= 
        \begin{pmatrix} 
         & 1_m \\ 
        1_m &  
        \end{pmatrix}
    =\colon \mathfrak{J}.
    \]
Then we have $i(g^{-1}\gamma \theta(g) \rtimes \theta)=i(g\rtimes 1)^{-1}\; ( i(\gamma\rtimes 1)\, \mathfrak{J}  ) \; i(g\rtimes 1)$ and $g^{-1}\gamma \theta(g)\in K^S(\fn_j)$ is equivalent to
    \[
    i(g\rtimes 1)^{-1}\; (i(\gamma\rtimes 1)\mathfrak{J}) \; i(g\rtimes 1)\in 
    \left\{ i(k\rtimes 1)\mathfrak{J}=
        \begin{pmatrix} 
        & k \\ 
        \theta(k) &  
        \end{pmatrix} \; \middle| \; k\in K^S(\fn_j) \right\}.
    \]
We write this set by $\tilde K^S(\fn_j)$.

For $x\in G^+(\A)$,  set $p(t,x)=\det(t1_{2m}+i(x))$ and let $a_u(x)$, $u=1,  \ldots,  2m$ be the $(2m-u)$-th coefficient of $p(t,  x)$:
    \[
    p(t,x)=t^{2m}+a_1(x)\, t^{2m-1}+a_2(x)\, t^{2m-2}
    +\cdots+a_{2m}(x).
    \]
Note that we have $p(t, 1\rtimes \theta)=(t^2-1)^m$.
Let $x\in G(\A)$.
Since $N(\fn_j)$ is sufficiently large,  we have 
    \[
     \prod_v | a_u(x\rtimes \theta)-a_u(1\rtimes \theta)|_v<1
    \]
for $i(x\rtimes \theta)\in i(C_S) \tilde K^S(\fn_j)$ and $1\le u\le 2m$.
On the other hand,  for $\gamma\in G(F)$ and $1\le u\le 2m$, $a_u(\gamma\rtimes \theta)-a_u(1\rtimes \theta)\neq 0$ implies
    \[
      \prod_v | a_u(\gamma\rtimes \theta)-a_u(1\rtimes \theta)|_v = 1.
      \]
Therefore,  when $N(\fn_j)$ is sufficiently large,  $g^{-1}\gamma \theta(g)\in K^S(\fn_j)$ implies
    \[
    a_u(\gamma\rtimes \theta)=a_u(1\rtimes \theta)
    \]
for $1\le u \le 2m$.
This means $p(t,\gamma\rtimes\theta)=p(t,1\rtimes\theta)$. 
Hence the eigenvalues of $i(\gamma\rtimes\theta)$ are only $\pm1$, and we obtain
   \[
   i(\gamma_s\, \theta(\gamma_s) \rtimes 1) 
   = i((\gamma_s\rtimes \theta)^2) 
   = i(\gamma_s\rtimes\theta)^2=1,
   \]
in other words,  $\gamma_s\theta(\gamma_s)=1_m$. 
\end{proof}

\begin{lemma}\label{lem:g2}
Suppose Case (1) $E=F$ and $n$ is odd.
Fix a compact subset $C_S$ of $G(F_S)^1$. 
Let $\gamma\in G(F)$.
Suppose that $j$ is sufficiently large and we have $g^{-1}\gamma \theta(g)\in C_S\, \delta_n \,K^S(\fn_j)$ for some $g\in G(\A)$. 
Then,  $(\gamma_s\theta(\gamma_s))^2=1_n$. 
\end{lemma}

\begin{proof}
Since the argument is similar to the previous cases, we sketch the proof.
Define the embedding $i\,\colon \GL_n(\fo)\rtimes \langle \theta \rangle \to \GL_{2n}(\fo)$ by
    \[
    i(g\rtimes 1)= 
        \begin{pmatrix} 
        g &  \\ 
         & \theta(g) 
        \end{pmatrix} ,\qquad 
        i(1\rtimes\theta)= 
        \begin{pmatrix} 
         & 1_n \\ 
        1_n &  
        \end{pmatrix}
    \eqqcolon \mathfrak{J}.
    \]
We set 
    \[
    \tilde K^S(\fn_j)=\left\{ i(k\rtimes 1)\mathfrak{J}=
        \begin{pmatrix} 
        & k \\ 
        \theta(k) &  
        \end{pmatrix} \; \middle| \; k\in K^S(\fn_j) \right\}.
    \]
Then $g^{-1}\gamma \theta(g)\in \delta_n K^S(\fn_j)$ is equivalent to
    \[
    i(g\rtimes 1)^{-1}\; (i(\gamma\rtimes 1)\mathfrak{J} ) \;  
    i(g\rtimes 1)\in  i(\delta_n\rtimes 1) \tilde K^S(\fn_j).
    \]
We set $p(t,  x)=\det(t1_{2n}+i(x))$ for $x\in G^+(\A)$ as in the previous cases.
Then $p(t,\delta_n\rtimes\theta)=(t^2+1)^{n-1}(t^2-1)$ and we obtain $p(t,\gamma\rtimes\theta)=p(t,\delta_n\rtimes\theta)$. 
Hence the only eigenvalues of $i(\gamma\rtimes\theta)$ are $\pm 1$ and $\pm \sqrt{-1}$. 
Therefore we have 
    \[
    i((\gamma_s\, \theta(\gamma_s))^2 \rtimes 1) 
    = (i((\gamma_s\rtimes \theta) \, (\gamma_s\rtimes \theta)))^2 
    = i(\gamma_s\rtimes\theta)^4=1
    \]
and $(\gamma_s\theta(\gamma_s))^2=1_n$.
\end{proof}

\begin{lemma}\label{lem:Ypq}
Suppose Case (1) $E=F$ and $n$ is odd. 
Then,
\[
\{ g \in  G \mid (g\theta(g))^2=1_n \}=\bigsqcup_{p+q=n, \; \text{$q$ is even}} Y_{p,q} ,
\]
where $Y_{p,q}$ consists of the elements $g\in G$ such that $g$ is $\theta$-conjugate to $\begin{pmatrix} B&  \\ & C \end{pmatrix}J_n^{-1}$ in $G$ for some $B\in\SM_p$, $C\in\AL_q$. 
Note that $Y_{1,n-1}$ equals $H(V^0)J_n^{-1}$. 
\end{lemma}
\begin{proof}
We suppose $(g\theta(g))^2=1$.  
Then the eigenvalues of $g\theta(g)$ are $\pm 1$. 
Hence, for some $y\in G$ and some $p$, $q\in\Z_{\ge 0}$, we have
\[
y^{-1}(g\theta(g))y=\begin{pmatrix} 1_p &  \\  & -1_q \end{pmatrix}, \quad \text{namely} \quad (y^{-1}g\theta(y))\theta(y^{-1}g\theta(y))=\begin{pmatrix} 1_p &  \\  & -1_q \end{pmatrix},
\]
Moreover
\[
y^{-1}g\theta(y) J_n =\begin{pmatrix} 1_p &  \\  & -1_q \end{pmatrix}\;  \t(y^{-1}g\theta(y)J_n) .
\]
Therefore, there exist $B\in\SM_p$ and $C\in\AL_q$ such that we have
\[
y^{-1}g\theta(y) J_n=\begin{pmatrix} B& \\ &C \end{pmatrix}.
\]
If $q$ is odd, then the right-hand side is not in $G$. 
Hence, $q$ must be even.
Note that $g'\coloneqq y^{-1}g\theta(y)$ is semisimple, that is, $g'=g_s'$. 
\end{proof}

\begin{lemma}\label{lem:rank}
Let $\fo\in\tilde{\cO}$ and $\gamma\in \fo$. 
    \begin{itemize}
        \item For Case (2) $E=F$ and $n$ is even or Case (3) $E\neq F$, we have $\gamma=\gamma_s$ if and only if  $\Pr_2(\gamma)=0$. 
        \item For Case (1) $E=F$ and $n$ is odd, we have $\gamma=\gamma_s\in H(V^0)J_n^{-1}$ if and only if $\mathrm{rank}(\Pr_2(\gamma))=1$. Note that $\Pr_2(\gamma)\neq 0$ holds always.
    \end{itemize}
\end{lemma}
\begin{proof}
    We can discuss the proof of this assertion over a closed field of $F$. Hence, we may choose $\gamma_s=1_n$ under Case {\it(2)} or Case {\it(3)}, and $\gamma_s=\begin{pmatrix}
    1_p &  \\  & J_q 
\end{pmatrix}J_n^{-1}$ for some $p$, $q$ under Case {\it(1)}. 
Since $\gamma_u$ is a unipotent element in $G_{\gamma_s}$, we can complete the proof by a direct calculation on the Lie algebras of $G_{\gamma_s}$ and using the exponential mapping. 
\end{proof}

\begin{lemma}\label{lem:glodd}
Suppose Case (1) $E=F$, $n=2m+1$, and $j$ is sufficiently large. Let $\fo\in \cO$, and take an element $\gamma\in\fo$. 
\begin{itemize}
    \item If $I_{G,\gamma}(h_j)\neq 0$, then we have $\mathrm{rank}(\Pr_1(\gamma))=n-1$. 
    \item Assume $\gamma=\gamma_s$ and $(\gamma_s\theta(\gamma_s))^2=1_n$ (that is, $\gamma_s\in Y_{p,q}$ for some $p$, $q$, see Lemma \ref{lem:Ypq}). Then, we have $I_{G,\gamma}(h_j)=0$ unless $\gamma$ is in $Y_{1,2m}=H(V^0)J_n^{-1}$.  
\end{itemize}
\end{lemma}
\begin{proof}
Suppose $I_{G,\gamma}(h_j)\neq 0$.  
By \cref{lem:range1} there exists an element $x=(x_v)_v$ in $G(\A)$ such that $x^{-1}\gamma\theta(x)$ is in $C_S \delta_n K^S(\fn_j)$. 
Recall the notation $r_{v,j}$ defined in \eqref{eq:level}. 
Since $j$ is sufficiently large, we can take a finite place $v\notin S$ such that $\fp_v\mid \fn_j$ and $q_v^{r_{v,j}}$ is large, that is, $r_{v,j}$ or $q_v$ is large. 
For some $A_v\in \M_n(\fo_{F,v})$, we have 
\[
x_v^{-1}\gamma \theta(x_v)=\delta_n + \varpi_v^{r_{v,j}} A_v .
\]
Write $E_{i,j}$ for the unit matrix of the $(i,j)$-component in $\M_n$, and let $e_j$ denote the $j$-th basic vector in $\M_{1\times n}$.
Then, we have $\delta_nJ_n= H(\delta_nJ_n-E_{m+1,m+1},e_{m+1})$, which is a point of $H(V^0(F_v))$. 
Since $q_v^{r_{v,j}}$ is large, 
it follows from a direct calculation on the actions of $K_v$ on $V(\fo_{F,v})$ that there exists an element $k_v\in K_v$ so that
\[
(x_vk_v)^{-1}\gamma \theta(x_vk_v) = \delta_n  +\varpi_v^{r_{v,j}} \begin{pmatrix} A_{11}'& 0 &A_{12}' \\ 0&0&0 \\ A'_{21} &0 & A'_{22}   \end{pmatrix}J_n^{-1} \quad \text{for some $\begin{pmatrix} A_{11}'& A_{12}' \\ A'_{21} &A'_{22} \end{pmatrix}\in \SM_{2m}(\fo_{F,v})$}.
\]
Hence, we obtain
\[
\Pr_1((x_vk_v)^{-1}\gamma \theta(x_vk_v))=\delta_nJ_n-E_{m+1,m+1}.
\]
This means $\mathrm{rank}(\Pr_1(\gamma))=n-1$. 
In addition, if $\gamma=\gamma_s$, then $\gamma$ is in $Y_{1,2m}$ and $(\gamma\theta(\gamma))^2=1_n$. 
\end{proof}

\begin{lemma}\label{lem:evenfo}
Suppose Case (2) $E=F$ and $n$ is even.
Then, $\tilde\cO$ consists of a single class $\fo_1$.
The class $\fo_1$ contains the unit element $1_n$, which is a representative element of the semisimple conjugacy class in $\fo_1$. 
Take an element $\gamma\in\fo_1$. 
When $\gamma_s=1_n$, the conjugacy class of $\gamma_u$ is regarded as a unipotent conjugacy class via the isomorphism $G_{\gamma_s}\simeq \Sp_n$. 
By this identification, we have
\[
(\fo_1)_\main=\fo_{\main,1}\sqcup\fo_{\main,2} ,
\]
where
\begin{align*}
\fo_{\main,1}\coloneqq & \{ \gamma\in\fo_1 \mid \gamma_s\theta(\gamma_s)=1_n, \quad \gamma_u=1_n \}, \\
\fo_{\main,2}\coloneqq & \{ \gamma\in\fo_1 \mid \gamma_s\theta(\gamma_s)=1_n, \quad \text{$\gamma_u$ is a minimal unipotent element in $G_{\gamma_s}$} \}.
\end{align*}
As a consequence, for $\gamma\in\fo_1$, we have $\mathrm{rank}(\Pr_2(\gamma))\ge 2$ if and only if $\gamma\in(\fo_1)_{\negl}$. 
\end{lemma}
\begin{proof}
If $\gamma_s\theta(\gamma_s)=1_n$, then $\gamma_sJ_n$ belongs to $\AL_n^0(F)$. Hence, $\cO$ consists of the single class of $1_n$, since $\AL_n^0(F)=J_n\cdot G(F)$. 
For the element $(J_n,e_1)\in V^0$, we have $H(J_n,e_1)J_n^{-1}=1_n+E_{1,n}\in \Sp_n$, which is a minimal unipotent element in $\Sp_n$. 
Since $V^0=(J_n,e_1)\cdot \GL_n$, the $\theta$-conjugacy class of $H(J_n,e_1)J_n^{-1}$ is equal to $H(V^0)J_n^{-1}$.
This completes the proof. 
\end{proof}

\subsection{Some lemmas}\label{sec:Some}

Let $v$ be a finite place of $F$. 
Define a subset $U_v$ of $K_v$ by
\[
U_v\coloneqq \{ k_v\, \delta_n \, \theta(k_v^{-1}) \mid k_v\in K_v\}.     
\]
Let $r\in\Z_{>0}$ and put
\[
\fM_v(r)\coloneqq \frac{\#\{k_v\in K_v \mod \varpi_v^r \fo_{F,v} \mid  k_v\delta_n \theta(k_v^{-1}) \equiv  \delta_n \mod \varpi_v^r\fo_{F,v} \} }{\#(K_v/K_v(\fp_v^r))}.
\]
We also set
\[
\fm'\coloneqq\begin{cases}
    \fm & \text{Case (1) $E=F$ and $n=2m+1$ or Case (3) $E\neq F$}, \\
    \fm+m & \text{Case (2) $E=F$ and $n=2m$},
\end{cases}
\]
see \eqref{eq:fmdef} for the notation $\fm$. 
\begin{lemma}\label{lem:fmbound}
We have 
\[
\fM_v(r) \le  q_v^{r\fm'} \times \fc_v , 
\]
where
\[
\fc_v\coloneqq\begin{cases}
    \prod_{t=1}^{m+1}(1-q_v^{-2t+1})^{-1} & \text{Case {\it(1)} $E=F$ and $n=2m+1$}, \\
    \prod_{t=1}^{m}(1-q_v^{-2t+1})^{-1} & \text{Case {\it(2)} $E=F$ and $n=2m$}, \\
    \prod_{t=1}^{n}(1-\eta_v(\varpi_v)^{t-1}q_v^{-t})^{-1} & \text{Case {\it(3)} $E\neq F$}. 
\end{cases}
\]
\end{lemma}
\begin{proof}
Since $\fM_v(1)=\# G_{\delta_n}(\F_{q_v})/\# G(\F_{q_v})$ $(\F_{q_v}\coloneqq\fo_{F,v}/\varpi_v\fo_{F,v})$, we obtain $\fM_v(1) = q_v^{\fm'}\times \fc_v$ by a direct calculation.
Hence, it is sufficient to prove $\fM_v(r+1)\le q_v^{\fm'} \times \fM_v(r)$.

Define a group $K_v(\fp_v^r,\delta_n)$ by
\[
K_v(\fp_v^r,\delta_n)\coloneqq \{ k_v\in K_v \mod \varpi_v^r\fo_{F,v} \mid k_v\delta_n\theta(k_v^{-1})\equiv \delta_n \mod \varpi_v^r\fo_{F,v}\}, 
\]
and consider the homomorphism
\[
f_r\colon K_v(\fp_v^{r+1},\delta_n)\to K_v(\fp_v^r,\delta_n) , \quad f_r(k_v\mod \varpi_v^{r+1}\fo_{F,v})=k_v\mod \varpi_v^r\fo_{F,v}.
\]
We put $\M'_n\coloneqq \M_n$ when $E=F$, and $\M'_n\coloneqq \Res_{E/F}\M_n$ when $E\neq F$. 
It is easy to obtain the bijection between $\mathrm{Ker}(f_r)$ and $\{ X\in \fM_n'(\F_{q_v}) \mid X\delta_nJ_n + \delta_n J_n \iota(\t X)=O  \}$ by $1_n+\varpi_v^r X \mapsto X$. 
Hence, by a direct calculation, we have
\[
\mathrm{Ker}(f_r) /  \#(K_v(\fp_v^{r})/K_v(\fp_v^{r+1})) \le q^{\fm'}.
\]
Since $\# \mathrm{Im}(f_r)\le \# K_v(\fp_v^{r},\delta_n)$, we find $\# K_v(\fp_v^{r+1},\delta_n)\le \# \mathrm{Ker}(f_r) \times \# K_v(\fp_v^r,\delta_n)$. 
Thus, we obatin $\fM_v(r+1)\le q_v^{\fm'} \times \fM_v(r)$ by $\fM_v(r)=\# K_v(\fp_v^r,\delta_n) / \#(K_v/K_v(\fp_v^r))$. 
\end{proof}

\begin{lemma}\label{lem:red}
A function $(\cf_{\delta_n K_v(\fp_v^r)})_{K_v}\in \cH(G(F_v))$ is defined by 
\[
(\cf_{\delta_n K_v(\fp_v^r)})_{K_v}(g_v)\coloneqq \int_{K_v} \cf_{\delta_n K_v(\fp_v^r)}(k_v^{-1} g_v \theta(k_v)) \, \rd k_v ,\qquad g_v\in G(F_v).
\]
Then, for any $r\in\Z_{\ge 0}$, we have
\[
(\cf_{\delta_n K_v(\fp_v^r)})_{K_v}= \fM_v(r) \times \cf_{K_v(\fp_v^r)U_v}.
\]
Note that we have $U_vK_v(\fp_v^r)=K_v(\fp_v^r)U_v$. 
\end{lemma}
\begin{proof}
Recall that $K_v(\fp_v^r)$ is a normal subgroup in $K_v$. 
Hence, for an arbitrary $g_v$ in $G(F_v)$, if we have $k_v^{-1}g_v \theta(k_v)\in \delta_n K_v(\fp_v^r)$ for some $k_v\in K_v$, then $g_v\, \theta(k_v)\delta_n k_v^{-1}\in K_v(\fp_v^r)$. 
This means that $(\cf_{\delta_n K_v(\fp_v^r)})_{K_v}(g_v)\neq 0$ if and only if $g_v\in K_v(\fp_v^r)U_v$. 
Furthermore, for each $g_v\in K_v(\fp_v^r)U_v$, we find 
\[
\int_{K_v} \cf_{\delta_n K_v(\fp_v^r)}( k_v^{-1}g_v\theta(k_v))   \rd k_v = \int_{K_v} \cf_{\delta_n K_v(\fp_v^r)}( k_v^{-1}\delta_n \theta(k_v))   \rd k_v  =\int_{k_v\delta_n\theta(k_v^{-1})\in \delta_n K_v(\fp_v^r)} \rd k_v =\mathfrak{M}_v(r)
\]
by a change of variable for $k_v$. 
\end{proof}

\begin{lemma}\label{lem:pr}
The image $\Pr_2(K_v(\fp_v^r)U_v)$ is included in 
\[
 \begin{cases}
\SM_n^1(\fo_{F,v})+\SM_n(\varpi_v^r\fo_{F,v}) & \text{Case {\it(1)} $E=F$ and $n$ is odd}, \\
\SM_n(\varpi_v^r\fo_{F,v}) & \text{Case {\it(2)} $E=F$ and $n$ is even}, \\
\HE'_{n}(\varpi_v^r\fo_{F,v}) & \text{Case {\it(3)} $E\neq F$}.
\end{cases}
\]
Here we set $\SM_n^1(\fo_{F,v})\coloneqq \{ x\in \SM_n(\fo_{F,v}) \mid \mathrm{rank}(x)\le 1\}$, and we note that $\SM_n^1(\fo_{F,v})$ is not a $\fo_{F,v}$-module, but the above sum 
\[
\SM_n^1(\fo_{F,v})+\SM_n(\varpi_v^r\fo_{F,v})\coloneqq \{ x+y \mid x\in \SM_n^1(\fo_{F,v}),\quad y\in \SM_n(\varpi_v^r\fo_{F,v})\}
\]
is an open compact subset in $\SM_n(F_v)$.
\end{lemma}
\begin{proof}
This is a consequence of the property that $\Pr_2$ is linear and $G$-equivariant. 
\end{proof}

\begin{lemma}\label{lem:mainexpress}
For every case, we have
\[
\sum_{\fo\in \tilde{\cO}}I_{\fo_\main}(h_j)=\sum_{\fo\in\tilde{\cO}} \int_{G(F)\bs G(\A)^1} \sum_{\gamma\in \fo_\main} h_j(g^{-1}\gamma \theta(g)) \, \rd g. 
\]    
When Case {\it(1)} $E= F$ and $n$ is odd, we also have
\[
\sum_{\fo\in \tilde{\cO}}\sum_{\gamma\in\fo_{\negl,\rss}}I_{G,\gamma}(h_j)=\sum_{\fo\in \tilde{\cO}} \int_{G(F)\bs G(\A)^1} \sum_{\gamma\in \fo_{\negl,\rss}} h_j(g^{-1}\gamma \theta(g)) \, \rd g. 
\]    
\end{lemma}
\begin{proof}
When $\gamma=\gamma_s$, the global coefficient $a^G(S_j',\gamma)$ is known in general, see \cite{Art86}*{Theorem 8.2}, that is, for a $F$-elliptic semisimple element $\gamma \in G(F)$,
\[
a^G(S_j',\gamma)=[G_\gamma^+(F):G_\gamma(F)]\, \int_{G_\gamma(F)\bs G_\gamma(\A)^1} \rd g_\gamma
\]
where $G_\gamma^+$ is the $\theta$-centralizer $\{ g\in G \mid g^{-1}\gamma\theta(g)=\gamma\}$, $G_\gamma(\A)^1=G_\gamma(\A)\cap G(\A)^1$, and $\rd g_\gamma$ is a Haar measure on $G_\gamma(\A)$. 
The local orbital integral $I_G^G(\gamma,f)$ over $F_{S_j'}$ can be extended into an ad\`elic orbital integral, see \cite{Art86}*{Remarks. 1. after Theorem 8.2}. 
We also have $|D^G(\gamma)|_v=1$ for almost all places $v$ and $\prod_v|D^G(\gamma)|_v=1$ by $D^G(\gamma)\in F^\times$. 
Hence, for a $F$-elliptic semisimple element $\gamma \in G(F)$, 
\[
\int_{G(F)\bs G(\A)^1} \sum_{\omega} h_j(g^{-1}\omega \theta(g)) \, \rd g = a^G(S_j',\gamma)\times J_G^G(\gamma,h_j) 
\]
where $\omega$ runs over all elements of the $\theta$-conjugacy class $\{ \delta^{-1}\gamma \theta(\delta) \mid \delta\in G(F)\}$.  
Therefore, we obtain the assertion for Case {\it(1)} $\fo_\main$, $\fo_{\negl,\rss}$ and Case {\it(3)} $\fo_\main$.

Suppose Case {\it(2)} $E=F$ and $n$ is even. 
By \cref{lem:evenfo}, $\tilde{\cO}$ consists of a single class $\fo_1$, and $(\fo_1)_\main$ is the union of $\fo_{\main,1}$ and $\fo_{\main,2}$. 
The subset $\fo_{\main,1}$ is the single $\theta$-conjugacy class of $1_n$, and $I_{G,1_n}(h_j)$ equals $\int_{G(F)\bs G(\A)^1} \sum_{\gamma\in \fo_{\main,1}} h_j(g^{-1}\gamma \theta(g)) \, \rd g$. 
For each $\gamma\in (G(F)\cap \fo_{\main,2})_{G,S_j'}$, the coefficient $a^G(S_j',\gamma)$ is the same as the global coefficient of a minimal unipotent element of $\Sp_n$ (cf. \cite{Art86}*{Lemma 7.1 and Theorem 8.1}). 
Hence, $a^G(S_j',\gamma)$ and $\int_{G(F)\bs G(\A)^1} \sum_{\gamma\in \fo_{\main,2}} h_j(g^{-1}\gamma \theta(g)) \, \rd g$ can be calculated by the same argument as in \cite{HW18}*{Theorem 7.1}. 
Thus, one can prove the assertion. 
\end{proof}

\subsection{Under Condition \ref{cond:asym} (A1)}\label{sec:UnderA1}

Throughout this subsection, we assume Condition \ref{cond:asym} (A1). 
Under this assumption, there exists a finite set $S'$ of places of $F$ such that $S_j'$ is included in $S'$ for any $j$. 
Then, the definition \eqref{eq:geomfo} of $I_{M,\gamma}(h_j)$ can be rewritten as 
\begin{equation}\label{eq:A1-1}
I_{M,\gamma}(h_j) = a^M(S',\gamma)\times I_M^G(\gamma,h_{v_0}) \times  J^M_M(\gamma,(h_{S_0})_P) \times \prod_{v\in S'\setminus S} J^M_M(\gamma,(\cf_{\delta_nK_v(\fn_j)})_P) .    
\end{equation}
Hence, it is sufficient to consider finitely many $(M,S')$-conjugacy classes in the geometric side \eqref{eq:ggeom}. 
Take a $\cO$-equivalence class $\fo\in \tilde{\cO}$, let $\gamma\in\fo$ be a representative element in a $(M,S')$-conjugacy class, and suppose $I_{M,\gamma}(h_j)\neq 0$. 
For each $(M,S')$-equivalence class in $M(F)\cap \fo$, we choose a representative element $\gamma$ and fix a normalization of a Haar measure on $M_\gamma(F_v)$ if $I_{M,\gamma}(h_j)\neq 0$. 
This normalization is necessary to treat the factors of $I_{M,\gamma}(h_j)$, because individual factors of $I_{M,\gamma}(h_j)$ depend on measures. 
Set
\[
S'' \coloneqq \left\{ v\in S'\setminus S \; \middle| \; \lim_{j\to\infty} r_{v,j}=\infty \right\} \qquad \text{(see \eqref{eq:level} for the notation $r_{v,j}$)}. 
\]
To get upper bounds of $I_{M,\gamma}(h_j)$, it is enough to estimate upper bounds of $J^M_M(\gamma,(\cf_{\delta_nK_v(\fn_j)})_P)$ for all $v\in S''$.

\begin{lemma}\label{lem:b0}
Suppose that $P=MN$ is a proper standard $\theta$-stable parabolic subgroup over $F$. 
Choose a finite place $v\notin S$ and assume $\fp_v^r\mid\fn_j$. 
If $r$ is sufficiently large, then
\[
J_M^M(\gamma,(\cf_{\delta_n K_v(\fp_v^r)})_P)\ll_\gamma \fc_v\, q_v^{r\, \fm' } \times \begin{cases} q_v^{r\, (-2[n/2]+1) } & \text{if $E=F$}, \\ q_v^{r\, (-2n+3) } & \text{if $E\neq F$,} \end{cases}
\]
where $[ \; ]$ is the Gaussian symbol. 
\end{lemma}
\begin{proof}
Applying Lemmas \ref{lem:fmbound} and \ref{lem:red} to $J_M^M(\gamma,(\cf_{\delta_n K_v(\fp_v^r)})_P)$ we have
\begin{equation}\label{eq:628}
J_M^M(\gamma,(\cf_{\delta_n K_v(\fp_v^r)})_P)\ll_\gamma \fc_v\, q_v^{r\, \fm' } \int_{M_\gamma(F_v)\bs M(F_v)}\int_{N(F_v) } \cf_{K_v(\fp_v^r)U_v}( m_v^{-1}\gamma \theta(m_v) n_v )  \, \rd n_v \, \rd m_v.
\end{equation}
If $\cf_{K_v(\fp_v^r)U_v}( m_v^{-1}\gamma \theta(m_v) n_v ) \neq 0$, then we have $m_v'\coloneqq m_v^{-1}\gamma \theta(m_v)\in   K_v(\fp_v^r)U_v\cap M(F_v)$, and so $|\det(m_v')|_v=1$. 
Hence, by the change of variable for $\rd n_v$, we derive
\begin{multline*}
\int_{M_\gamma(F_v)\bs M(F_v)}\int_{N(F_v)} \cf_{K_v(\fp_v^r)U_v}( m_v' n_v ) \, \rd n_v \, \rd m_v \\
= \int_{M_\gamma(F_v)\bs M(F_v)} \cf_{K_v(\fp_v^r)U_v}(m_v') \, \rd m_v \times \int_{N(F_v) } \cf_{K_v(\fp_v^r)U_v}(n_v )  \, \rd n_v. 
\end{multline*}
An upper bound of $\int_{M_\gamma(F_v)\bs M(F_v)} \cf_{K_v(\fp_v^r)U_v}(m')\, \rd m_v$ depends only on $\gamma$. 
Using Lemma \ref{lem:pr}, we can calculate a desired upper bound of $\int_{N(F_v) } \cf_{K_v(\fp_v^r)U_v}(n )  \, \rd n_v$. 
Thus, this completes the proof.
\end{proof}

\begin{lemma}\label{lem:a1b1}
Assume that \cref{cond:asym} {\rm(A1)} holds and $M$ is a proper Levi subgroup over $F$.
Then, we have
\[
I_{M,\gamma}(h_j)\ll_\gamma \Nm(\fn_j)^{\fm } \times \begin{cases} \Nm(\fn_j)^{-2m+1} & \text{Case {\it(1)} $E=F$ and $n=2m+1$}, \\ \Nm(\fn_j)^{-m+1} & \text{Case {\it(2)} $E=F$ and $n=2m$}, \\ \Nm(\fn_j)^{-2n+3} & \text{Case {\it(3)} $E\neq F$.} \end{cases}
\]
\end{lemma}
\begin{proof}
Since $h_S=h_{v_0}\otimes h_{S_0}$ is fixed, the integrals $I_M^G(\gamma,h_{v_0})$ and $J^M_M(\gamma,(h_{S_0})_P)$ depend only on $\gamma$. 
The coefficient $a^M(S',\gamma)$ also depends only on $\gamma$, because $S'$ is fixed.  
Since $r_{v,j}$ is bounded for every $v\in S'\setminus(S\sqcup S'')$, the integral $\prod_{v\in S'\setminus(S\sqcup S'')} J^M_M(\gamma,(\cf_{\delta_nK_v(\fn_j)})_P)$ has an upper bound, which is independent of $j$. 
By \cref{lem:b0},
\[
\prod_{v\in S''} J^M_M(\gamma,(\cf_{\delta_nK_v(\fn_j)})_P)  \ll_\gamma \prod_{v\in S''} \Nm(\fp_v^{r_{v,j}})^{\fm' } \times \begin{cases} \prod_{v\in S''} \Nm(\fp_v^{r_{v,j}})^{-2[n/2]+1} & \text{if $E=F$}, \\ \prod_{v\in S''} \Nm(\fp_v^{r_{v,j}})^{-2n+3} & \text{if $E\neq F$.} \end{cases}
\]
Note that the constant $\fc_v$ (defined in \cref{lem:fmbound}) is independent of $j$. 
Thus, the proof is completed by $\Nm(\fn_j)\asymp \prod_{v\in S''} \Nm(\fp_v^{r_{v,j}})$ and \eqref{eq:A1-1}. 
\end{proof}

Fix a large natural number $r_0$, which is independent of $j$. 
In the level $\fp_v^{r_0}$, we can consider a matrix calculation similar to the proof of Lemma \ref{lem:glodd}. 
\begin{lemma}\label{lem:b1}
Consider Case (i) $E=F$ and $n$ is odd, and $M=G$. 
Choose a finite place $v\notin S$. 
Suppose that $\gamma\neq \gamma_s$.
If $r$ is sufficiently large rather than $r_0$, then we have
\[
J_G^G(\gamma,\cf_{\delta_n K_v(\fp_v^r)})\ll_{\gamma} \fc_v\,  q_v^{r\, \fm' } \times q_v^{-r} .
\]
\end{lemma}
\begin{proof}
Let $n=2m+1$. 
Take an element $x_v\in G(F_v)$ satisfies $x_v^{-1}\gamma \theta(x_v)\in \delta_n K_v(\fp_v^{r_0})$. 
Note that $x_v$ depends on $\gamma$ and $r_0$, but not $r$. 
Since $q_v^{r_0}$ is large, by the proof of Lemma \ref{lem:glodd} we find that there exists an element $k_v'\in K_v$ such that
\begin{equation}\label{eq:E=Foddlocal1}
\gamma'\coloneqq k_v^{\prime -1}x_v^{-1}\gamma \theta(x_vk_v'), \qquad \Pr_1(\gamma')=\begin{pmatrix} B_{11}'& 0 &B_{12}' \\ 0&0&0 \\ B'_{21} &0 & B'_{22}   \end{pmatrix} \quad \left( \begin{pmatrix} B_{11}'& B_{12}' \\ B'_{21} &B'_{22} \end{pmatrix}\in J_{2m}\cdot \GL_{2m}(\fo_{F,v}) \right),     
\end{equation}
\begin{equation}\label{eq:E=Foddlocal2}
    \Pr_2(\gamma')=E_{m+1,m+1} +\varpi_v^{r_0} \diag(a_1,a_2,\dots,a_m,0, a_{m+1},a_{m+2},\dots,a_{2m}),
\end{equation}
where $ a_1,\dots,a_{2m} \in\fo_{F,v}$ are taken from representative elements of $\fo_{F,v}/(\fo_{F,v})^2$.

Write $\rd g_{\gamma,v}$ for the Haar measure on $G_\gamma(F_v)$, which is fixed for each $\gamma$.
Since the quotient measure $\rd g_{\gamma,v}\bs \rd g_v$ is right $G(F_v)$-invarint, for any $f\in\cH(G(F_v))$,
\begin{align}\label{eq:K_vintegration}
J_G^G(\gamma,f)
&= J_G^G(\gamma,f) \,\int_{K_v} \rd k_v=\int_{K_v} J_G^G(\gamma,f) \, \rd k_v  \\
&=  |D^G(\gamma)|_v^{\frac12} \int_{K_v}\int_{G_\gamma(F_v)\bs G(F_v)} f( g_v^{-1}\gamma \theta(g_v)  )   \,  \frac{\rd g_v}{\rd g_{\gamma,v}} \, \rd k_v \nonumber \\
&= |D^G(\gamma)|_v^{\frac12} \int_{K_v}\int_{G_\gamma(F_v)\bs G(F_v)} f( k_v^{-1}g_v^{-1}\gamma \theta(g_vk_v)  )   \,  \frac{\rd g_v}{\rd g_{\gamma,v}} \, \rd k_v  \nonumber\\
&= |D^G(\gamma)|_v^{\frac12} \int_{G_\gamma(F_v)\bs G(F_v)} \int_{K_v} f( k_v^{-1}g_v^{-1}\gamma \theta(g_vk_v)  )  \, \rd k_v \,  \frac{\rd g_v}{\rd g_{\gamma,v}} .  \nonumber
\end{align}
Hence, by this equality and \cref{lem:red}, 
\begin{align*}
J_G^G(\gamma,  \cf_{\delta_n K_v(\fp_v^r)}) \, &\ll_\gamma \int_{G_\gamma(F_v)\bs G(F_v)} (\cf_{\delta_n K_v(\fp_v^r)})_{K_v}( g_v^{-1}\gamma \theta(g_v)  )  \,  \frac{\rd g_v}{\rd g_{\gamma,v}} \\
& = \fM_v(r)  \int_{G_\gamma(F_v)\bs G(F_v)} \cf_{K_v(\fp_v^r)U_v}( g_v^{-1}\gamma \theta(g_v)  )  \,  \frac{\rd g_v}{\rd g_{\gamma,v}} .   
\end{align*}
In addition, using \cref{lem:fmbound}, we have
\begin{equation}\label{eq:upperboundJGG}
J_G^G(\gamma,  \cf_{\delta_n K_v(\fp_v^r)})  \ll_\gamma \fc_v\,  q_v^{r\, \fm' } \int_{G_\gamma(F_v)\bs G(F_v)} \cf_{K_v(\fp_v^r)U_v}( g_v^{-1}\gamma \theta(g_v)  )  \,  \frac{\rd g_v}{\rd g_{\gamma,v}} .    
\end{equation}
We take a Haar measure $\rd g_{\gamma',v}$ on $G_{\gamma'}(F_v)$ which is induced from $\rd g_{\gamma,v}$ via the $\theta$-conjugation of $x_vk_v'$. 
Then, by \eqref{eq:upperboundJGG},
\[
J_G^G(\gamma,  \cf_{\delta_n K_v(\fp_v^r)})  \ll_\gamma \fc_v\,  q_v^{r\, \fm' } \int_{G_{\gamma'}(F_v)\bs G(F_v)} \cf_{K_v(\fp_v^r)U_v}(x_v k_v'\, g_v^{-1}\gamma' \theta(g_v) \theta(x_v k_v')^{-1} )  \,  \frac{\rd g_v}{\rd g_{\gamma',v}} .   
\]
Since $k_v'$ belongs to $K_v$ and $x_v$ depends only on $\gamma$, there exist a non-negative integer $l$ and compactly supported functions $\Phi_1$ on $\AL_n(F_v)$ and $\Phi_2$ on $\SM_n(F_v)$ such that the support of $\Phi_2$ is contained in $\SM_n^1(\varpi_v^{-l}\fo_{F,v})+\SM_n(\varpi^{r-l}\fo_{F,v})$ and we have 
\[
J_G^G(\gamma,  \cf_{\delta_n K_v(\fp_v^r)}) \ll_\gamma \fc_v\,   q_v^{r\, \fm' } \int_{G_{\gamma'}(F_v)\bs G(F_v)} \Phi_1( \Pr_1(g_v^{-1}\gamma' \theta(g_v))) \, \Phi_2(\Pr_2(g_v^{-1}\gamma' \theta(g_v)))  \,  \frac{\rd g_v}{\rd g_{\gamma',v}} .   
\]

Put $\delta_n'\coloneqq  \Pr_2(\gamma')J_n^{-1}$, and set $G_{\delta_n'}\coloneqq \{g\in G \mid g^{-1}\delta_n'\theta(g)=\delta_n' \}$. 
Take a Haar measure $\rd g_{\delta_n',v}$ on $G_{\delta_n'}(F_v)$, which is independent of $\gamma$. 
By \eqref{eq:E=Foddlocal1},
\begin{align*}
    J_G^G(\gamma,  \cf_{\delta_n K_v(\fp_v^r)}) \, & \ll_\gamma \fc_v\,  q_v^{r\, \fm' } \int_{G_{\delta_n'}(F_v)\bs G(F_v)} \frac{\rd g_v}{\rd g_{\delta_n',v}} \int_{G_{\gamma'}(F_v)\bs G_{\delta_n'}(F_v)} \frac{\rd g_{\delta_n',v}}{\rd g_{\gamma',v}}  \\
    & \qquad \qquad \Phi_1( \Pr_1(g_v^{-1} \delta_n' \theta(g_v))) \, \Phi_2(\Pr_2(g_v^{-1} g_{\delta_n'}^{-1}\gamma' \theta(g_{\delta_n'}g_v)))  .
\end{align*}
The domain of the above integration for $\frac{\rd g_v}{\rd g_{\delta_n',v}}$ is compact, because the $G(F_v)$-orbit of $\Pr_1(\delta_n')$ is open in $\AL_n(F_v)$, see \cref{lem:glodd}.  
Hence, there exist a non-negative integer $l'$ and a compactly supported function $\Phi_3$ on $\SM_n(F_v)$ such that the support of $\Phi_3$ is contained in $\SM_n^1(\varpi_v^{-l'}\fo_{F,v})+\SM_n(\varpi^{r-l'}\fo_{F,v})$ and we have 
\[
    J_G^G(\gamma,  \cf_{\delta_n K_v(\fp_v^r)}) \ll_\gamma  \fc_v\,  q_v^{r\, \fm' }  \int_{G_{\gamma'}(F_v)\bs G_{\delta_n'}(F_v)}  \Phi_3(g_{\delta_n'}^{-1}\Pr_2(\gamma')\t g_{\delta_n'}^{-1}) \, \frac{\rd g_{\delta_n',v}}{\rd g_{\gamma',v}}  .
\]

If $\Pr_2(\gamma')\in\SM_n^1(\fo_{F,v})$, then by \cref{lem:rank} we see $\gamma=\gamma_s$, which contradicts to the assumption.  
Hence, we see $\Pr_2(\gamma')\in (\SM_n^1(\fo_{F,v})+\SM_n(\varpi^{r_0}\fo_{F,v})) \setminus \SM_n^1(\fo_{F,v})$, and $t=\mathrm{rank}(\Pr_2(\gamma'))$ is greater than $1$. 
This means that we have $a_1\cdots a_{t-1}\neq 0$ in \eqref{eq:E=Foddlocal2}.
Thus, it is possible to compute the desired upper bound of the right hand side of the above inequality by using $G_{\delta_n'}\simeq (\Sp_{2m}\times \mathbb{G}_m)\rtimes \mathbb{G}_a^{2m}$ and taking a suitable coordinate on $G_{\gamma'}(F_v)\bs G_{\delta_n'}(F_v)$.
\end{proof}

\begin{lemma}\label{lem:b2}
Consider Case {\it(2)} $E=F$ and $n$ is even, and $M=G$. 
By \cref{lem:evenfo}, $\tilde\cO$ consists of a single class $\fo_1$, and for each $\gamma\in\fo_1$, the conjugacy class of $\gamma_u$ corresponds to a unipotent conjugacy class in $\Sp_n$. 
Choose a finite place $v\notin S$. 
We suppose $\gamma\neq \gamma_s$ and $\gamma_u$ is not minimal in $\Sp_n$.
If $r$ is sufficiently large rather than $r_0$, then we have
\[
J_G^G(\gamma,\cf_{K_v(\fp_v^r)})\ll_{\gamma} \fc_v\, q_v^{r\, \fm' } \times q_v^{-r \, (2m-1)} .
\]
\end{lemma}
\begin{proof}
Take an element $x_v\in G(F_v)$ satisfies $x_v^{-1}\gamma \theta(x_v)\in K_v(\fp_v^{r_0})$. 
By an argument similar to the proof of Lemma \ref{lem:glodd}, we can take an element $k_v\in K_v$ so that
\[
\gamma'\coloneqq k_v^{-1}x_v^{-1}\gamma \theta(x_vk_v), \quad \Pr_1(\gamma')=J_n, \quad  \Pr_2(\gamma')\in \varpi_v^{r_0} \SM_n(\fo_{F,v}).
\]
Hence, in this case, we see $\gamma'_s=1_n$, i.e., $\gamma'=\gamma_u'$.
By $\gamma'\in\Sp_n$ we have
\[
\Pr_2(\gamma')=\frac{1}{2}(\gamma'J_n-J_n\t \gamma')=\frac{1}{2}(\gamma'-\gamma^{\prime -1})J_n.
\]
Hence, by taking a representative element of each unipotent orbit in $\Sp_n$, we find that $\Pr_2(\gamma')J_n^{-1}$ is a nilpotent element in $\mathfrak{sp}_n\coloneqq\{X\in\M_n \mid XJ_n+J_n\t X=O \}$. 
In addition, the assumptions implies that $\Pr_2(\gamma')J_n^{-1}$ is neither $O$ nor minimal. 
By using the same argument as in the proof of \cref{lem:b1}, there exist a non-negative integer $l'$ and a compactly supported function $\Phi$ on $\SM_n(F_v)$ such that the support of $\Phi$ is contained in $\SM_n(\varpi^{r-l'}\fo_{F,v})$ and we have 
\[
    J_G^G(\gamma,  \cf_{K_v(\fp_v^r)}) \ll_\gamma  \fc_v\, q_v^{r\, \fm' }  \int_{G_{\gamma'}(F_v)\bs G_{1_n}(F_v)}  \Phi(g_{1_n}^{-1}(\Pr_2(\gamma')J_n^{-1}) g_{1_n} J_n) \, \frac{\rd g_{1_n,v}}{\rd g_{\gamma',v}}  ,
\]
where $\rd g_{1_n}$ is a Haar measure on $G_{1_n}(F_v)= \Sp_n(F_v)$. 
Since the above integral is regarded as a nilpotent orbital integral on $\mathfrak{sp}_n(F_v)$, we can compute the desired upper bound by \cite{RR72}.
\end{proof}

\begin{lemma}\label{lem:b3}
Consider Case {\it(3)} $E\neq F$ and $M=G$. 
Choose a finite place $v\notin S$. 
We suppose $\gamma\neq \gamma_s$. 
If $r$ is sufficiently large rather than $r_0$, then we have
\[
J_G^G(\gamma,\cf_{K_v(\fp_v^r)})\ll_{\gamma} \fc_v\,  q_v^{r\, \fm' } \times q_v^{-r } .
\]
\end{lemma}
\begin{proof}
This lemma can be proved by the same argument as in the proof of \cref{lem:b1}.     
\end{proof}

\begin{lemma}\label{lem:a1b2}
Assume that \cref{cond:asym} {\em(A1)} holds and $M=G$.
Let $\gamma\in \fo_\negl$. 
Then, we have
\[
I_{G,\gamma}(h_j)\ll_\gamma \Nm(\fn_j)^{\fm} \times \begin{cases} \Nm(\fn_j)^{-1} & \text{Case {\it(1)} $E=F$ and $n$ is odd or Case {\it(3)} $E\neq F$}, \\ \Nm(\fn_j)^{-m+1} & \text{Case {\it(2)} $E=F$ and $n=2m$.} \end{cases}
\]
\end{lemma}
\begin{proof}
By the same argument as in the proof of \cref{lem:a1b1}, $I_G^G(\gamma,h_{v_0})$, $J^G_G(\gamma,h_{S_0})$, $a^G(S',\gamma)$, $\prod_{v\in S'\setminus(S\cup S'')} J^G_G(\gamma,\cf_{\delta_nK_v(\fn_j)})$ are bounded by a constant independent of $j$. 
By Lemmas \ref{lem:b1}, \ref{lem:b2}, and \ref{lem:b3}, for large $j$,
\[
\prod_{v\in S''} J^G_G(\gamma,\cf_{\delta_nK_v(\fn_j)})  \ll_\gamma \prod_{v\in S''}\Nm(\fp_v^{r_{v,j}})^{\fm'} \times \begin{cases} \prod_{v\in S''}\Nm(\fp_v^{r_{v,j}})^{-1} & \text{Case {\it(1)} or Case {\it(3)}}, \\ \prod_{v\in S''}\Nm(\fp_v^{r_{v,j}})^{-2m+1} & \text{Case {\it(2)}.} \end{cases}
\]
Hence, we obtain the assertion by $\Nm(\fn_j)\asymp \prod_{v\in S''} \Nm(\fp_v^{r_{v,j}})$ and \eqref{eq:A1-1}.
\end{proof}

\subsection{Under Condition \ref{cond:asym} (A2)}\label{sec:UnderA2}

Throughout this subsection, we assume Condition \ref{cond:asym} (A2). 
If $\lim_{j\to\inf}|S_j|=\inf$, then we do not know any upper bounds of $a^M(S_j',\gamma)$, and so we avoid evaluating upper bounds of individual terms $I_{G,\gamma}(h_j)$. 
Hence, under Condition \ref{cond:asym} (A2), we use the method of \cite{FLM15} to evaluate upper bounds of $I_{\fo_\negl}(h_j)$. 

\begin{lemma}\label{lem:A2e1}
Under Condition \ref{cond:asym} {\rm(A2)}, for sufficiently large $j$, we obtain
\begin{equation*}
I_\geom(h_j)= \sum_{\fo\in\tilde{\cO}} \int_{G(F)\bs G(\A)^1} \sum_{\gamma\in \fo\setminus\fo_{\negl,\rss}} h_j(g^{-1}\gamma \theta(g)) \, \rd g.
\end{equation*}    
\end{lemma}
\begin{proof}
Recall that $\sigma_{v_0}$ is a supercuspidal representation of $G(F_{v_0})$ and $h_{v_0}$ is a matrix coefficient of $\sigma_{v_0}$. 
For any $\Pi_{v_0}\in\Irr_\ru(G(F_{v_0}))$ such that $\Pi_{v_0}\not\simeq \sigma_{v_0}$, we have $\Pi_{v_0}(h_{v_0})=0$. 
Hence, it follows from the definition \cite{Art88a}*{\S2} that $I_M^G(\gamma,h_{v_0})=J_M^G(\gamma,h_{v_0})$ holds for every Levi subgroup $M\in\cL$. 
By using \eqref{eq:ggeom}, \eqref{eq:geomfo} and Propositions \ref{prop:asym1} and \ref{prop:asym2ss}, and this fact $I_M^G(\gamma,h_{v_0})=J_M^G(\gamma,h_{v_0})$, one obtains
\begin{multline}\label{eq:simple}
I_\geom(h_j)= \sum_{\fo\in \tilde{\cO}}  \sum_{\gamma\in (G(F)\cap(\fo\setminus\fo_{\negl,\rss}))_{G,S_j'}} a^G(S'_j,\gamma)\times  J^G_G(\gamma,h_S) \times \prod_{v\in S_j'\setminus S} J^G_G(\gamma,\cf_{\delta_nK_v(\fn_j)})  \\
+ \sum_{\fo\in \tilde{\cO}}\sum_{M\in\cL,M\neq G} \frac{|W^M_0|}{|W^G_0|} \sum_{\gamma\in (M(F)\cap\fo)_{M,S_j'}} a^M(S'_j,\gamma) \times J_M^G(\gamma,h_{v_0})\\  \times  J^M_M(\gamma,(h_{S_0})_P) \times \prod_{v\in S_j'\setminus S} J^M_M(\gamma,(\cf_{\delta_nK_v(\fn_j)})_P) .
\end{multline}

We also use the fact that $h_{v_0}$ is a cusp form, that is, for the unipotent radical $N$ of any proper parabolic subgroup,
\begin{equation}\label{eq:cuspform}
\int_{N(F_{v_0})}h_{v_0}(x^{-1}ny)\, \rd n =0 \quad (\forall x, y \in G(F_{v_0})).
\end{equation}
Let $J_\fo^T(f)$ denote the modified kernel for $f\in \cH(G(\A)^1)$ and $T\in\fa_0$. 
We refer to \cite{Par19}*{Theorem 4.1 (1)} for the definition of $J_\fo^T(f)$ (there $\fo$ is replaced by the notation $\mathscr{O}$). 
The integral $J_\fo^T(f)$ contains the term $k_{P,\fo}(g,f)$ defined by
\[
k_{P,\fo}(g,f)\coloneqq \int_{N(F)\bs N(\A)} \sum_{\delta\in P(F)\cap\fo} f(g^{-1}\delta n \theta(g)) \,\rd n =\int_{N(\A)} \sum_{\delta\in M(F)\cap\fo} f(g^{-1}\delta n \theta(g)) \,\rd n
\]
where $g\in G(\A)^1$, $M\in\cL$, and $P=MN$ is a standard $\theta$-stable parabolic subgroup, see \cite{Par19}*{Eq. (3) in p.537}. 
Since \eqref{eq:cuspform} implies $k_{P,\fo}(g,h_j)=0$ for any $P\neq G$ and any $g\in G(\A)^1$, by the definition of $J_\fo^T(f)$, 
\[
J^T_\fo(h_j)= \int_{G(F)\bs G(\A)^1} \sum_{\gamma\in \fo} h_j(g^{-1}\gamma \theta(g))\, \rd g.
\]
This means that $J^T_\fo(h_j)$ does not depend on $T\in\fa_0$. 
Hence, we can set $J_\fo(h_j)\coloneqq J^T_\fo(h_j)$, and this description is consistent with the general theory, see \cite{LW13}*{\S11.3}.  
Applying the fine expansion \cite{Art86}*{Lemma 7.1} to $J_\fo(h_j)$, 
for $\fo\in \tilde{\cO}$, 
\[
J_\fo(h_j)=   \sum_{M\in\cL} \frac{|W^M_0|}{|W^G_0|} \sum_{\gamma\in (M(F)\cap\fo)_{M,S_j'}} a^M(S'_j,\gamma) \, J_M^G(\gamma,h_{v_0}\otimes \mathbf{h}_j), \quad  \mathbf{h}_j=h_{S_0}\otimes \left( \otimes_{v\in S_j'\setminus S} \cf_{\delta_nK_v(\fn_j)} \right) .
\]
There is a splitting formula of $J_M^G(\gamma)$ which obtained from \cite{Art88a}*{Corollary 7.4}, see also \cite{Art05}*{(18.7)} and its proof.
Applying it to $J_M^G(\gamma,h_{v_0}\otimes \mathbf{h}_j)$ and using an argument similar to the proof of \cref{lem:geometricsidesimple}, we obtain
\[
J_M^G(\gamma,h_{v_0}\otimes \mathbf{h}_j)=J_M^G(\gamma,h_{v_0})\, J_M^M(\gamma,(\mathbf{h}_j)_P),
\]
since \eqref{eq:cuspform} implies $(h_{v_0})_P\equiv 0$ for any $P\neq G$. 
Therefore, 
\begin{align*}
&\sum_{\fo\in\tilde{\cO}} \int_{G(F)\bs G(\A)^1} \sum_{\gamma\in \fo} h_j(g^{-1}\gamma \theta(g))\, \rd g = \sum_{\fo\in\tilde{\cO}}J_\fo(h_j) \\
&=\sum_{\fo\in \tilde{\cO}}\sum_{M\in\cL} \frac{|W^M_0|}{|W^G_0|} \sum_{\gamma\in (M(F)\cap\fo)_{M,S_j'}} a^M(S'_j,\gamma) \, J_M^G(\gamma,h_{v_0}) \,  J^M_M(\gamma,(h_{S_0})_P) \, \prod_{v\in S_j'\setminus S} J^M_M(\gamma,(\cf_{\delta_nK_v(\fn_j)})_P) .   
\end{align*}
Comparing this with \eqref{eq:simple}, we obtain
\[
\sum_{\fo\in\tilde{\cO}} \int_{G(F)\bs G(\A)^1} \sum_{\gamma\in \fo} h_j(g^{-1}\gamma \theta(g))\, \rd g-I_\geom(h_j)=\sum_{\fo\in\tilde{\cO}}\sum_{\gamma\in\fo_{\negl,\rss}}I_{G,\gamma}(h_j).
\]
This completes the proof together with \cref{lem:mainexpress} for $\fo_{\negl,\rss}$. 
\end{proof}

By Lemmas \ref{lem:mainexpress} and \ref{lem:A2e1} we have
\[
I_\geom(h_j)= \sum_{\fo\in \tilde{\cO}}I_{\fo_\main}(h_j) +  \sum_{\fo\in\tilde{\cO}} \int_{G(F)\bs G(\A)^1}\sum_{\gamma\in \fo_{\negl,\rns}} h_j(g^{-1}\gamma \theta(g)) \, \rd g. 
\]
Combining this with \eqref{eq:mainnegl}, we get
\begin{equation}\label{eq:A2negl}
\sum_{\fo\in \tilde{\cO}}I_{\fo_\negl}(h_j)=I_\geom(h_j)-\sum_{\fo\in \tilde{\cO}}I_{\fo_\main}(h_j)=\sum_{\fo\in\tilde{\cO}} \int_{G(F)\bs G(\A)^1} \sum_{\gamma\in \fo_{\negl,\rns}} h_j(g^{-1}\gamma \theta(g)) \, \rd g.    
\end{equation}
From now on, in this subsection, we will find an upper bound on RHS of \eqref{eq:A2negl}.

For $T\in\fa_0$ and $g\in G(\A)^1$, we denote by $F^G(g,T)$ the truncation function on $G(F)\bs G(\A)^1$, which is the characteristic function of a certain compact subset of $G(F)\bs G(\A)^1$, see \cite{Art78}*{p.941}. 
Set
\[
I^T_{\negl,1}(h_j)\coloneqq  \sum_{\fo\in\tilde{\cO}} \int_{G(F)\bs G(\A)^1} (1-F^G(g,T))  \sum_{\gamma\in \fo_{\negl,\rns}} h_j(g^{-1}\gamma \theta(g)) \, \rd g,
\]
\[
I^T_{\negl,2}(h_j)\coloneqq \sum_{\fo\in\tilde{\cO}} \int_{G(F)\bs G(\A)^1} F^G(g,T) \sum_{\gamma\in \fo_{\negl,\rns}} h_j(g^{-1}\gamma \theta(g)) \, \rd g,
\]
for $T\in\fa_0$. 
Since $\sum_{\fo\in \tilde{\cO}}I_{\fo_\negl}(h_j)=I^T_{\negl,1}(h_j)+I^T_{\negl,2}(h_j)$ by \eqref{eq:A2negl}, our purpose is to evaluate their upper bounds.
We use some results of the paper \cite{Par19}, which treats the absolute convergence of twisted trace formulas. 
Note that the Root cone lemma of \cite{Par19}*{Lemma 4.5} is assumed in general, but it is proved in \cite{Par19}*{Lemma 6.2} for our case.

Let $\Delta_0$ denote the set of simple root of $\mathbb{S}_0$ with respect to $P_0$. 
Write $\hat\Delta_0^\vee$ for the set of simple co-weights, that is, $\hat\Delta_0^\vee$ is the basis of $\fa_0$ dual to $\Delta_0$. 
Throughout this subsection, we assume that $T$ is in the half line $\R_{>0}(\sum_{\hat\varpi^\vee\in\hat\Delta_0^\vee} \hat\varpi^\vee)$, which is given in \cite{Par19}*{Lemma 4.8}, and $d(T)$ is sufficiently large, where $d(T)\coloneqq \min_{\alpha\in\Delta_0} \langle \alpha, T\rangle$. 
This assumption is referred to as ``$T$ is suitably large multiple of the sum of positive coroots" in \cite{Par19}*{Theorem 4.1}. 
\begin{lemma}\label{lem:bound1}
There are real numbers $l$, $c_0$, $d_0>0$ such that 
\[
\left|\, I^T_{\negl,1}(h_j) \, \right| \ll \Nm(\fn_j)^l \, (1+\|T\|)^{d_0}\, e^{-c_0\, d(T)},
\]
where $\| \; \|$ is the Euclidean norm on $\fa_0$. 
\end{lemma}
\begin{proof}
   It follows from \eqref{eq:cuspform} and \cite{Par19}*{Theorem 4.1} that
   \begin{align*}
       & \sum_{\fo\in\tilde{\cO}}\left| \int_{G(F)\bs G(\A)^1} (1-F^G(g,T))  \sum_{\gamma\in \fo} h_j(g^{-1}\gamma \theta(g)) \, \rd g \, \right| \\
       & =\sum_{\fo\in\tilde{\cO}} \left| \int_{G(F)\bs G(\A)^1} F^G(g,T) \sum_{\gamma\in\fo} h_j(g^{-1}\gamma \theta(g)) \, \rd g  - J_\fo^T(h_j)  \, \right| \ll \Nm(\fn_j)^l \, (1+\|T\|)^{d_0}\, e^{-c_0\, d(T)}.
   \end{align*}
   See \cite{Par19}*{Theorem 4.1 (1)} for the definition of $J_\fo^T(h_j)$. 
   Hence, to obtain the assertion, it is sufficient to prove
   \begin{equation}\label{eq:mainnegss}
    \sum_{\fo\in\tilde{\cO}}\left| \int_{G(F)\bs G(\A)^1} (1-F^G(g,T))  \sum_{\gamma\in \mathscr{O}(\fo)} h_j(g^{-1}\gamma \theta(g)) \, \rd g \, \right| \ll  \Nm(\fn_j)^l \, (1+\|T\|)^{d_0}\, e^{-c_0\, d(T)}.        
   \end{equation}
   where $\mathscr{O}(\fo)\coloneqq \fo_\main\sqcup\fo_{\negl,\rss}$. 
   This inequality \eqref{eq:mainnegss} follows from a twisted version of \cite{FL16}*{Theorem 7.1}. 
   That is not addressed in \cite{Par19}, but it is hopefully what the argument of \cite{FL16}*{\S6} can prove. 
   So here is a brief explanation of how to use \cite{FL16}*{Lemma 6.4} for this special setting \eqref{eq:mainnegss}. 
   For Case {\it(2)} $n=2$, we have $\mathscr{O}(\fo)=\fo$ by \cref{lem:evenfo}, and so this case is excluded. 
   Then, every element in $\mathscr{O}(\fo)$ is not induced from any proper parabolic subgroups, see \cite{FL16}*{\S6.3} for the meaning and use \cref{lem:evenfo} for Case {\it(2)} $n>2$. 
   Under this situation, it is expected to omit the truncation, see $k_{\fo,P}$ in \cite{FL16}*{p.453} (non-twisted case), and so it is enough to use the decomposition
   \begin{equation}\label{eq:decompositionArt}
   \sum_{P\supset P_0} \sum_{\delta\in P(F)\bs G(F)} F^P(\delta g,T)\, \tau_P^G(H_0(\delta g)-T)  =1 ,\qquad g\in G(\A)^1 ,    
   \end{equation}
   where $P$ runs over standard parabolic subgroups, see \cite{Art78}*{Lemma 6.4} for the details (it is also explained in \cite{FL16}*{Eq. (3) in p.429} and \cite{Par19}*{Eq. (2) in p.535}). 
   Write $Q$ for the smallest $\theta$-stable standard parabolic subgroup including $P$. 
   Then, since $d(T)$ is sufficiently large and the support of $h_j$ is contained in a compact subset in $G(\A)^1$, which is independent of $j$, by using \cite{LW13}*{Corollaire 3.6.7} (see also \cite{Art78}*{p.943--944} and \cite{Art05}*{p.44}), the LHS of \eqref{eq:mainnegss} is bounded by a sum of 
   \[
   \sum_{\fo\in\tilde{\cO}}\int_{P(F)\bs G(\A)^1} \sum_{\eta\in  M_Q(F)} \sum_{\nu\in N_Q(F), \eta\nu\in\mathscr{O}(\fo)}  \left| h_j(g^{-1}\eta \nu \theta(g))\right| \, F^P(g,T)\, \tau_P^G(H_0(g)-T) \, \rd g 
   \]
   over $P_0\subset P \subset Q\subset G$, where $Q=M_QN_Q$ and $M_Q\in\cL$. 
   Since the sum over $\mu$ is finite, see \cite{LW13}*{Proof of Proposition 9.1.1}, it is enough to consider the case $Q\neq G$.  
   Applying \cite{FL16}*{Lemma 6.4} to $\sum_{\nu\in N_Q(F), \eta\nu\in\mathscr{O}(\fo)}  \left| h_j(g^{-1}\eta \nu \theta(g))\right|$, we can prove \eqref{eq:mainnegss}. 
   This completes the proof. 
\end{proof}

Since $F^G(g,T)$ is right $K$-invariant from its definition, by a calculation similar to \eqref{eq:K_vintegration} we obtain
\[
I^T_{\negl,2}(h_j)\, = \sum_{\fo\in\tilde{\cO}} \int_{G(F)\bs G(\A)^1} F^G(g,T) \sum_{\gamma\in\fo_{\negl,\rns}}\int_{K} h_j(k^{-1}g^{-1}\gamma \theta(g)\theta(k))\, \rd k \, \rd g. 
\]
Hence, by using \cref{lem:red} we deduce
\begin{equation}\label{eq:rem2}
I^T_{\negl,2}(h_j)= \fM^S(\fn_j)  \sum_{\fo\in\tilde{\cO}} \int_{G(F)\bs G(\A)^1} F^G(g,T) \sum_{\gamma\in\fo_{\negl,\rns}} h_{j,K}(g^{-1}\gamma \theta(g)) \, \rd g 
\end{equation}
where
\[
\fM^S(\fn_j)\coloneqq \prod_{v\notin S} \fM_v(r_{v,j}), \quad \fn_j=\prod_{v\notin S} \fp_v^{r_{v,j}},
\]
\[
\cf^S_{j,K^S}\coloneqq  \bigotimes_{v\notin S} \cf_{K_v(\fn_j)U_v}, \quad h_{S,K_S}(x)\coloneqq \int_{K_S} h_S(k_S^{-1}x\, \theta(k_S)) \, \rd k_S ,
\]
\begin{equation}\label{eq:h_jK}
\tilde{h}_{j,K}\coloneqq h_{S,K_S}\otimes \cf^S_{j,K^S}, \quad h_{j,K}(g)\coloneqq \int_{\R_{>0}} \tilde{h}_{j,K}(ag)\, \rd^\times a.    
\end{equation}

We will need the following lemma.
\begin{lemma}\label{lem:1upper}
    Let $\Omega_1$ be a compact subset in $F_\inf$, and $a\in\R_{>0}$. 
    Then, we have $\#\{ a\Omega_1\cap (\fn_j\setminus \{0\})\}\ll_{\Omega_1} a^{[F:\Q]}\, \Nm(\fn_j)^{-1}$. 
\end{lemma}
\begin{proof}
    See \cite{FLM15}*{Lemma 3.9}. 
\end{proof}

\begin{lemma}\label{lem:flm}
For sufficiently large $j$ and $d(T)$, we have 
\begin{equation}\label{eq:geop}
\int_{G(F)\bs G(\A)^1} F^G(g,T) \sum_{\gamma\in\fo_{\negl,\rns}} h_{j,K}(g^{-1}\gamma \theta(g)) \, \rd g 
\end{equation}
is bounded by a constant multiple of 
\[
(1+\|T\|)^{\dim\fa_0}\times\begin{cases}
\Nm(\fn_j)^{-1} & \text{Case {\it(1)} $E=F$ and $n$ is odd or Case {\it(3)} $E\neq F$,} \\
\Nm(\fn_j)^{-2m+1} & \text{Case {\it(2)} $E=F$ and $n=2m$.} 
\end{cases} 
\] 
\end{lemma}
\begin{proof}
This lemma is proved by arguments similar to \cite{FLM15}*{\S 3}.

First, we apply the argument of Arthur in \cite{Art78}*{\S 5}. 
Denote by $A_{P_0}$ the connected component of $1_n$ in $\bS_0(\R)$. 
Note that $A_{P_0}$ is identified with $\fa_0$ via the exponential map. 
For a parabolic subgroup $P\supset P_0$ with a Levi component $M\supset M_0$, denote by $\Delta_P$ the set of simple roots of $\bS_M$ with respect to $P$, and set $A_P\coloneqq\bS_M(\R)\cap A_{P_0}$. 
Fix a parameter $T_1\in\fa_0$, and we set 
\[
A_{P}(T_1)\coloneqq \{a\in A_{P} \mid \alpha(a)>e^{\alpha(T_1)} \;\; \text{for all $\alpha\in\Delta_P$} \},
\]
\[
A_{P_0}(T_1,T)\coloneqq \{ a\in A_{P_0}(T_1) \mid \beta(a) < e^{\beta(T)} \;\; \text{for all $\beta\in\Delta_P$} \},
\]
where $\hat\Delta_0$ denotes the set of simple weights. 
There exists a suitable compact subset $\Gamma$ in $P_0(\A)$ (Siegel set) so that, for
\[
x=pak\in G(\A)^1,\quad p\in\Gamma, \;\; a\in A_{P_0}(T_1,T),\;\; k\in K,
\]
we have $a\in A_{P_0}(T_1,T)$ if $F^G(x,T)=1$. 
Since $y=a^{-1}pa\in \Gamma$, we see 
\[
\eqref{eq:geop} \ll \int_{A_{P_0}(T_1,T)} \delta_{P_0}(a)^{-1} \sum_{\gamma\in \fo_{\negl,\rns}}\int_{\Gamma}\left| (h_{j,K} (y^{-1}a^{-1} \gamma \theta(a)\theta(y)) \right| \, \rd y \, \rd a ,
\]
where $\delta_{P_0}$ is the modular function of $P_0$. 
Since $\int_{A_{P_0}(T_1,T)}\rd a \ll (1+\|T\|)^{\dim\fa_0}$, \eqref{eq:geop} is bounded by a constant multiple of the product of $(1+\|T\|)^{\dim\fa_0}$ and 
\begin{equation}\label{eq:b1}
\sup_{a\in A_{P_0}(T_1)} \delta_{P_0}(a)^{-1}\sum_{\gamma\in \fo_{\negl,\rns}} \phi(a^{-1}\gamma \theta(a)),
\end{equation}
where
\[
\phi(x)\coloneqq \sup_{y\in\Gamma} \left| (h_{j,K} (y^{-1} x \theta(y)) \right| . 
\]
There exist a finite set $S_1(\supset S)$ of places of $F$ and compact subsets $\Gamma_v'$ of $G(F_v)$ for all $v\in S_1$ so that
\[
\Gamma\subset \prod_{v\in S_1} \Gamma_v' \times\prod_{v\notin S'} K_v.
\]
For each $v\notin S$, we set $U_v'\coloneqq U_v$ when $v\notin S_1$, and $U_v'\coloneqq \{ g \delta_n \theta(g^{-1}) \mid g\in\Gamma_v' \}$ when $v\in S_1\setminus S$.
And we define a function $\phi_1$ on $G(\A)$ as
\[
\phi_1= \phi_{1,S}\otimes \left(\bigotimes_{v\notin S} \cf_{K_v(\fn_j)U_v'} \right) , \quad \text{$\phi_{1,S}\in\cH(G(F_S))$ is a non-negative function.}
\]
Since $\phi$ is bouned by $\phi_1$, \eqref{eq:b1} is bounded by
\begin{equation}\label{eq:bn1}
\sup_{a\in A_{P_0}(T_1)} \delta_{P_0}(a)^{-1}\sum_{\gamma\in \fo_{\negl,\rns}} \phi_1(a^{-1}\gamma \theta(a)).
\end{equation}

For a parabolic subgroup $P(\supset P_0)$, a subset $\Delta_0^P$ of $\Delta_0$ is defined by $A_P=\{ a\in A_{P_0} \mid  \alpha(a)=1$ for all $ \alpha\in \Delta_0^P\}$.  
We write $Q$ for the $\theta$-stable parabolic subgroup such that $\Delta_0^Q=\Delta_0^P\cup\theta(\Delta_0^P)$
(cf. \cite{LW13}*{\S 2.11}). 
Set
\[
A_Q^\theta(T_1)\coloneqq \{a\in A_Q(T_1) \mid a=\theta(a) \} . 
\]
Dividing $A_{P_0}(T_1)$ by the decomposition \eqref{eq:decompositionArt} of \cite{Art78}*{Lemma 6.4}, and using \cite{LW13}*{Corollary 3.6.7} and \cite{LW13}*{Proof of Proposition 9.1.1}, we see that \eqref{eq:bn1} is bounded by a constant multiple of a finite sum of 
\begin{equation}\label{eq:b2}
\sup_{a\in A_Q^\theta(T_1)} \delta_Q(a)^{-1} \sum_{\nu\in N_Q(F): \, \mu\nu\in \fo_{\negl,\rns}} \phi_2(a^{-1}\mu\nu a)
\end{equation}
for $P(\supset P_0)$ and finitely many $\mu \in M_P(F)$. 
Note that the range of $\mu$ in the sum is independent of $j$, and here we define a finction $\phi_2$ as
\[
\phi_2(x)\coloneqq \sup_{b\in B} \delta_{P_0}(b)^{-1} \phi_1 (b^{-1}x \theta(b))
\]
for some compact subset $B$ in $A_0$, and $\phi_2$ is bounded by 
\begin{equation}\label{eq:support}
 \phi_{2,S}\otimes \bigotimes_{v\notin S} \cf_{K_v(\fn_j)U_v'} \quad \text{for some non-negative function $\phi_{2,S}\in\cH(G(F_S))$.}        
\end{equation}
Therefore, it is sufficient to prove a desired upper bound of $\eqref{eq:b2}$ for each $P$ and $\mu$. 
From here, we fix $P$ and $\mu$. 

Let $\nu\in N_P(F)$ and $\mu\nu\in\fo_{\negl,\rns}$. 
By Lemmas \ref{lem:rank} and \ref{lem:evenfo}, the condition $\mu\nu\in \fo_{\negl,\rns}$ implies 
\begin{equation}\label{eq:rank}
\begin{cases}
\mathrm{rank}(\Pr_2(\mu\nu))\ge 2 & \text{if $E=F$,} \\
\mathrm{rank}(\Pr_2(\mu\nu))\ge 1 & \text{if $E\neq F$}.
\end{cases}    
\end{equation}
Let $W\coloneqq \{n-1_n \mid n\in N_P\}$.
As vector spaces over $F$, $W$ is a subspace of $\M_n$ if $E= F$, and of $\Res_{E/F}\M_n$ if $E\neq F$.
Considering the support of $\phi_2$ by \eqref{eq:support}, there exists a compact subset $\Omega$ in $W(F_\inf)$ and a $\fo_F$-lattice $L_j$ in $W(F)$ such that, if $\phi_2(a^{-1}\mu\nu \theta(a))\neq 0$, then
\[
\mu\nu \in  \mu + a\Omega \theta(a)^{-1} \bigcap L_j. 
\]
When $E=F$ (resp. $E\neq F$), we set 
\[
W_1\coloneqq W\cap \AL_nJ_n^{-1} \quad (\text{resp.} \quad W_1\coloneqq W\cap \HE_n J_n^{-1}),
\]
\[
W_2\coloneqq W\cap \SM_nJ_n^{-1} \quad (\text{resp.} \quad W_2\coloneqq W\cap \HE_n' J_n^{-1}) . 
\]
Then, $W=W_1\oplus W_2$. 
Let $\Lambda_1$ and $\Lambda_2$ be $\fo_F$-lattices such that $\Lambda_j\subset W_j$.
Regarding Case {\it(2)} $E=F$ and $n$ is even and Case {\it(3)} $E\neq F$, we set $\Lambda_{2,j}\coloneqq \fn_j \Lambda_2$. 
Regarding Case {\it(1)} $E=F$ and $n$ is odd, we set 
\[
\Lambda_{2,j}\coloneqq \bigsqcup_{x\in \Lambda_2^{(1)}} \left( \fo_F x +\fn_j \Lambda_2 \right)
\]
where $\Lambda_2^{(1)}\coloneqq\{x\in \Lambda_2/\fn_j \Lambda_2 \mid \mathrm{rank}(x)\leq 1\}$.
By \cref{lem:pr} we can choose $\Lambda_1$ and $\Lambda_2$ so that $L_j\subset \Lambda_1\oplus \Lambda_{2,j}$. 
Hence, by \eqref{eq:rank}, the sum
\[
\delta_Q(a)^{-1} \sum_{\nu\in N_Q(F): \, \mu\nu\in \fo_{\negl,\rns}} \phi_2(a^{-1}\mu\nu a)
\]
in \eqref{eq:b2} is bounded by a constant multiple of
\begin{equation}\label{eq:bn2}
\delta_Q(a)^{-1}\times \#\{\gamma\in a\Omega \theta(a)^{-1} \cap (\Lambda_1\otimes\Lambda_{2,j}) \mid \mathrm{rank}(\Pr_2(\gamma))\ge r_{E,n}\} 
\end{equation}
where $r_{E,n}=1$ when $E\neq F$, and $r_{E,n}=2$ when $E=F$. 
Let $\Phi_+^Q$ denote the set of positive roots on $\fa_Q^\theta$ corresponding to $W_2$. 
By \cref{lem:1upper} and a bound of the sum over $\Lambda_1$, \eqref{eq:bn2} is bounded by a constant multiple of
\begin{equation}\label{eq:bn3}
\delta_Q^\theta (a)^{-1}\times \#\{\gamma\in a (\Omega\cap W_2(F_\inf))  \theta(a)^{-1} \cap \Lambda_{2,j} \mid \mathrm{rank}(\Pr_2(\gamma))\ge r_{E,n}\}     
\end{equation}
where $\delta_Q^\theta$ is the sum of positive roots in $\Phi_+^Q$ with multiplicities.

Any element $x\in \SM_n(F)$ with $\mathrm{rank}(x)= 1$ is expressed as $x=ay\t y$ for some $y\in \M_{1\times r}(F)$ and some $a\in F^\times$. 
Hence, for $a_1$, $a_2\in F^\times$ and $x_1$, $x_2\in \SM_n(F)$ such that $\mathrm{rank}(x_1)=\mathrm{rank}(x_2)=1$, the rank of $a_1x_1+a_2x_2$ is less than $2$ only if $x_1$ and $x_2$ are linearly dependent.  
For any $\fo_F$-lattice $\Lambda$ in $\SM_n(F)$ such that $\mathrm{rank}(\Lambda)\ge 2$, using this fact one can prove
\[
\# \{ x\in \Lambda/\fn_j\Lambda \mid \mathrm{rank}(x)\le 1 \} \ll \Nm(\fn_j)^{\mathrm{rank}(\Lambda)-1}. 
\]
This means $\#(\Lambda_2^{(1)})\ll \Nm(\fn_j)^{\mathrm{rank}(\Lambda_2)-1}$ for Case {\it(1)} $E=F$ and $n$ is odd. 
Therefore, for Case {\it(1)} $E=F$ and $n$ is odd or Case {\it(3)} $E\neq F$, we obtain
\[
\eqref{eq:bn3} \ll \Nm(\fn_j)^{-1}
\]
by \cref{lem:1upper} and the condition for $r_{E,n}$, since $\alpha(a)^{-1}$ is bounded for any $a\in A_Q^\theta(T_1)$ and any positive root $\alpha$ on $W_2$. 
Thus, the proof is completed for those cases.

From now on, we treat only Case {\it(2)} $E=F$ and $n$ is even. 
We use the argument in \cite{FLM15}*{Proof of Lemma 3.9}. 
Since $\Lambda_{2,j}$ is always a lattice in $W_2(F)$ of $P_0$, the evaluation for \eqref{eq:bn3} can be reduced to the case $P=Q=P_0$. 
Note that $a$ and $\Omega$ are fixed under the evaluation for \eqref{eq:bn3}. 
Suppose $P=Q=P_0$ and let $\Phi_+\coloneqq \Phi_+^{P_0}$. 
For each $X$ in $W_2$ and $\alpha\in\Phi_+$, write $X_\alpha$ for the projection of $X$ to the subspace $W_{2,\alpha}\coloneqq \{ Y\in W \mid \mathrm{Ad}(a)Y=\alpha(a) Y\}$. 
For $X\in W_2$, we set
\[
\cS(X)\coloneqq \{\alpha\in \Phi_+\mid X_\alpha\neq 0\}.
\]
For each subset $\cS(\neq\emptyset)$ in $\Phi_+$ and $a\in A_{P_0}^\theta(T_1)$, we have only to prove an upper bound of 
\begin{equation}\label{eq:bn4}
\delta_{P_0}^\theta(a)^{-1} \times \#\{\gamma\in a (\Omega\cap W_2(F_\inf))  \theta(a)^{-1} \cap \Lambda_{2,j} \mid \mathrm{rank}(\Pr_2(\gamma))\ge 2\}   
\end{equation}
with respect to $\Nm(\fn_j)$. 
By \cref{lem:1upper} we obtain
\begin{equation}\label{eq:bn5}
    \eqref{eq:bn4} \ll \prod_{\alpha\notin\cS} \alpha(a)^{-[F:\Q]} \times \Nm(\fn_j)^{-\# \cS}. 
\end{equation}
As a condition for the set of lattice points of \eqref{eq:bn4} not to disappear, we see
\[
    \beta(a)^{-[F:\Q]} \ll \Nm(\fn_j)^{-1} \qquad \text{for any $\beta\in\cS$.}
\]
Hence, for $\alpha\notin \cS$ and $\beta\in \cS$ we have
\begin{equation}\label{eq:bn6}
   \alpha(a) > \beta(a) \quad \Rightarrow \quad  \alpha(a)^{-[F:\Q]} \ll \Nm(\fn_j)^{-1} .
\end{equation}
Define a semi-order on $\Phi_+$ as
\[
\alpha <\beta \quad \Leftrightarrow \quad \alpha(a)<\beta(a)
\]
where $a\in A_{P_0}^\theta(T_1)$. 
In addition, define a semi-order on the power set $\fP(\Phi_+)$ of $\Phi_+$ as
\[
\cS_1 <\cS_2 \quad \Leftrightarrow \quad \text{we have $\alpha<\beta$ for all $\alpha\in\min(\cS_1)$ and $\beta\in\min(\cS_2)$} 
\]
where $\min(\cS_j)$ means the subset of minimum elements in $\cS_j$. 
By \eqref{eq:bn5} and \eqref{eq:bn6}, if $\cS_1<\cS_2$, the upper bound of \eqref{eq:bn4} for $\cS_1$ with respect to $\Nm(\fn_j)$ is better than that for $\cS_2$.  
The largest element in $\fP(\Phi_+)$ is the empty set, but it is excluded by the condition $\mathrm{rank}(\Pr_2(\gamma))\ge 2$. 
Define roots $\alpha_1$ and $\alpha_2$ in $\Phi_+$ by
\[
\alpha_1(a)=a_1^2 \quad \text{and} \quad \alpha_2(a)=a_1a_2 \quad \text{for $a=\diag(a_1,a_2,\dots,a_2^{-1},a_1^{-1})\in A_{P_0}^\theta(T_1)$.}
\]
The next largest element in $\fP(\Phi_+)$ is $\{\alpha_1\}$, but it is also excluded by $\mathrm{rank}(\Pr_2(\gamma))\ge 2$.
The third largest elements in $\fP(\Phi_+)$ are $\cS=\{ \alpha_2\}$ or $\cS=\{ \alpha_1,\alpha_2\}$. 
They have the same upper bound. 
For $\cS=\{ \alpha_2\}$, we derive by \eqref{eq:bn5} and \eqref{eq:bn6}  
\[
 \eqref{eq:bn4} \ll\prod_{\alpha\notin \{\alpha_2\}} \alpha(a)^{- [F:\Q] } \times N(\fn_j)^{-1} \ll N(\fn_j)^{-2m+1}.
\]
This completes the proof for Case {\it(2)} $E=F$ and $n$ is even. 
\end{proof}

Consider a sequence $\{T_j\}_{j\in\N}$ of truncation parameters such that $T_j$ is on the half-line required for \cref{lem:bound1}. 
Let $l$, $c_0$, $d_0$ be the real numbers of \cref{lem:bound1}, and let $c_1>0$ be a real number such that $c_1-c_0^{-1}l>n$. 
We can choose a sequence $\{T_j\}_{j\in\N}$ so that
\[
d(T_j) = c_1\times \log \Nm(\fn_j).
\]
Note that for such $T_j$, we have $1+\|T_j\| \ll d(T_j)$. 
Then, by \cref{lem:bound1} we have
\begin{equation}\label{eq:A2b1}
I^{T_j}_{\negl,1}(h_j) \ll \Nm(\fn_j)^{l-c_0c_1 }  \, (1+\log \Nm(\fn_j))^{d_0} .   
\end{equation}
Since \cref{lem:fmbound} deduces $\fM^S(\fn_j) \ll \Nm(\fn_j)^{\fm'} \, \log \Nm(\fn_j)$, by using \eqref{eq:rem2} and \cref{lem:flm} we have
\begin{multline}\label{eq:A2b2}
I^{T_j}_{\negl,2}(h_j) \ll  (1+\log \Nm(\fn_j))^{\dim\fa_0+1} \\
\times\Nm(\fn_j)^{\fm} \times \begin{cases}
\Nm(\fn_j)^{-1} & \text{if [$E=F$ and $n$ is odd] or $E\neq F$,} \\
\Nm(\fn_j)^{-m+1} & \text{if $E=F$ and $n=2m$.} 
\end{cases} 
\end{multline}

\subsection{Main terms}\label{sec:Main}

\begin{proposition}\label{prop:mainterm1}
For Case {\it(1)} $E=F$ and $n$ is odd, there exists a constant $c>0$ such that   
\[
\sum_{\fo\in \tilde{\cO}} I_{\fo_\main}(h_j)= c \, \fm^S(\fn_j) \sum_{\omega_S \in \cQ_S} \#\cQ(\omega_S,j) \,  I^\theta(h_S,\omega_S)  .
\]    
\end{proposition}
\begin{proof}
Recall the notations $V$ and $H$ in \S\ref{sec:vecsp}, and let $\mathrm{Pf}(x)$ denote the Pfaffian for $x\in \AL_n$. 
An isomorphism $\fT\colon V\to \AL_{n+1}$ (as vector spaces) is defined by
\[
\fT(x_1,x_2)=\begin{pmatrix} x_1& \t x_2 \\ -x_2 &0 \end{pmatrix} , 
\]
and we have $V^0=\{ x\in V \mid \mathrm{Pf}(\mathfrak{T}(x))\neq 0 \}$. 
Take a generic point $x_0\coloneqq \left( \begin{pmatrix} 0&0 \\ 0& - J_{n-1} \end{pmatrix}, e_1 \right)$ in $V^0$, where $e_j$ is the $j$-th basic vector in $\M_{1\times n}$. 
Then, $\fT(x_0)=J_{n+1}$. 
Note that we have $V^0(k)=x_0 \cdot G(k)$ for an arbitrary field $k$ of characteristic $0$ (cf. \cite{HKK88}). 

Let $\gamma$ be an element in $\sqcup_{\fo\in\tilde{\cO}} \fo_\main$. 
Since $\gamma\in Y_{1,n-1}$, we have $\Pr_1(\gamma)\in\AL_n$ and $\mathrm{rank}(\Pr_1(\gamma))=n-1$. 
Hence, there exists an element $g\in G(F)$ such that
\[
\Pr_1(\t g \, \gamma \, \theta(\t g^{-1})) = \t g \, \Pr_1(\gamma) \, g= \begin{pmatrix} 0&0 \\ 0& - J_{n-1} \end{pmatrix}.
\]
Since $\gamma\in Y_{1,n-1}$, we have $\Pr_2(\t g \, \gamma \, \theta(\t g^{-1}))=a \, \t x_2 \, x_2$ for some $a\in F^\times$ and some $x_2\in \M_{1\times n}(F)$ such that the first entry of $x_2$ is non-zero. 
Hence an arbitrary element $\gamma$ in $\sqcup_{\fo\in\tilde{\cO}} \fo_\main$ is expressed as $\gamma = a \, H(x) J_n^{-1}$ for some $a\in F^\times$ and some $x\in V^0(F)$.
On the other hand, by $H(x_0)J_n^{-1}\in Y_{1,n-1}$ and $V^0(k)=x_0 \cdot G(k)$, one can prove that $a \, H(x)J_{2m+1}^{-1}$ belongs to $\sqcup_{\fo\in\tilde{\cO}} \fo_\main$ for any $a\in F^\times$ and $x\in V^0(F)$. 
Therefore, we obtain
\begin{equation}\label{eq:deco1}
\bigsqcup_{\fo\in\tilde{\cO}} \fo_\main=\bigsqcup_{a\in F^\times/(F^\times)^2} \{ a \, H(x)J_{2m+1}^{-1} \mid x\in V^0(F)\}.
\end{equation}

Write $G_{x_0}(\simeq \Sp_{n-1})$ for the stabilizer of $x_0$ in $G$. 
Set $c_v=1$ if $v\mid\inf$, and $c_v=(1-q_v^{-1})^{-1}$ if $v<\inf$. 
According to \cite{Wei82}*{Ch.2}, we denote by $\rd x$ a gauge form on $V$ defined over $F$, and a Haar measure $\rd x_v$ on $V(F_v)$ is induced from $\rd x$. 
Then, an $G(\A)$-invariant measure $\rd^\times x$ on $V(\A)$ is defined by
\begin{equation}\label{eq:measureAL}
\rd^\times x\coloneqq c'\prod_v  \rd^\times x_v ,\qquad \rd^\times x_v\coloneqq \frac{c_v\, \rd x_v}{|\mathrm{Pf}\circ\fT(x_v)|_v^{n}}
\end{equation}
for a constant $c'>0$. 
Let $\rd g$ (resp. $\rd g_{x_0}$) denote the Tamagawa measure on $G(\A)$ (resp. $G_{x_0}(\A)$). 
Then, the quotient measure $\rd g/\rd g_{x_0}$ equals a constant multiple of the invariant measure $\rd^\times x$, that is, for some $c''>0$ we have $\rd g/\rd g_{x_0}=c''\times\rd^\times x$.  
By using \cref{lem:mainexpress}, \eqref{eq:deco1}, $\int_{G_{x_0}(F)\bs G_{x_0}(\A)}\rd g_{x_0}=1$ and $V^0(F)=x_0\cdot G(F)$ we have
\begin{align}\label{eq:oddmain1}
\sum_{\fo\in \tilde{\cO}} I_{\fo_\main}(h_j) =& \sum_{a\in F^\times/(F^\times)^2} \frac{2}{n}\int_{G_{x_0}(F)\bs G(\A)} \tilde{h}_j( a \cH(x_0\cdot g) J_n^{-1})\, \rd g \\
=& \frac{2}{n} \sum_{a\in F^\times/(F^\times)^2} \int_{G_{x_0}(F)\bs G_{x_0}(\A)}\rd g_{x_0}\int_{G_{x_0}(\A)\bs G(\A)} \tilde{h}_j( a \cH(x_0\cdot g) J_n^{-1})\, \frac{\rd g}{\rd g_{x_0}} \nonumber \\
=&\frac{2c''}{n} \sum_{a\in F^\times/(F^\times)^2} \int_{V(\A)} \tilde{h}_j( a H(x)J_n^{-1})\, \rd^\times x \nonumber \\
=&\frac{2c''}{n} \sum_{a\in F^\times/(F^\times)^2} \int_{\A^\times} \int_{\A^\times\bs V(\A)} \tilde{h}_j( ab^2 H(x)J_n^{-1})\, \rd^1 x\, \rd^\times b. \nonumber 
\end{align}
Here, we set $\rd^1 x=c'\prod_v \rd^1 x_v$, $\rd^1 x_v=\rd^\times x_v/\rd^\times b_v$, $\rd^\times b=\prod_v \rd^\times  b_v$, $\rd^\times b_v=c_v \, \rd b_v/ |b_v|_v$.
We will use the following formula (see, e.g., \cite{HW18}*{Ch.3}). 
\begin{equation}\label{eq:LL}
\sum_{a\in F^\times/(F^\times)^2} \int_{\A^1} f(a b^2)\, \rd^\times b=\sum_{\chi\in \cQ} \int_{\A^1} f( b) \, \chi(b) \, \rd^\times b \qquad (f\in \cS(\A)).
\end{equation}
Applying \eqref{eq:LL} to the last line of \eqref{eq:oddmain1}, we see for some $c'''>0$
\begin{equation}\label{eq:oddmain2}
\sum_{\fo\in \tilde{\cO}} I_{\fo_\main}(h_j)=c'''\sum_{\chi=\otimes_v \chi_v\in\cQ}   I^\theta(h_S,\chi_S) \times  \prod_{v\notin S} I^\theta(\cf_{\delta_n K^S(\fn_j)},\chi_v),
\end{equation}
where $\chi_S\coloneqq\otimes_{v\in S}\chi_v$. 
For each $v\notin S$, we derive by a change of variable
\begin{multline}\label{eq:oddmainlocal}
I^\theta(\cf_{\delta_n K_v(\fn_j)},\chi_v)=\int_{F_v^\times/(F_v^\times)^2} \rd^\times a\,  \int_{\AL_n(F_v)}  \int_{M_{1,n}(F_v)} \frac{ \rd x_1 \, \rd x_2 }{|\mathrm{Pf}\circ\mathfrak{T}(x_1,x_2)|_v^{n} } \\ \cf_{\delta_n K_v(\fn_j)}((x_1 + a\, \t x_2 x_2)J_n^{-1})\, \chi_v(a).  
\end{multline}
Here, we have normalized the measures as
\[
\int_{\AL_n(\fo_{F,v})}\rd x_1\times\int_{M_{1,n}(\fo_{F,v})}\rd x_2=c_v, \qquad \int_{\fo_{F,v}^\times/(\fo_{F,v}^\times)^2}\rd^\times a=1.
\]

Take a finite place $v\notin S$ such that $\fp_v\nmid \fn_j$ (i.e., $v\notin S_j$). 
Then $\delta_nJ_n K_v(\fn_j)= K_v$. 
Let $(x_1,x_2)\in V(F_v)$ and $a\in (\fo_{F,v}^\times\sqcup \varpi_v\fo_{F,v}^\times)/(\fo_{F,v}^\times)^2$. 
Using $\Pr_1((x_1 + a\, \t x_2 x_2)J_n^{-1})=x_1$ and $\Pr_2((x_1 + a\, \t x_2 x_2)J_n^{-1})=a\, \t x_2 x_2$, 
one can prove that 
\[
\text{$x_1 + a\, \t x_2 x_2\in K_v$ if and only if $(x_1,x_2)\in x_0\cdot K_v$ and $a\in \fo_{F,v}^\times/(\fo_{F,v}^\times)^2$. }
\]
Since $\fT(x_0\cdot K_v)=J_{n+1}\cdot K_v$, the volume $\int_{x_0\cdot K_v}\rd x_1 \, \rd x_2$ is the multiple of $q_v^{-\dim \AL_{n+1}}$ with 
\begin{align*}
\#\{J_{n+1}\cdot \GL_{n+1}(\fo_{F,v})\mod \varpi_v\fo_{F,v} \}&=\#(\GL_{n+1}(\F_{q_v})/\Sp_{n+1}(\F_{q_v})) \\
&= q_v^{(m+1)(2m+1)}\prod_{j=1}^{m+1}(1-q_v^{-2j+1}) .
\end{align*}
where $n=2m+1$.
Therefore, from \eqref{eq:oddmainlocal} we find
\begin{equation}\label{eq:lc1}
 I^\theta(\cf_{K_v},\chi_v)=\begin{cases} \prod_{j=2}^{m+1}(1-q_v^{-2j+1})  & \text{$\chi_v$ is unramified,} \\ 0& \text{$\chi_v$ is ramified.} \end{cases}
\end{equation}

Take a finite place $v\notin S$ such that $\fp_v\mid \fn_j$ (i.e., $v\in S_j\setminus S$ and $r_{v,j}\ge 1$). 
Then, $\delta_n K_v(\fn_j)=\delta_n +\varpi_v^{r_{v,j}} \M_n(\fo_{F,v})$. 
Let $(x_1,x_2)\in V(F_v)$ and $a\in (\fo_{F,v}^\times\sqcup \varpi_v\fo_{F,v}^\times)/(\fo_{F,v}^\times)^2$. 
Then, $x_1 + a\, \t x_2 x_2\in \delta_n J_n+\varpi_v^{r_{v,j}}M_n(\fo_{F,v})$ if and only if 
\[
\begin{cases} x_1\in \delta_nJ_n-E_{m+1,m+1}+\varpi_v^{r_{v,j}}\AL_n(\fo_{F,v}),  \\
x_2\in  e_{m+1}+\varpi_v^{r_{v,j}} M_{1,n}(\fo_{F,v}) ,\\
a\in (1+\varpi_v^{r_{v,j}} \fo_{F,v})/(\fo_{F,v}^\times)^2, \end{cases}
\]
where $n=2m+1$. 
Hence, for each $v\in S_j\setminus S$, it follows from \eqref{eq:oddmainlocal} that
\begin{equation}\label{eq:lc2}
 I^\theta(\cf_{\delta_n K_v(\fn_j)},\chi_v)=\begin{cases} 0& \text{$v\mid 2$ and $\chi_v$ is non-trivial on $\det(K_v(\fn_j))$,} \\
c_v^2 \, 2^{-2\mathfrak{e}_v-1} q_v^{-r_{v,j}(2m^2+3m+1)}  & \text{otherwise.} 
   \end{cases}
\end{equation}

Finally, inserting \eqref{eq:lc1} and \eqref{eq:lc2} into \eqref{eq:oddmain2}, we obtain  
\[
 \prod_{v\notin S} I^\theta(\cf_{\delta_n K^S(\fn_j)},\chi_v)=\begin{cases} 0 & \text{[$\chi_v$ is ramified for some $v\notin S_j$] or } \\
& \text{[$v\mid 2$ and $\chi_v$ is non-trivial on $\det(K_v(\fn_j))$],} \\
\fm^S(\fn_j) & \text{otherwise.} \end{cases}
\]
This completes the proof. 
\end{proof}

\begin{proposition}\label{prop:mainterm2}
For Case {\it(2)} $E=F$ and $n$ is even, there exist constants $c$, $\tilde{c}>0$ such that   
\[
\sum_{\fo\in \tilde{\cO}} I_{\fo_\main}(h_j)=\tilde{c}\, \tilde\fm^S(\fn_j)\,  \tilde{I}^\theta(h_S) +  c \, \fm^S(\fn_j) \sum_{\omega_S \in \cQ_S} c(\omega_S,j)\, I^\theta(h_S,\omega_S) .
\]
\end{proposition}
\begin{proof}
In this case, $\tilde{\cO}$ consists of a singule class $\fo_1$. 
By $(\fo_1)_\main=\fo_{\main,1}\sqcup\fo_{\main,2}$ (cf. \cref{lem:evenfo}) and \cref{lem:mainexpress}, we have
\[
\sum_{\fo\in \tilde{\cO}} I_{\fo_\main}(h_j)=I_{(\fo_1)_\main}(h_j)=I_{\fo_{\main,1}}(h_j)+I_{\fo_{\main,2}}(h_j),
\]
\[
I_{\fo_{\main,t}}(h_j)\coloneqq \int_{G(F)\bs G(\A)^1} \sum_{\gamma\in \fo_{\main,t}} h_j(g^{-1}\gamma \theta(g)) \, \rd g \quad (t=1,2).
\]

Since $\AL_{n}^0(F)=J_n \cdot \GL_n(F)$ and $\AL_n^0(F)\ni x\mapsto xJ_n^{-1}\in \fo_{\main,1}$ is bijective, we obtain for constants $c'$, $c''>0$ 
\begin{align*}
I_{\fo_{\main,1}}(h_j)=& \int_{\Sp_n(F)\bs \GL_n(\A)^1}  h_j(\t g J_n g J_n^{-1}) \, \rd g\\
=& c' \int_{\Sp_n(\A)\bs \GL_n(\A)} \tilde{h}_j(\t g J_n g J_n^{-1}) \, \rd g = c'' \, \tilde{I}^\theta(h_S) \times \prod_{v\notin S} \tilde{I}^\theta(\cf_{K_v(\fn_j)}).
\end{align*}
Here, in the same way as \eqref{eq:measureAL} (see also \cite{Wei82}*{Ch.2}), $\rd g$ is the Tamagawa measure on $G(\A)$, $\rd x$ is a $F$-gauge form on $\AL_n$, and a $G(\A)$-invariant measure $\rd^\times x$ on $\AL_n(\A)$ is defined by 
\[
\rd^\times x\coloneqq c'''\prod_v  \rd^\times x_v ,\qquad \rd^\times x_v\coloneqq \frac{c_v\, \rd x_v}{|\mathrm{Pf}(x)|_v^{n-1}}
\]
for some constant $c'''>0$. 
When $\fp_v\nmid \fn_j$, by the same argument in the proof of \cref{prop:mainterm1} we have 
\[
\tilde{I}^\theta(\cf_{K_v})=\prod_{j=2}^m(1-q_v^{-2j+1}) 
\]
where $n=2m$. 
When $\fp_v\mid \fn_j$, we easily calculate $\tilde{I}^\theta(\cf_{K_v(\fn_j)})=c_v\, q_v^{r_{v,j}\tilde\fm}$. 
Therefore, we obtain $I_{\fo_{\main,1}}(h_j)=\tilde{c}\, \tilde\fm^S(\fn_j)\, \prod_{v\in S} \tilde{I}^\theta(h_v)$ for some constant $\tilde{c}>0$.

In the same way as in the proof of \cref{prop:mainterm1}, we treat the vector space $V$ and the mapping $H\colon V\to \M_n$. 
The open dense orbit $V^0$ is given by 
\[
V^0= \{ (x_1,x_2)\in V \mid \mathrm{Pf}(x_1)\neq 0 , \;\; x_2\neq 0 \}. 
\]
Take a generic point $x_0=(J_n,e_1)$ in $V^0$, and then $H(x_0)J_n^{-1}=1_n+E_{1,n}$, where $E_{i,j}$ is the unit matrix of the $(i,j)$-component in $\M_n$.  
By \cref{lem:evenfo} we have
\begin{equation}\label{eq:glevendeco}
\fo_{\main,2}=\bigsqcup_{a\in F^\times/(F^\times)^2} \{ a H(x)J_n^{-1} \mid x\in V^0(F)\} .
\end{equation}
Denote by $G_{x_0}$ the stabilizer of $x_0$ in $G$. 
Let $Q$ denote the Jacobson-Morozow parabolic subgroup of the minimal unipotent element $\cH(x_0)J_n^{-1}$ in $\Sp_n$. 
We write $M_Q$ for a Levi subgroup of $Q$, which contains diagonal matrices in $\Sp_n$, and $N_Q$ the unipotent radical of $Q$.  
Then, there exists a subgroup $M'_Q$ in $M_Q$ such that $M'_Q$ is isomorphic to $\Sp_{n-2}$ and we have $G_{x_0}=M_Q'N_Q$. 
Note that $G_{x_0}$ does not equal the $\theta$-centralizer of $H(x_0)J_n^{-1}$ (which is not connected). 
Since $G_{x_0}$ is unimodular, we obtain a $G(\A)$-invariant measure $\rd^\times x$ on $V(\A)$ as
\[
\rd^\times x\coloneqq c'\prod_v  \rd^\times x_v ,\qquad \rd^\times x_v\coloneqq \frac{c_v\, \rd x_{1,v}}{|\mathrm{Pf}(x_{1,v})|_v^{n}} \, \rd x_{2,v} , 
\]
where $c'>0$ is a constant, and $\rd x_1\, \rd x_2$ is a $F$-gauge form on $V$. 
Using \eqref{eq:glevendeco}, $V^0(F)=x_0\cdot G(F)$, \eqref{eq:oddmain1}, and \eqref{eq:LL}, for some contant $c''>0$ we obtain the following in the same way as the proof of \eqref{eq:oddmain2}: 
\begin{equation}\label{eq:evenmain2}
I_{\fo_{\main,2}}(h_j)=c''\sum_{\chi=\otimes_v \chi_v\in\cQ}   I^\theta(h_S,\chi_S) \times  \prod_{v\notin S} I^\theta(\cf_{K^S(\fn_j)},\chi_v),
\end{equation}
where $\chi_S\coloneqq\otimes_{v\in S}\chi_v$. 
For each $v\notin S$, by a change of variable, 
\begin{multline*}
I^\theta(\cf_{\delta_n K_v(\fn_j)},\chi_v)=\int_{F_v^\times} \rd^\times a\,  \int_{\AL_n(F_v)}  \frac{\rd x_1}{|\mathrm{Pf}(x_1)|^{n-1}} \, \int_{F_v^\times\bs M_{1,n}(F_v)} \rd x_2 \\ \cf_{K_v(\fn_j)}((x_1 + a\, \t x_2 x_2)J_n^{-1})\, \chi_v(a) \, |a|_v^m
\end{multline*}
where $n=2m$, and we normalize the measures as 
\[
\int_{\AL_n(\fo_v)}\rd x_1\times\int_{\mathfrak{O}_v}\rd x_2=c_v, \quad \int_{\fo_v^\times}\rd^\times a=1,
\]
$\mathfrak{O}_v\coloneqq\fo_v^\times\bs \{(x_{2,1},\dots,x_{2,n})\in M_{1,n}(\fo)\mid x_{2,j}\in \fo^\times$ for some $j\}$.

Take a finite place $v$ such that $\fp_v\nmid \fn_j$, that is, $K_v(\fn_j)=K_v$ and $v\notin S_j$. 
For $\gamma=(x_1+a\t x_2 x_2)J_n^{-1}$, we have $\Pr_1(\gamma)=x_1=\gamma_s J_n$ and $\Pr_2(\gamma)=a\t x_2 x_2$. 
Hence, $\gamma\in K_v$ if and only if $x_1\in J_n\cdot K_v$, $x_2\in \mathfrak{O}_v$, and $a\in\fo_v$. 
Therefore, 
\begin{equation}\label{eq:lc3}
 I^\theta(\cf_{K_v},\chi_v)=\begin{cases} (1-\chi_v(\varpi_v)q_v^{-m})^{-1} \prod_{j=2}^{m}(1-q_v^{-2j+1})  & \text{$\chi_v$ is unramified,} \\ 0& \text{$\chi_v$ is ramified.} \end{cases}
\end{equation}

Next, we take a finite place $v$ such that $\fp_v\mid \fn_j$, that is, $v\in S_j\setminus S$ and $r_{v,j}\ge 1$. 
Then, $(x_1 + a\, \t x_2 x_2)J_n^{-1}\in K_v(\fn_j)$ if and only if $x_1\in J_n+\varpi_v^{r_{v,j}}\AL_n(\fo_v)$, $x_2\in  \mathfrak{O}_v$, and $a\in \varpi_v^{r_{v,j}} \fo_v$. 
Hence,
\begin{equation}\label{eq:lc4}
 I^\theta(\cf_{K_v(\fn_j)},\chi_v)=\begin{cases} c_v q_v^{r_{v,j}\, \fm} \chi_v(\varpi_v)^{r_{v,j}} (1-\chi_v(\varpi_v)q_v^{-m})^{-1}& \text{$\chi_v$ is unramified,} \\
0  & \text{$\chi_v$ is ramified.} 
   \end{cases}
\end{equation}

Applying \eqref{eq:lc3} and \eqref{eq:lc4}, one has
\[
 \prod_{v\notin S} I^\theta(\cf_{K^S(\fn_j)},\chi_v)=\begin{cases} 0 & \text{$\chi_v$ is ramified for some $v\notin S_j$,} \\
\fm^S(\fn_j) \, L^S(m,\chi) \, \prod_{v\in S_j\setminus S}\chi_v(\varpi_v^{r_{v,j}})& \text{otherwise.} \end{cases}
\]
From this and \eqref{eq:evenmain2}, we obtain the assertion for $I_{\fo_{\main,2}}(h_j)$.
\end{proof}

\begin{proposition}\label{prop:mainterm3}
For Case {\it(3)} $E\neq F$, there exists a constant $c>0$ such that
\[
I_\mathrm{main}(h_j)= c \, \fm^S(\fn_j) \left(  I^\theta(h_S,\trep_S) +  I^\theta(h_S,\eta_S) \right) 
\]    
for any large $j\in \N$. 
\end{proposition}
\begin{proof}
For some constant $c'>0$, it follows from \cref{lem:mainexpress} that
\begin{equation}\label{eq:hermmain}
I_\mathrm{main}(h_j)= c' \, \int_{G(F)\bs G(\A)} \sum_{x\in\HE_n^0(F)} \tilde{h}_j((x\cdot g)J_n^{-1})\, \rd g .
\end{equation}
In the above two cases, we had a single generic orbit $V^0(F)=x_0\cdot G(F)$, but in this case there are infinitely many rational orbits in $\HE^0(F)/G(F)$. 
Therefore, we use Saito's formula \cites{Sai99,Sai03} to clarify relations for rational and adelic orbits as well as local and global measures. 
Note that the action of $G$ on $\HE_n$ is not faithful, and the paper \cite{Sai99} does not treat the non-faithful representations, but that is not a big problem, as explained in \cite{Sai03}. 
Hence, we slightly modify the argument in \cite{Sai99}*{type (1), case b) in \S3} according to \cite{Sai03}, and it makes a slight difference if $n$ is odd. 
Note also that Saito's formula is essentially the same as Labesse's stabilization \cite{Lab04} for the twisted trace formula, but Saito's study reveals specific details. 
In other words, Saito's formula is also regarded as a generalization of Kottwitz's stabilization \cite{Kot86}, and in fact, as \eqref{eq:Saito}, the ad\`elic orbits with $F$-points and those without are separated by character values.

Let us briefly explain some data needed for Saito's formula.
For a connected reductive group $H$, we need the following invariants. 
\begin{itemize}
    \item $\tau(H)$ is the Tamagawa number of $H$. 
    \item $A(H)$ is the torsion group of $H$, cf. \cite{Sai99}*{p.594} (see also \cite{Bor98}). 
    \item $\mathrm{ker}^1(H)$ is the cardinality of the kernel of the Hasse map of $H$, cf. \cite{Sai99}*{p.593} (see also \cite{Kot84}).  
    \item $|A(H)|=\tau(H)\, \mathrm{ker}^1(H)$, cf. \cite{Sai99}*{p.603} (see also \cite{Kot84}). 
\end{itemize}
Take an element $x_0\in \HE_n^0(F)$, and we set
\[
G_{x_0}\coloneqq \{ g\in G \mid x_0\cdot g=x_0\}. 
\]
Note that $G_{x_0}$ is a unitary group and connected. 
By a direct calculation, we obtain 
\[
A(G)=1 \quad \text{and} \quad A(G_{x_0})=\Z/2\Z. 
\]
Hence, if $\iota_{x_0,A}^0$ is the canonical map from $A(G_{x_0})$ to $A(G)$, 
then we have $\mathrm{Ker}\, \iota_{x_0,A}^0=A(G_{x_0})=\Z/2\Z$. 
By \cite{Sai99}*{p.596}, the non-trivial character on $\mathrm{Ker}\, \iota_{x_0,A}^0$ induces a function $\chi=\otimes_v \chi_v $ on $\HE_n^0(\A)$ such that 
\begin{itemize}
    \item $\chi_v(x_v)= 1$ if $x_v\in x_0\cdot G(F_v)$, and $\chi_v(x_v)=-1$ otherwise. 
    \item $\chi(x)= 1$ if $x\in \HE_n^0(F)\cdot G(\A)$. 
\end{itemize}
We may take $x_0\in \HE_n^0(F)$ so that $\det(x_0)=1$. 
Then, $\chi$ agrees with $\eta\circ\det$ by the above conditions. 
Therefore, for $x\in \HE_n^0(\A)$,
\begin{equation}\label{eq:Saito}
\eta\circ\det(x)=\begin{cases} 1 & \text{if $x\in  \HE_n^0(F)\cdot G(\A)$}, \\ -1 & \text{if $x\notin \HE_n^0(F)\cdot G(\A)$}.\end{cases}    
\end{equation}
In addition,
\begin{equation}\label{eq:Saito2}
\frac{1+\eta\circ\det(x)}{2}=\begin{cases} 1 & \text{if $x\in \HE_n^0(F)\cdot G(\A)$}, \\ 0 & \text{if $x\notin \HE_n^0(F)\cdot G(\A)$} . \end{cases}    
\end{equation}
It is known that $\tau(G_x)=2$. 
Hence, by $A(H)=\tau(H)\, \mathrm{ker}^1(H)$ we have $\mathrm{ker}^1(G_x)=1$, that is, the Hasse map of $G_x$ is injective. 
Hence, by \cite{Sai99}*{Proposition 1.4}, 
\begin{equation}\label{eq:naturalmap}
\text{the mapping $\HE_n^0(F)/G(F) \ni x \cdot G(F)  \mapsto x\cdot G(\A) \in \HE_n^0(F)\cdot G(\A)/ G(\A)$ is bijective.}
\end{equation}
Choose a complete system $\mathfrak{R}(F)$ of representatives of $\HE_n^0(F)/G(F)$.
For any element $x\in \HE_n^0(F)\cdot G(\A)$, by \eqref{eq:naturalmap}, there uniquely exist $y\in \mathfrak{R}(F)$ and $g\in G_y(\A)\bs G(\A)$ such that $x=y\cdot g$.     
Let $\rd g$ denote the Tamagawa measure on $G(\A)$, and for each $x\in \mathfrak{R}(F)$, let $\rd g_x$ denote the Tamagawa measure on $G_x(\A)$, hence $\int_{G_x(F)\bs G_x(\A)}\rd g_x=\tau(G_x)=2$. 
Let $\rd x$ denote a $F$-gauge form on $\HE_n$ and a $G(\A)$-invariant measure $\rd^\times x$ can be deifned by
\[
\rd^\times x\coloneqq c'\prod_v  \rd^\times x_v ,\qquad \rd^\times x_v\coloneqq \frac{c_v\, \rd x_v}{|\det(x)|_v^n}
\]
where $c'>0$ is a constant and $\rd x_v$ is the Haar measure on $\HE_n(F_v)$ induced from $\rd x$. 
According to \cite{Wei82}*{Ch.2} (see also \cite{Sai99}*{\S2}), there exists a constant $c''>0$ such that $c''\, \rd^\times x$ equals the quotient measure $\rd g/\rd g_x$ for any $x\in\mathfrak{R}(F)$. 
Using \cref{lem:mainexpress}, \eqref{eq:Saito2} and \eqref{eq:naturalmap}, we have for constant $c'''>0$
\begin{align*}
I_{\fo_{\main}}(h_j)=&\, c''' \sum_{x\in\HE_n^0(F)}\int_{G(F)\bs G(\A)}  \tilde{h}_j((x\cdot g)J_n^{-1})\, \rd g = c'''\sum_{x\in\mathfrak{R}(F)}\int_{G_x(F)\bs G(\A)}  \tilde{h}_j((x\cdot g)J_n^{-1})\, \rd g \\
= & c'''\sum_{x\in\mathfrak{R}(F)}\tau(G_x) \int_{G_x(\A)\bs G(\A)}  \tilde{h}_j((x\cdot g)J_n^{-1})\, \frac{\rd g}{\rd g_x} =  2c''c'''\sum_{x\in\mathfrak{R}(F)}\int_{x\cdot G(\A)}  \tilde{h}_j(xJ_n^{-1})\, \rd^\times x \\
=& 2c''c'''\int_{\HE_n^0(F)\cdot G(\A)}  \tilde{h}_j(xJ_n^{-1})\, \rd^\times x =2c''c'''\int_{\HE_n(\A)}  \tilde{h}_j(xJ_n^{-1})\,\frac{1+\eta\circ\det(x)}{2}\, \rd^\times x\\
=& c''c'''\times \left(   I^\theta(h_S,\trep_S) + I^\theta(h_S,\eta_S)   \right) \times \prod_{v\notin S}I^\theta(\cf_{K_v(\fn_j)},\trep_v).
\end{align*}
Here, we note on $I^\theta(\cf_{K_v(\fn_j)},\trep_v)=I^\theta(\cf_{K_v(\fn_j)},\eta_v)$ for any $v\notin S$. 
Therefore, we have only to calculate $\prod_{v\notin S}I^\theta(\cf_{K_v(\fn_j)},\trep_v)$ explicitly.

Take a finite place $v$ such that $\fp_v\nmid \fn_j$, that is, $K_v(\fn_j)=K_v$ and $v\notin S_j$. 
There exists a constant $c_{\nu}>0$ (depending only on the embedding $\nu\colon \fo_E \to \M_2(\fo_F)$) such that, if $q_v$ is greater than $c_\nu$, then we have
\begin{multline*}
I^\theta(\cf_{K_v},\trep_v)= q_v^{-n^2} \times \#\{J_n\cdot G(\fo_{F,v})\mod \varpi_v\fo_{F,v} \} \\
= q_v^{-n^2} \times \#(\GL_n(\F_{q_v})\times\GL_n(\F_{q_v}) /\GL_n(\F_{q_v}))      
\end{multline*}
when $E_v=F_v\times F_v$, and
\[
I^\theta(\cf_{K_v},\trep_v)= q_v^{-n^2} \times \#\{J_n\cdot G(\fo_{F,v})\mod \varpi_v\fo_{F,v} \}=\#(\GL_n(\F_{q_v^2})/\U_n(\F_{q_v}) )
\]
when $E_v$ is an unramified extension over $F_v$. 
Hence we obtain for any $v\notin S$ $(q_v>c_\nu)$
\[
I^\theta(\cf_{K_v},\trep_v)=\prod_{j=2}^n(1-\eta_v(\varpi_v)^{j-1} \, q_v^{-j}) .
\]
Next, take a finite place $v$ such that $\fp_v\mid \fn_j$, i.e., $r_{v,j}\ge 1$ and $v\in S_j\setminus S$. 
In this case, we can easily compute $I^\theta(\cf_{K_v(\fn_j)},\trep_v)= c_v\, q_v^{r_{v,j}\fm}$. 
Thus, there exists a constant $c''''>0$ such that we obtain $\prod_{v\notin S}I^\theta(\cf_{K_v(\fn_j)},\trep_v)=c''''\times \fm^S(\fn_j)$ for any large $j$. This completes the proof. 
\end{proof}

\newpage

\part{Fourier transforms of twisted orbital integrals}\label{part:2}

In this part, we discuss the Fourier transforms of the twisted orbit integrals $I^\theta(h_v,\omega_v)$ and $\tilde{I}^\theta(h_v)$, which appeared in the asymptotic formula of \cref{thm:asym}. 
In the case $E\neq F$, the Fourier transform of $I^\theta(h_v,\trep_v)$ was studied in \cite{HII08} for supercuspidal representations by using intertwining operators and applied to the study of the formal degree conjecture. 
After that, the Fourier transforms of $I^\theta(h_v,\trep_v)$ in any cases were established in \cites{BP21b,BP21a} for all tempered representations. 
We will give the Fourier transforms of $I^\theta(h_v,\omega_v)$ and $\tilde{I}^\theta(h_v)$ for all the cases by using their methods.

\section{Setup}

In the following, we consider the four cases (I), (I\hspace{-.1em}I), (I\hspace{-.1em}I\hspace{-.1em}I), (I\hspace{-.1em}V). 
As will be discussed below, this division of cases is numbered for each intertwining operator to be introduced, and these numbers do not directly correspond to Cases (i), (ii), (iii), (iv) in \S\ref{sec:intro1} or Cases (1), (2), (3) in \cref{thm:asym}. 
The main results of (I), (I\hspace{-.1em}I), (I\hspace{-.1em}I\hspace{-.1em}I), (I\hspace{-.1em}V) are given in Corollaries \ref{cor:measureGL}, \ref{cor:measureU}, \ref{cor:measureSOodd}, \ref{cor:measureSpSOeven} respectively, and the relation between them and (i), (ii), (iii), (iv) and (1), (2), (3) is explained in the proof of \cref{thm:maintheorem1}.

\subsection{Notation}
Let $F$ be a local field of characteristic $0$, and let $|\; |$ denote the normalized valuation of $F$. 
Let $\trep$ denote the trivial character on $F^\times$.
When $F$ is non-archimedean, let $\fo_F$ be the ring of integers. 
Take a non-trivial additive character $\psi$ on $F$.  

Let $\cG$ be a connected reductive group over $F$. 
Denote by ${}^L\mathcal{G}$ the $L$-group of $\cG$. 
We write by $\Irr(\cG(F))$ the set of isomorphism classes of irreducible admissible representations of $\cG(F)$ and $\Temp(\cG(F))\subset\Irr(\cG(F))$ (resp. $\Pi_2(\cG(F))\subset\Irr(\cG(F))$) the subset of tempered (resp. square-integrable) representations.

\subsection{Case (I)}
Set $G^\sharp=\GL_{2n}$ and let $P^\sharp=M^\sharp N^\sharp$ be a maximal parabolic subgroup given by
    \[
    M^\sharp=\left\{
        \begin{pmatrix}
        g_1 & 0 \\
        0 & g_2
        \end{pmatrix}
    \,\middle|\, g_1, \, g_2\in \GL_n \right\},   \quad N^\sharp=\left\{
        \begin{pmatrix}
        1_n & x \\
        0 & 1_n
        \end{pmatrix}
    \,\middle|\, x\in X \right\}.
    \]
where $X=\M_n$. 
Let $G=\GL_n\times\GL_n$.
The group $G(F)$ acts on $X(F)$ by $x\cdot (g_1,g_2)=g_1^{-1}x g_2$ for $(g_1,g_2)\in G(F)$ and $x\in X(F)$. 
Set $X'=X\cap \GL_n$.

Let $Z_G$ denote the center of $G$, that is, $Z_G(F)\simeq F^\times\times F^\times$.
Set 
\[
Z\coloneqq \{ (a 1_n, a^{-1} 1_n)\in Z_G \mid a\in \mathbb{G}_m\}.
\]
We consider the space $\Cc(G(F)/Z(F))$ of smooth functions $f$ on $G(F)$ satisfying $f(zg)= f(g)$ $(z\in Z(F)$, $g\in G(F))$ and the support is compact modulo $Z(F)$. 
We consider the right translation action of $G(F)$ on $\Cc(G(F)/Z(F))$. 
Set 
    \[
    I(s)=\Ind_{P^\sharp}^{G^\sharp}(\Cc(G(F)/Z(F))\otimes(|\det|^{\frac{s}{2}}\otimes|\det|^{-\frac{s}{2}})).
    \]
Let $W^{G^\sharp}(M^\sharp)=N_{G^\sharp}(M^\sharp)/M^\sharp$ be the relative Weyl group.
The unique non-trivial element $w\in W^{G^\sharp}(M^\sharp)$ is given by
    \[
    w=
        \begin{pmatrix}
        0 & 1_n \\
        -1_n & 0
        \end{pmatrix}.
    \]
The standard intertwining operator $M(s)\,\colon I(s)\to I(w(s))$ is defined as a meromorphic continuation of
    \[
    M(s)\varphi(g)=\int_{N^\sharp(F)}\varphi(w^{-1}ug) \rd u,  \qquad \varphi\in I(s),  \ g\in G^\sharp(F).
    \]    

Let $\bar{N}^\sharp$ be the unipotent radical of the opposite of $P^\sharp$.
We will identify $\bar{N}^\sharp(F)$ with $X(F)$. 
When $F$ is non-archimedean,  let $\mathbf{1}_L\in\Cc(\bar{N}^\sharp(F))$ denote the characteristic function of $L=X(\fo_F)$.
When $F$ is archimedean, we take a function $\phi\in\Cc(\R)$ such that $\phi(0)>0$ and set $\mathbf{1}_L(x)\coloneqq\phi(\|x\|)$, where $\|\cdot\|$ is the Euclidean norm on $X(F)=\bar{N}^\sharp(F)$.

For $f\in\Cc(G(F)/Z(F))$,  we define a section $\varphi_{s,  f}\in I(s)$ so that it is supported on $P^\sharp(F)\bar{N}^\sharp(F)$ and satisfies for $g\in P^\sharp(F)\bar{N}^\sharp(F)$ and $m'\in M^\sharp(F)$,
    \[
    (\varphi_{s,  f}(g))(m')=(\delta_{P^\sharp}^\frac12 (|\det|^{\frac{s}{2}}\otimes|\det|^{-\frac{s}{2}}))(m)\mathbf{1}_L(\bar{u})f(m'm),     
    \]
where $g=um\bar{u}$ with $u\in N^\sharp(F)$,  $m\in M^\sharp(F)$ and $\bar{u}\in\bar{N}^\sharp(F)$.

When $F$ is non-archimedean, take a compact open subgroup $K_{M^\sharp}$ of $M^\sharp(F)$ so that $f$ is left and right invariant under $K_{M^\sharp}$.
We may suppose that $L$ is stable by $K_{M^\sharp}$-conjugation.
Then $(\varphi_{s,  f}(g))(1)$ is bi-$K_{M^\sharp}$-invariant.
When $F$ is archimedean, we fix a finite set $\Gamma$ of $K$-types.
We assume that the $K$-types of the representation of $M^\sharp(F)$ generated by $f$ under the right and left translation are all contained in $\Gamma$.

For a smooth representation $\pi$ of $M^\sharp(F)$ which is trivial on $Z(F)$ and $v\in\pi$, we define a section $\varphi_{s, v}\in  I_\pi(s)\coloneqq\Ind_{P^\sharp}^{G^\sharp}(\pi\otimes (|\det|^{\frac{s}{2}}\otimes|\det|^{-\frac{s}{2}}))$ similarly as above: it is supported on $P^\sharp(F)\bar{N}^\sharp(F)$ and satisfies  for $g\in P^\sharp(F)\bar{N}^\sharp(F)$,
    \[
    \varphi_{s,  v}(g)=(\delta_{P^\sharp}^\frac12 (|\det|^{\frac{s}{2}}\otimes|\det|^{-\frac{s}{2}}))(m)\mathbf{1}_L(\bar{u})\pi(m)v, 
    \]
where $g=um\bar{u}$ with $u\in N^\sharp(F)$,  $m\in M^\sharp(F)$ and $\bar{u}\in\bar{N}^\sharp(F)$.
It is easy to check that for all $g\in G^\sharp(F)$,
    \[
    \pi(\varphi_{s,  f}^\vee(g))v=\varphi_{s, \pi(f^\vee)v}(g),
    \]    
where $\varphi_{s,  f}^\vee(g)\coloneqq \varphi_{s,  f}(g^{-1})$ and $f^\vee(g)\coloneqq f(g^{-1})$.

Write $\Ad\,\colon {}^L\GL_n\to\GL_{n^2}(\C)$ for the adjoint representation. 
For each $\pi_0\in\Temp(\GL_n(F))$, applying Shahidi's theorem \cite{Sha90}*{Theorem 3.5} to $P^\sharp=M^\sharp N^\sharp$ and $I_{\pi_0\otimes\pi_0}(s,\trep)$, we obtain the $\gamma$-factor $\gamma(s,\pi_0,\Ad,\psi)$ $(s\in\C)$.

\subsection{Case (I\hspace{-.1em}I)}
Let $E/F$ be a quadratic extension of $F$, and let $\iota$ denote the Galois conjugation. 
The group $G^\sharp$ is the quasi-split unitary group
    \[
    G^\sharp=\U(n,  n)_{E/F}=\{g\in\Res_{E/F}\GL_{2n} \mid g J_{2n}\iota(\t g)=J_{2n}\},  \quad J_{2n}=
        \begin{pmatrix}
        0 & J_n \\
        (-1)^n J_n & 0
        \end{pmatrix},
    \]
where $J_n$ was defined in \eqref{eq:J_n}, and $P^\sharp=M^\sharp N^\sharp$ is a maximal parabolic subgroup given by
    \[
    M^\sharp=\left\{
        \begin{pmatrix}
        a & 0 \\
        0 & \theta(a)
        \end{pmatrix}
    \,\middle|\, a\in \Res_{E/F}\GL_n \right\},   \quad N^\sharp=\left\{
        \begin{pmatrix}
        1_n & x \\
        0 & 1_n
        \end{pmatrix}
    \,\middle|\, x\in X \right\},
    \]
where $\theta(a)=J_n\iota(\t a^{-1})J_n^{-1}$ and $X=\{x\in\Res_{\fo_E/\fo_F}\M_n \mid   x J_n= \iota(\t(xJ_n))\}$.
Let $G=\GL_n$.
The group $G(E)=\GL_n(E)\simeq M^\sharp(F)$ acts on $X(F)$ by $x\cdot g=g^{-1}x\theta(g)$ for $g\in G(E)$ and $x\in X(F)$.
Set $X'=X\cap \Res_{E/F}\GL_n$.

Set $N_{E/F}(a)\coloneqq a \, \iota(a)$ $(a\in E)$. 
Take a unitary character $\chi$ of $E^\times$ which is trivial on $N_{E/F}(E^\times)$.
By abuse of notation, we write the character $\chi\circ\det$ on $G(E)$ by $\chi$. 
The group $N_{E/F}(E^\times)$ is regarded as the subgroup of the center of $G(E)$. 
We consider the space $\Cc(G(F)/N_{E/F}(E^\times))$ of smooth functions $f$ on $G(E)$ satisfying $f(zg)=f(g)$, $z\in N_{E/F}(E^\times)$, $g\in G(E)$ and the support is compact modulo $N_{E/F}(E^\times)$.
We consider the right translation action of $G(E)$ on $\Cc(G(E)/N_{E/F}(E^\times))$.
Set 
    \[
    I(s, \chi)=\Ind_{P^\sharp}^{G^\sharp}(\Cc(G(E)/N_{E/F}(E^\times))\otimes\chi|\det|_E^{\frac{s}{2}}),
    \]
where $|a|_E\coloneqq |N_{E/F}(a)|$ $(a\in E)$. 
Let $W^{G^\sharp}(M^\sharp)=N_{G^\sharp}(M^\sharp)/M^\sharp$ be the relative Weyl group.
The unique non-trivial element $w\in W^{G^\sharp}(M^\sharp)$ is given by
    \[
    w=
        \begin{pmatrix}
        0 & 1_n \\
        -1_n & 0
        \end{pmatrix}.
    \]
The standard intertwining operator $M(s,\chi)\,\colon I(s, \chi)\to I(w(s),  {}^w\chi)$ is defined as a meromorphic continuation of
    \[
    M(s,  \chi)\varphi(g)=\int_{N^\sharp(F)}\varphi(w^{-1}ug) \rd u,  \qquad \varphi\in I(s,  \chi),  \ g\in G^\sharp(F).
    \]

When $F$ is non-archimedean,  let $\mathbf{1}_L\in\Cc(\bar{N}^\sharp(F))$ denote the characteristic function of $L=X(\fo_F)$.
When $F$ is archimedean, we take a function $\phi\in\Cc(\R)$ such that $\phi(0)>0$ and set $\mathbf{1}_L(x)\coloneqq\phi(\|x\|)$, where $\|\cdot\|$ is the Euclidean norm on $X(F)=\bar{N}^\sharp(F)$.

For $f\in\Cc(G(E)/N_{E/F}(E^\times))$, we define a section $\varphi_{s,  f}\in I(s, \chi)$ so that it is supported on $P^\sharp(F)\bar{N}^\sharp(F)$ and satisfies for $g\in P^\sharp(F)\bar{N}^\sharp(F)$ and $m'\in M^\sharp(F)$,
    \[
    (\varphi_{s,  f}(g))(m')=(\delta_{P^\sharp}^\frac12 \chi|\det|_E^{\frac{s}{2}})(m)\mathbf{1}_L(\bar{u})f(m'm),     
    \]
where $g=um\bar{u}$ with $u\in N^\sharp(F)$,  $m\in M^\sharp(F)$ and $\bar{u}\in\bar{N}^\sharp(F)$.

When $F$ is non-archimedean, take a compact open subgroup $K_{M^\sharp}$ of $M^\sharp(F)$ so that $f$ is left and right invariant under $K_{M^\sharp}$.
We may suppose that $L$ is stable by $K_{M^\sharp}$-conjugation.
Then $(\varphi_{s,  f}(g))(1)$ is bi-$K_{M^\sharp}$-invariant.
When $F$ is archimedean, we fix a finite set $\Gamma$ of $K$-types.
We assume that the $K$-types of the representation of $M^\sharp(F)$ generated by $f$ under the right and left translation are all contained in $\Gamma$.

For a smooth representation $\pi$ of $M^\sharp(F)$ and $v\in\pi$, we define a section $\varphi_{s, v}\in  I_\pi(s, \chi)\coloneqq \Ind_{P^\sharp}^{G^\sharp}(\pi\otimes\chi|\det|_E^{\frac{s}{2}})$ similarly as above: it is supported on $P^\sharp(F)\bar{N}^\sharp(F)$ and satisfies  for $g\in P^\sharp(F)\bar{N}^\sharp(F)$,
    \[
    \varphi_{s,  v}(g)=(\delta_{P^\sharp}^\frac12\chi|\det|_E^{\frac{s}{2}})(m)\mathbf{1}_L(\bar{u})\pi(m)v, 
    \]
where $g=um\bar{u}$ with $u\in N^\sharp(F)$,  $m\in M^\sharp(F)$ and $\bar{u}\in\bar{N}^\sharp(F)$.
It is easy to check that for all $g\in G^\sharp(F)$,
    \[
    \pi(\varphi_{s,  f}(g)^\vee)v=\varphi_{s, \pi(f^\vee)v}(g).
    \]

We denote by $\As^+$ and $\As^- \,\colon {}^L(\Res_{E/F} \GL_n)\to\GL_{n^2}(\C)$ the Asai representations (cf. \cite{Mok15}*{Ch. 2}). 
Let $\pi\in\Temp(\GL_n(E))$. 
We apply Shahidi's theorem \cite{Sha90}*{Theorem 3.5} to $P^\sharp=M^\sharp N^\sharp$ and $I_{\pi}(s,\chi)$.
Then, we obtain the $\gamma$-factor $\gamma(s,\pi,\As^+,\psi)$ when $\chi|_{F^\times}$ is trivial, and the $\gamma$-factor $\gamma(s,\pi,\As^-,\psi)$ when $\chi|_{F^\times}$ is non-trivial.

\subsection{Case (I\hspace{-.1em}I\hspace{-.1em}I)}
The group $G^\sharp$ is the even split special orthogonal group
    \[
    G^\sharp=\SO(2n,  2n)=\{g\in\SL_{4n} \mid gQ_{2n}\t g=Q_{2n}\}, \quad Q_{2n}\coloneqq \begin{pmatrix}
      0  & J_n \\  \t J_n & 0 
    \end{pmatrix}
    \]
and $P^\sharp=M^\sharp N^\sharp$ is given by
    \[
    M^\sharp=\left\{
        \begin{pmatrix}
        a & 0 \\
        0 & \theta(a)
        \end{pmatrix}
    \,\middle|\, a\in\GL_{2n} \right\},   \quad N^\sharp=\left\{
        \begin{pmatrix}
        1_n & x \\
        0 & 1_n
        \end{pmatrix}
    \,\middle|\, x\in X \right\},
    \]
where $\theta(a)=J_{2n}\t a^{-1}J_{2n}^{-1}$ and $X=\{x\in\M_{2n} \mid J_{2n}\t x=xJ_{2n}\}$.
The group $G(F)=\GL_{2n}(F)\simeq M^\sharp(F)$ acts on $X(F)$ by $x\cdot g=g^{-1}x\theta(g)$ for $g\in G(F)$ and $x\in X(F)$.
Set $X'=X\cap G$.

Similarly as Case (I), we define $\Cc(G(F)/Z_G(F))$, where $Z_G$ is the center of $G$, and $I(s)=\Ind_{P^\sharp}^{G^\sharp}(\Cc(G(F)/Z_G(F))\otimes|\det|^{\frac{s}{2}})$. 
The unique non-trivial element $w\in W^{G^\sharp}(M^\sharp)$ is given by
    \[
    w=
        \begin{pmatrix}
        0 & 1_{2n} \\
        -1_{2n} & 0
        \end{pmatrix}.
    \]
The standard intertwining operator $M(s)\,\colon I(s)\to I(w(s))$ is defined as in the other cases.

Take an open subset $L$ of $\bar{N}^\sharp(F)$.
For $f\in\Cc(G(F)/Z_G(F))$, we define a section $\varphi_{s, f}\in I(s)$ in the same way as Case (I).
For a smooth representation $\pi$ of $M^\sharp(F)$ and $v\in\pi$, we also define a section $\varphi_{s, v}\in  I_\pi(s)\coloneqq \Ind_{P^\sharp}^{G^\sharp}(\pi\otimes|\det|^{\frac{s}{2}})$.
When $F$ is non-archimedean,  let $\mathbf{1}_L\in\Cc(\bar{N}^\sharp(F))$ denote the characteristic function of $L=X(\fo_F)$.
When $F$ is archimedean, we take a function $\phi\in\Cc(\R)$ such that $\phi(0)>0$ and set $\mathbf{1}_L(x)\coloneqq\phi(\|x\|)$, where $\|\cdot\|$ is the Euclidean norm on $X(F)=\bar{N}^\sharp(F)$.

Write $\wedge^2\,\colon{}^L\GL_{2n}\to\GL_{n(2n-1)}(\C)$ for the exterior square representation. 
Take a tempered representation $\pi\in\Temp(\GL_{2n}(F))$.  
Applying Shahidi's theorem \cite{Sha90}*{Theorem 3.5} to $P^\sharp=M^\sharp N^\sharp$ and $I_\pi(s)$, we obtain the $\gamma$-factor $\gamma(s,\pi,\wedge^2,\psi)$.

\subsection{Case (I\hspace{-.1em}V)}
The group $G^\sharp$ is the odd split special orthogonal group
    \[
    G^\sharp=\SO(n+1,  n)=\{g\in\SL_{2n+1} \mid gR_n \t g=R_n\},  \quad R_n=
        \begin{pmatrix}
        && A_n \\
        &1& \\
        \t A_n &&
        \end{pmatrix},
    \]
where we define $w_n\in\GL_n(F)$ recursively as
    \[
    w_2=
        \begin{pmatrix}
        0 & 1 \\
        1 & 0
        \end{pmatrix},  \quad 
    w_{n+1}=
        \begin{pmatrix}
        0 & w_n \\
        1 & 0
        \end{pmatrix},
    \]
    and 
    \[
    A_n=
        \begin{pmatrix}
        && w_{(n-1)/2} \\
        &1& \\
        -w_{(n-1)/2} &&
        \end{pmatrix} \;\; \text{when $n$ is odd}, \quad 
    A_n=
        \begin{pmatrix}
        & w_{n/2} \\
        -w_{n/2} &
        \end{pmatrix} \;\; \text{when $n$ is even.}
    \]
Note that $\t A_n=A_n^{-1}$.
The maximal parabolic subgroup $P^\sharp=M^\sharp N^\sharp$ is given by
    \[
    M^\sharp=\left\{
        \begin{pmatrix}
        a & 0 & 0 \\
        0 & 1 & 0 \\ 
        0 & 0 & \theta(a)
        \end{pmatrix}
    \,\middle|\, a\in\GL_n \right\},   \quad N^\sharp=\left\{
        \begin{pmatrix}
        1_n & y & H(x,  y) \\
        0 & 1 & -\t yA_n \\
        0 & 0 & 1_n
        \end{pmatrix}
    \,\middle|\, (x,  y)\in X \right\},
    \]
where $\theta(a)=A_n^{-1}\t a^{-1}A_n$,  $X=\{(x,  y)\in\M_n\times\M_{n\times 1} \mid x=-\t x \}$ and $H(x,  y)=(x-\frac12 y\t y)A_n$.
The group $G(F)=\GL_n(F)\simeq M^\sharp(F)$ acts on $X(F)$ by $(x,  y)\cdot g=(g^{-1}x\t g^{-1},  g^{-1}y)$ for $g\in G(F)$ and $(x,  y)\in X(F)$.
Note that $H((x,  y)\cdot g)=g^{-1}H(x,  y)\theta(g)$.

The unique non-trivial element $w\in W^{G^\sharp}(M^\sharp)$ is given by $w=-R_n$ when $n$ is odd and
    \[
    w=
        \begin{pmatrix}
        0 & 0 & -A_n \\
        0 & 1 & 0 \\
        -\t A_n & 0 & 0
        \end{pmatrix}
    \]
when $n$ is even.

Take a quadratic character $\chi$ of $F^\times$, let $Z_G$ denote the center of $G$, and we regard $\chi$ as a character of $Z_G(F)\simeq F^\times$. 
We consider the space $\Cc(G(F)/Z_G(F))_\chi$ of smooth functions $f$ on $G(F)$ satisfying
    \[
    f(zg)=\chi(z)f(g), \qquad z\in Z_G(F), \ g\in G(F)
    \]
and the support is compact modulo $Z_G(F)$.
Set
    \[
    I_\chi(s)=\Ind_{P^\sharp}^{G^\sharp}(\Cc(G(F)/Z_G(F))_\chi\otimes|\det|^{\frac{s}{2}}),
    \]
where the group $G(F)$ acts on $\Cc(G(F)/Z_G(F))_\chi$ by right translation.
The standard intertwining operator $M_\chi(s)\,\colon I_\chi(s)\to I_\chi (w(s))$ is defined as a meromorphic continuation of
    \[
    M_\chi(s)\varphi(g)=\int_{N^\sharp(F)}\varphi(w^{-1}ug) \rd u,  \qquad \varphi\in I_\chi(s),  \ g\in G^\sharp(F).
    \]    

Take an open subset $L$ of $\bar{N}^\sharp(F)$.
For $f\in\Cc(G(F)/Z_G(F))_\chi$, we define a section $\varphi_{s, f}\in I_\chi(s)$ in the same way as Case (I).
For a smooth representation $\pi$ of $M^\sharp(F)$ and $v\in\pi$, we also define a section $\varphi_{s,  v}\in  I_\pi(s)\coloneqq\Ind_{P^\sharp}^{G^\sharp}(\pi\otimes|\det|^{\frac{s}{2}})$.
When $F$ is non-archimedean,  let $\mathbf{1}_L\in\Cc(\bar{N}^\sharp(F))$ denote the characteristic function of $L=X(\fo_F)$.
When $F$ is archimedean, we take a function $\phi\in\Cc(\R)$ such that $\phi(0)>0$ and set $\mathbf{1}_L(x)\coloneqq\phi(\|x\|)$, where $\|\cdot\|$ is the Euclidean norm on $X(F)=\bar{N}^\sharp(F)$. 
In both cases, $\mathbf{1}_L(x, y)$ depends only on $H(x, y)$ and we often write $\mathbf{1}_L(H(x, y))$ for $\mathbf{1}_L(x, y)$.

Set $X'=\{(x, y)\in X \mid \det(H(x, y))\neq 0\}$ when $n$ is odd and $X'=\{(x, y)\in X \mid \det(H(x, y))\neq 0, y\neq0\}$ when $n$ is even.
Then $X'$ is Zariski open dense in $X$.
It is known that $(G, X)$ is a prehomogeneous vector space and $\det(H(x,y))$ is the square of its relative invariant.
In particular, we have
    \[
    \det(H(x, y))=\det
        \begin{pmatrix}
        x & y \\
        -\t y & 0
        \end{pmatrix}
    \]
when $n$ is even and $\det(H(x, y))=\det(x)$ when $n$ is odd.
See \cite{HKK88} for details.
Note that $\t H(x,  y)=\t A_n(-x-\frac12y\t y)=\t A_nH(-x,  y)A_n^{-1}$.
In particular,  $\det(H(x,  y))=\det(H(-x,  y))$.

Denote the symmetric square representation by $\Sym^2\,\colon{}^L\GL_n\to\GL_{n(n+1)/2}(\C)$. 
We apply Shahidi's theorem \cite{Sha90}*{Theorem 3.5} to $P^\sharp=M^\sharp N^\sharp$ and $I_{\pi}(s)$ for each $\pi\in\Temp(G(F))$.
Then, we obtain the $\gamma$-factor $\gamma(s,\pi,\Sym^2,\psi)$.

\subsection{Goal}
The goal is to compute the limits
\begin{align*}
    & \lim_{s\to0+} s\cdot (M(s)\varphi_{s,  f}(1))(1)  \qquad \text{Cases (I) and (I\hspace{-.1em}I\hspace{-.1em}I)}, \\
    & \lim_{s\to0+} s\cdot (M(s,\chi)\varphi_{s,  f}(1))(1)  \qquad \text{Case (I\hspace{-.1em}I)}, \\
    & \lim_{s\to0+} s\cdot (M_\chi(s)\varphi_{s,  f}(1))(1)  \qquad \text{Case (I\hspace{-.1em}V)}
\end{align*}
in two ways: the geometric and the spectral expansion.
Here, $f$ is a test function on $G(F)$. 
In the geometric expansion, their limits coincide with the twisted orbit integrals $I^\theta$ up to constant multiples, see
\begin{itemize}
    \item \cref{geom_gl} and \S\ref{sec:con1} for Case (I).
    \item \cref{geom_unitary} and \S\ref{sec:con2} for Case (I\hspace{-.1em}I).
    \item \cref{geom_evenorth} and \S\ref{sec:con3} for Case (I\hspace{-.1em}I\hspace{-.1em}I).
    \item \cref{geom_oddortho} and \S\ref{sec:con4} for Case (I\hspace{-.1em}V).
\end{itemize}
In the spectral expansion, by applying the Harish-Chandra Plancherel formula to $f$, their limits are rewritten as integrals over subsets of $\Temp^\theta(G(F))$, see
\begin{itemize}
    \item \cref{spec_gl} and \S\ref{sec:con1} for Case (I).
    \item \cref{spec_unitary} and \S\ref{sec:con2} for Case (I\hspace{-.1em}I).
    \item \cref{spec_evenortho} and \S\ref{sec:con3} for Case (I\hspace{-.1em}I\hspace{-.1em}I).
    \item \cref{spec_oddortho} and \S\ref{sec:con4} for Case (I\hspace{-.1em}V).
\end{itemize}
By combining these, we will obtain the desired Fourier transforms: Corollaries \ref{cor:measureGL}, \ref{cor:measureU}, \ref{cor:measureSOodd}, \ref{cor:measureSpSOeven}.

\section{Spectral measures}\label{sec:measure}


\subsection{Definition of measures}\label{sec:measure2}

Let $E$ be an extension of $F$ with $[E:F]\leq 2$. 
For each of Cases (I), (I\hspace{-.1em}I), (I\hspace{-.1em}I\hspace{-.1em}I), (I\hspace{-.1em}V), we recall some conditions as follows:
\begin{itemize}
    \item[(I)] $E=F$, $G=\GL_n\times\GL_n$, $r=\mathrm{Std}\otimes\mathrm{Std}^\vee$.
    \item[(I\hspace{-.1em}I)] $E\neq F$, $G=\GL_n$. 
    \begin{itemize}
    \item[(I\hspace{-.1em}I$+$)] $r=\As^+$.  
    \item[(I\hspace{-.1em}I$-$)] $r=\As^-$.  
    \end{itemize}
    \item[(I\hspace{-.1em}I\hspace{-.1em}I)] $E=F$, $G=\GL_{2n}$, $r=\wedge^2$.  
    \item[(I\hspace{-.1em}V)] $E=F$, $G=\GL_n$, $r=\Sym^2$.  
\end{itemize}

First, let us explain a spectral measure on $\Temp(\GL_n(E))$. 
Note that we consider both cases $E=F$ and $E\neq F$. 
Since any element in $\Irr_{\gen}(\GL_n(E))$ is induced from an essentially square-integrable reprensetation of a Levi subgroup, cf. \cref{lem:BZ}, any representations $\pi\in \Temp(\GL_n(E))$ is expressed as the form 
\[
\bigtimes_{j=1}^m \nu_j^{r_j} \qquad \text{where $\nu_j\in\Pi_2(\GL_{d_j}(E))$ and $\nu_j\not\simeq \nu_{j'}$ for all $j\neq j'$}. 
\]
Let $P(E)$ be the standard parabolic subgroup of $\GL_n(E)$ with a standard Levi subgroup $M(E)=\GL_{d_1}(E)\times\GL_{d_2}(E)\times\cdots \times\GL_{d_r}(E)$ and $W(M)=N_{\GL_n(E)}(M)/M(E)$ be the associated Weyl group.
We regard $\nu=\boxtimes_{j=1}^m \nu_j^{r_j}$ as a representation of $M(E)$.
Set $W(\nu)=\{w\in W(M) \mid w\nu\simeq\nu\}$.
For $\lambda\in\sqrt{-1}\fa_M$, let $\nu_\lambda\in\Pi_2(M(E))$ be the twist of $\nu$ by the unramified character of $M(E)$ ocrresponding to $\lambda$ and let $\pi_\lambda$ be the representation of $G(E)$ parabolically induced from $\nu_\lambda$.
Then there is a natural map $\sqrt{-1}\fa_M\to \Temp(\GL_n(E))$ given by $\lambda\mapsto \pi_\lambda$. 
Let $\cV$ be a sufficiently small $W(\nu)$-stable open neighborhood of $0$ in $\sqrt{-1}\fa_M$.
Then the above map induces a topological isomorphism from $\cV/W(\nu)$ to an open neighborhood $\cU$ of $\pi$ in $\Temp(\GL_n(E))$.
There is a unique measure $\rd_{\Temp(\GL_n(E))}(\pi)$ on $\Temp(\GL_n(E))$ satisfying
\begin{equation}\label{eq:meagl}
    \int_\cU \varphi(\pi) \rd_{\Temp(\GL_n(E))}(\pi) = \frac{1}{|W(\nu)|} \int_\cV \varphi(\pi_\lambda) \rd\lambda    
\end{equation}    
for all $\varphi\in C_c(\cU)$,  where $\rd\lambda$ is the Lebesgue measure. 

For $\pi\in \Temp(\GL_n(E))$, we can express $\pi$ as the form of $\pi_1\times\pi_2\times \cdots \times \pi_r$, where $\pi_j\in \Pi_2(\GL_{d_j}(E))$ $(d_1+d_2+\cdots+d_l=n)$. 
Then we define the component group $S_\pi$ by 
\begin{equation}\label{eq:Sgroup}
    S_\pi= S_{\pi_1}\times\cdots\times S_{\pi_r}, \qquad
    S_{\pi_j}=\Z/[E:F]d_j\Z.    
\end{equation}
The cardinality $|S_\pi|$ is necessary to describe the Plancherel measure on $\Temp(\GL_n(E))$.

In Case (I), we set $\EL(G(E),r)\coloneqq\Temp(\GL_n(F))$ and $\rd_{\EL(G(E),r)}(\pi)\coloneqq\rd_{\Temp(\GL_n(F))}(\pi)$.

In what follows, we consider Cases (I\hspace{-.1em}I), (I\hspace{-.1em}I\hspace{-.1em}I), (I\hspace{-.1em}V). 
For $\pi\in \Irr(\GL_n(E))$, we write $\pi^\vee$ for the contragredient representations of $\pi$, and we set
\[
\pi^*\coloneqq \begin{cases}
    \pi^\vee\circ \iota & \text{Case (I\hspace{-.1em}I)}, \\
    \pi^\vee & \text{Cases (I\hspace{-.1em}I\hspace{-.1em}I) and (I\hspace{-.1em}V)}.
\end{cases}
\]
The subset $\EL(G(E),r)$ consists of $\pi\in\Temp(G(E))$ of the form
    \[
    \pi=(\tau_1\times\tau_1^*)\times\cdots\times(\tau_k\times\tau_k^*)
    \times(\sigma_1\times\cdots\times\sigma_l),
    \]
where 
\begin{itemize}
\item $\tau_i\in\Pi_2(\GL_{n_i}(E))$, $\sigma_j\in\Pi_2(\GL_{m_j}(E))$ with some $n_i, m_j\geq 1$. 
\item $2\sum_{i=1}^k n_i+\sum_{j=1}^l m_j=2n$ in Case (I\hspace{-.1em}I\hspace{-.1em}I) and $2\sum_{i=1}^k n_i+\sum_{j=1}^l m_j=n$ in Case (I\hspace{-.1em}I) and Case (I\hspace{-.1em}V).
\item $\gamma(0, \sigma_j, r, \psi)=0$ for all $j$.
\end{itemize}
Note that the $\gamma$-factor condition means that $\sigma_j$ comes from a twisted endoscopic group, see \cref{rem:gamma} when $F$ is non-archimedean and see \cite{Sha85} when $F$ is archimedean. 
In particular, for a square-integrable representation $\pi\in\Pi_2(G(E))$, $\pi\in \EL(G(E),r)$ is equivalent to $\gamma(0, \pi, r,\psi)=0$. 
In Case (I\hspace{-.1em}V), for each quadratic character $\chi$ on $F^\times$, we denote by $\EL(G(F),r)_\chi$ the subset of $\pi\in \EL(G(E),r)$ whose central character equals $\chi$.

We equip $\EL(G(E),r)$ with a measure as follows.
Take $\pi\in \EL(G(E),r)$.
We can write it as 
    \begin{equation}\label{eq:1110}
    \pi=\left(\bigtimes_{i=1}^s (\tau_i\times\tau_i^\ast)^{p_i}\right)\times \bigtimes_{j=1}^t \mu_j^{q_j} \times \bigtimes_{l=1}^u \nu_l^{r_l},
    \end{equation}
where 
\setlength{\leftmargini}{20pt}
\begin{itemize}\renewcommand{\itemsep}{5pt}
\item For $i=1,  \ldots,  s$,  $\tau_i\in\Pi_2(\GL_{d_i}(E))$ such that $\tau_i\not\simeq\tau_i^\ast$ and $\tau_i\not\simeq\tau_{i'}$,  $\tau_i\not\simeq\tau_{i'}^\ast$ for all $i\neq i'$.
\item For $j=1,  \ldots,  t$,  $\mu_j\in\Pi_2(\GL_{e_j}(E))$ such that $\mu_j\simeq\mu_j^\ast$ and $\gamma(0, \mu_j, r,  \psi)\neq0$ and each $q_j$ is even.
Moreover,  $\mu_j\not\simeq\mu_{j'}$ for all $j\neq j'$.
\item For $l=1,  \ldots,  u$,  $\nu_l\in\Pi_2(\GL_{f_l}(E))$ such that $\nu_l\simeq\nu_l^\ast$ and $\gamma(0, \nu_l, r,  \psi)=0$.
Moreover,  $\nu_l\not\simeq\nu_{l'}$ for all $l\neq l'$.
\end{itemize}
Let $P(E)$ be the standard parabolic subgroup of $G(E)$ with a standard Levi subgroup $M(E)=\GL_{d_1}(E)^2\times\cdots \times\GL_{d_s}(E)^2\times \GL_{e_1}(E)\times\cdots \times\GL_{e_t}(E)\times \GL_{f_1}(E)\times\cdots \times \GL_{f_u}(E)$ and $W(G, M)=N_{G(E)}(M)/M(E)$ be the associated Weyl group.
We regard $\tau=\left(\boxtimes_{i=1}^s (\tau_i\boxtimes\tau_i^\ast)^{p_i}\right)\boxtimes\left(\boxtimes_{j=1}^t \mu_j^{q_j}\right) \boxtimes \left(\boxtimes_{l=1}^u \nu_l^{r_l}\right)$ as a representation of $M(E)$.
Set $W(\tau)=\{w\in W(G, M) \mid w\tau\simeq\tau\}$.
For $\lambda\in\sqrt{-1}\fa_M$, let $\tau_\lambda\in\Pi_2(M(E))$ be the twist of $\tau$ by the unramified character of $M(E)$ ocrresponding to $\lambda$ and let $\pi_\lambda$ be the representation of $G(E)$ parabolically induced from $\tau_\lambda$.

Let $\fa_M'$ be the subspace of $\fa_M$ consisting of 
    \[
    (x_{1,  1},  x^\ast_{1,  1},  \ldots,  x_{s,  p_s},  x^\ast_{s,  p_s},  y_{1,  1},  \cdots,  y_{t,  q_t},  z_{1,  1},  \ldots,  z_{u,  r_u})\in\fa_M=\prod_{i=1}^s(\R^2)^{p_i}\times \prod_{j=1}^t\R^{q_j} \times \prod_{l=1}^u\R^{r_l}
    \]
such that
\begin{itemize}\renewcommand{\itemsep}{5pt}
\item For $1\leq i\leq s$,  $1\leq v\leq p_i$,  $x_{i,  v}+x_{i,  v}^\ast=0$;
\item For $1\leq j\leq t$,  $1\leq v\leq \frac{q_j}{2}$,  $y_{j,  v}+y_{j,  q_j+1-v}=0$;
\item For $1\leq l\leq t$,  $1\leq v\leq [\frac{r_l}{2}]$,  $z_{j,  v}+z_{j,  r_l+1-v}=0$.
\end{itemize}
Then there is a natural map $\sqrt{-1}\fa_M'\to \EL(G(E),r)$ given by $\lambda\mapsto \pi_\lambda$. 
Let $\cV$ be a sufficiently small $W(\tau)$-stable open neighborhood of $0$ in $\sqrt{-1}\fa_M'$.
Then the above map induces a topological isomorphism from $\cV/W(\tau)$ to an open neighborhood $\cU$ of $\pi$ in $\EL(G(E),r)$.
There is a unique measure $\rd_{\EL(G(E),r)}(\pi)$ on $\EL(G(E),r)$ satisfying
    \[
    \int_\cU \varphi(\pi) \rd_{\EL(G(E),r)}(\pi) = \frac{1}{|W(\tau)|} \int_\cV \varphi(\pi_\lambda) \rd\lambda
    \]
for all $\varphi\in C_c(\cU)$,  where $\rd\lambda$ is the Lebesgue measure.

For $\pi\in \EL(G(E),r)$, we define an auxiliary group $\fS_\pi$ as follows. 
We can write $\pi$ in the form of
    \begin{equation}\label{eq:1111}
    \pi=\bigtimes_{i=1}^m (\tau_i\times\tau_i^\ast)\times \bigtimes_{j=1}^k \pi_j,
    \end{equation}
where 
\setlength{\leftmargini}{20pt}
\begin{itemize}\renewcommand{\itemsep}{5pt}
\item For $i=1,  \ldots,  m$,  $\tau_i\in\Pi_2(\GL_{d_i}(E))$.
\item For $j=1,  \ldots,  k$,  $\pi_j\in\Pi_2(\GL_{f_j}(E))$ such that $\gamma(0, \pi_j, r,  \psi)=0$ and $\pi_j\not\simeq\pi_{j'}$ for all $j\neq j'$.
\end{itemize}
In the notation of \eqref{eq:1110}, 
$k$ equals the number of $l$ such that $r_l$ is odd.
Then we set 
    \[
    \fS_\pi=\bigtimes_{i=1}^m S_{\tau_i}\times  (\Z/2\Z)^k.
    \]

\subsection{Local Langlands correspondence and endoscopic lift}\label{sec:Langlands}

We briefly recall the local Langlands correspondence and the endoscopic lift for classical groups. 
A part of them are conjecture. 
The local Langlands correspondence is used only for the proof of Corollaries \ref{cor:llc2}, \ref{cor:llc3}, \ref{cor:llc4}. 
They clarify the meaning of $\EL(G(E),r)$ and the measures for Cases (I\hspace{-.1em}I), (I\hspace{-.1em}I\hspace{-.1em}I), (I\hspace{-.1em}V). 
However, they are not used in the proofs of the main theorem (\cref{thm:maintheorem1}) and its applications (Theorems \ref{thm:globalization} and \ref{thm:density}), because Corollaries \ref{cor:measureU}, \ref{cor:measureSOodd}, \ref{cor:measureSpSOeven} are sufficient for the definition of measures and the proofs. 
As for Case (I), the measure is the Plancherel measure on $\Temp(\GL_n(F))$, hence it is unnecessary to give any interpretation on it.

\subsubsection{Local Langlands correspondence}

Let $W_F$ be the Weil group of $F$ and
    \[
    WD_F=
        \begin{cases}
        W_F\times\SL_2(\C) & \text{$F$ is non-archimedean}, \\
        W_F & \text{$F$ is archimedean}
        \end{cases}
    \]
the Weil-Deligne group of $F$.
Let ${}^LG=\hat{G}(\C)\rtimes W_F$ be the $L$-group of $G$, where $\hat{G}(\C)$ is the Langlands dual group of $G$.
When $E/F$ is a quadratic extension, we let ${}^L\GL_n(E)=(\GL_n(\C)\times\GL_n(\C))\rtimes W_F$, where $W_E\subset W_F$ acts trivially on $\GL_n(\C)\times\GL_n(\C)$ and an element of $W_F\setminus W_E$ switches the first and the second components.

An $L$-parameter of $G$ is a continuous homomorphism $\varphi\,\colon WD_F\to {}^LG$ which commutes with the projections $WD_F\to W_F$ and ${}^LG\to W_F$, $\varphi(W_F)$ consists of semi-simple elements and $\varphi|_{SL_2(\C)}$ is algebraic when $F$ is non-archimedean.
Let $\Phi(G)$ denote the set of $\hat{G}(\C)$-conjugacy classes of $L$-parameters of $G$ and $\Phi_\temp(G)$ the subset of $\varphi\in\Phi(G)$ such that the projection of $\varphi(W_F)$ in $\hat{G}(\C)$ is bounded.

The local Langlands conjecture states that there is a finite-to-one map $\Irr(G(F))\to\Phi(G),\, \pi\mapsto\varphi_\pi$ which satisfies many properties.
This map is called the local Langlands correspondence.

When $G$ is a general linear group, this correspondence is established by \cite{HT01}, \cite{Hen00}.
In this case, the correspondence is bijective.
It is also established for symplectic groups and odd split special orthogonal groups by \cite{Art13} and for quasi-split unitary groups by \cite{Mok15}.

We fix a non-trivial additive character $\psi$ of $F$.
For an $L$-parameter $\varphi\in\Phi(G)$, we associate a local $L$-factor $L(s, \varphi)$ and a local $\varepsilon$-factor $\varepsilon(s, \varphi, \psi)$ as in \cite{Tat79}.
The local $\gamma$-factor of $\varphi$ is defined as
    \[
    \gamma(s, \varphi, \psi)=\varepsilon(s, \varphi, \psi)
    \frac{L(1-s, \varphi^\vee)}{L(s, \varphi)},
    \]
where $\varphi^\vee$ is the contragredient of $\varphi$.
For $\pi\in\Irr(G(F))$ and a representation $r$ of ${}^LG$, we set 
    \[
    \gamma(s, \pi, r, \psi)=\gamma(s, r\circ\varphi_\pi, \psi)
    \]
assuming the local Langlands correspondence for $G(F)$.

\subsubsection{Classical groups}

\medskip
(I\hspace{-.1em}I) 
Let 
    \[
    H(F)=\U(n)_{E/F}=\{g\in\GL_n(E) \mid \iota(\t g) w_ng=w_n\}
    \]
be the quasi-split unitary group. 

\medskip
(I\hspace{-.1em}I\hspace{-.1em}I) 
Let 
    \[
    H(F)=\SO(n+1, n)=\{g\in\GL_{2n+1}(F) \mid \t gw_{2n+1}g=w_{2n+1}\}
    \]
be the odd split special orthogonal group.

\medskip
(I\hspace{-.1em}V) 
Let 
    \[
    H(F)=
        \begin{cases}
        \Sp(n-1)=\{g\in\GL_n(F) \mid \t gJ_{n-1}g=J_{n-1}\} & \text{$n$ is odd},\\
        \SO(n, \chi)=\{g\in\SL_n(F) \mid \t gw_{n, \chi}g=w_{n,\chi}\} 
        & \text{$n$ is even}
        \end{cases}
    \]
be the symplectic group when $n$ is odd and the even quasi-split special orthogonal group when $n$ is even.
Here, 
    \[
    w_{n,\chi}=
        \begin{pmatrix}
         & & & w_{n/2-1} \\
         & & \beta & \\
         & 1 & &\\
        w_{n/2-1} & & & 
        \end{pmatrix}
    \]

\medskip
In the following, consider one of Cases (II), (III), or (IV). 
To simplify the notation, we set $E=F$ in Case (I\hspace{-.1em}I\hspace{-.1em}I) and Case (I\hspace{-.1em}V). 
We will discuss the endoscopic lift of representations of $H(F)$ to those of $G(F)$.

Since $G(E)$ is a general linear group in all cases, the correspondence $\Irr(G(E))\to\Phi(G)$ is bijection.
Hence, given a homomorphism ${}^LH\to{}^LG$ and assuming the local Langlands conjecture for $H(F)$, it induces a map $\Irr(H(F))\to\Irr(G(F))$.
We summarize some properties of this map we will use.
Except for the Case (I\hspace{-.1em}I\hspace{-.1em}I) with $n$ even, all these properties are established in \cite{Art13} and \cite{Mok15}.

In Case (I\hspace{-.1em}I), we have ${}^LH=\GL_n(\C)\rtimes W_F$, where $W_E\subset W_F$ acts trivially on $\GL_n(\C)\times\GL_n(\C)$ and $\sigma\in W_F\setminus W_E$ acts by $\sigma(g)=J_n\t g^{-1} J_n^{-1}$.
Note that ${}^LG=(\GL_n(\C)\times\GL_n(\C))\rtimes W_F$.
The homomorphism ${}^LH\to{}^LG$ is given as $g\rtimes\sigma\to(g, J_n\t g^{-1}J_n^{-1})\rtimes\sigma$.
The induced map $\mathrm{BC}\,\colon \Irr(H(F))\to\Irr(G(E))$ is called the standard base change map.
Let $\kappa$ denote a character of $E^\times$ such that $\kappa|_{F^\times}=\eta_{E/F}$. 
The mapping $\mathrm{BC}_\kappa \colon \Irr(H(F))\to\Irr(G(E))$ $(\mathrm{BC}_\kappa(\sigma)\coloneqq\mathrm{BC}(\sigma)\otimes\kappa)$ is called the non-standard base change map. 
Set
\[
\el_r\coloneqq \begin{cases}
    \mathrm{BC} & \text{if [Case (I\hspace{-.1em}I$+$) and $n$ is odd] or [Case (I\hspace{-.1em}I$-$) and $n$ is even]}, \\
    \mathrm{BC}_\kappa & \text{if [Case (I\hspace{-.1em}I$-$) and $n$ is odd] or [Case (I\hspace{-.1em}I$+$) and $n$ is even]}. 
\end{cases}
\]

In Case (I\hspace{-.1em}I\hspace{-.1em}I), we have ${}^LH=\Sp_{2n}(\C)\times W_F$ and we consider the natural inclusion ${}^LH\hookrightarrow{}^LG=\GL_{2n}(\C)\times W_F$.
Let $\el_r\,\colon\Irr(H(F))\to\Irr(G(E))$ be the corresponding map.

In Case (I\hspace{-.1em}V), we have ${}^LH=\SO_n(\C)\rtimes W_F$.
When $n$ is odd or $\chi$ is trivial, the action of $W_F$ on $\hat{H}(\C)=\SO_n(\C)$ is trivial and we consider the natural inclusion ${}^LH\hookrightarrow{}^LG=\GL_n(\C)\times W_F$.
Suppose that $n$ is even and $\chi$ is non-trivial. 
Let $E$ denote the quadratic extension of $F$ associated with the quadratic character $\chi$.
The action of $W_F$ on $\hat{H}(F)=\SO_n(\C)$ factors through $W_F/W_E\simeq\Gal(E/F)$ and the generator $\gamma\in\Gal(E/F)$ acts by conjugation of
    \[
    \nu\coloneqq
        \begin{pmatrix}
        1_{n/2-1} & & \\
        & w_2 & \\
        & & 1_{n/2-1}
        \end{pmatrix}\in\O_n(\C).
    \]
We consider the embedding ${}^LH\to{}^LG$, which factors through $\SO_n(\C)\rtimes\Gal(E/F)$, by sending $\gamma$ to $\nu$.
In each case, let $\el_r\,\colon\Irr(H(F))\to\Irr(G(E))$ be the corresponding map.
Note that when $n$ is even, the image of $\el_r$ is contained in $\Irr(\GL_n(F))_\chi$.

In the above setting, we obtain $\EL(G(E),r)=\el_r(\Temp(H(F)))$. 
We also have
    \begin{equation}\label{eq:gamma}
    \gamma(s, \sigma, \Ad, \psi)=\frac{\gamma(s, \el_r(\sigma), \Ad, \psi)}
    {\gamma(s, \el_r(\sigma), r, \psi)}
    \end{equation}
for $\sigma\in\Temp(H(F))$.
This is \cite{BP21a}*{(2.12.9)} in Case (I\hspace{-.1em}I+) and \cite{Duh19}*{(117)} in Case (I\hspace{-.1em}I\hspace{-.1em}I).

For $\pi\in\Temp(\GL_n(E))$, we have already defined the component group $S_\pi$ in \eqref{eq:Sgroup}, but originally the component group $S_\pi$ is defined as the group of connected components of the centralizer of the $L$-parameter $\phi_\pi$ in $\widehat{(G/Z_G)}(\C)$.
We only use the cardinality $|S_\pi|$, and the precise definition is not needed, hence the explicitly calculated results were used to define the group $S_\pi$.

For $\sigma\in\Temp(H(F))$, the component group $S_\sigma$ is defined in the same way, and its group structure is given as follows.
Suppose that $\sigma$ embeds in $\pi_1\times\cdots\times\pi_r\rtimes\sigma_0$, where $\pi_i\in\Pi_2(\GL_{n_i}(E))$ and $\sigma_0$ is a square-integrable representation of a quasi-split classical group of the same type as $H(F)$.
Then
    \[
    S_\sigma\simeq S_{\pi_1}\times\cdots\times S_{\pi_r}
    \times S_{\sigma_0},
    \]
with $S_{\pi_i}$ defined above and $S_{\sigma_0}\simeq(\Z/2\Z)^k$ where $k$ is such that $\el_r(\sigma_0)\simeq\pi'_1\times\cdots\times\pi'_k$ for some $\pi'_j\in\Pi_2(\GL_{m_j}(E))$.
Note that $\fS_\pi\simeq S_\sigma$ for $\sigma\in\Temp(H(F))$ satisfying $\el_r(\sigma)=\pi$. 

\section{Fourier transforms}\label{sec:Fourier}

Let $q_F$ denote the order of the residue field when $F$ is non-archimedean. 
We set
\begin{equation}\label{eq:localfactor}
    \zeta_F(s)=
        \begin{cases}
        (1-q_F^{-s})^{-1} & \text{if $F$ is non-archimedean}, \\
        \pi^{-\frac{s}{2}}\Gamma(\tfrac{s}{2}) & \text{if $F=\R$}, \\
        2(2\pi)^{-s}\Gamma(s) & \text{if $F=\C$}.
        \end{cases}    
\end{equation}
Recall the factor $\gamma^*(0,\pi,r,\psi)$ which was defined in \eqref{eq:gamma*} by $\zeta_F(s)$ and $\gamma(s,\pi,r,\psi)$.  
This factor and its inverse are of moderate growth \cite{BP21a}*{Lemma 2.45}.

\subsection{Case (I)}

Recall that $G(F)=\GL_n(F)\times \GL_n(F)$. 

\subsubsection{The geometric side}

For $f\in \Cc(G(F)/Z(F))$, we have
\begin{align*}
(M(s)\varphi_{s,  f}(1_{2n}))(1_n)&=\int_{X'(F)} (\varphi_{s,  f}\left(
        \begin{pmatrix}
        0 & -1_n \\
        1_n & x
        \end{pmatrix}
    \right))(1_n) \rd x \\
    &=\int_{X'(F)} (\varphi_{s,  f}\left(
        \begin{pmatrix}
        x^{-1} & -1_n \\
        0 & x
        \end{pmatrix}
        \begin{pmatrix}
        1_n & 0 \\
        x^{-1} & 1_n
        \end{pmatrix}
    \right))(1_n) \rd x 
    \end{align*}
Since $\delta_{P^\sharp}\left(
        \begin{pmatrix}
        g_1 & 0 \\
        0 & g_2
        \end{pmatrix}
    \right)=|\det(g_1^{-1}g_2)|^n$ for $(g_1,g_2)\in G(F)$ and $\rd^\times x\coloneqq |\det(x)|^{-n} \rd x$ is a $G(F)$-invariant measure on $X'(F)$,  the last expression becomes    
    \begin{align*}
    \int_{X'(F)}\mathbf{1}_L(x^{-1})|\det(x)|^{-s}f(x^{-1},x) \rd^\times x 
    &=\int_{X'(F)}\mathbf{1}_L(x)|\det(x)|^sf(x,x^{-1}) \rd^\times x \\
    &=\int_{F^\times \bs X'(F)}\alpha_s(x)f(x, x^{-1}) \frac{\rd^\times x}{\rd^\times z}.
\end{align*}
Here, $\rd^\times z$ is a Haar measure on $F^\times\simeq Z(F)$ and we set $\alpha_s(x)=|\det(x)|^{s} \int_{F^\times} \mathbf{1}_L(zx)\, |z|^{ns}\, \rd^\times z$. 

When $F$ is non-archimedean, for a fixed $x\in X'(F)$ we can explicitly compute it to obtain 
    \[
    \alpha_s(x)=|\det(x)|^{s}\cdot q_F^{mns}(1-q_F^{-ns})^{-1}(1-q_F^{-1}),  
    \]
where $m\coloneqq\min(\ord_F(x_{ij}))$ for $x=(x_{ij})$. 
In this case we have $\lim_{s\to0+} s\,\alpha_s(x)=\frac{1-q_F^{-1}}{n\log q_F}> 0$. 
When $F=\R$, we have
    \begin{align*}
    \alpha_s(x)&=|\det(x)|^{s} \int_{\R^\times} \phi(|z|\|x\|)|z|^{ns} \rd^\times z    =2\, |\det(x)|^{s}\|x\|^{-ns} \, \int_0^\infty  \phi(r)\, r^{ns-1} \rd r    \\
    &=-2\, |\det(x)|^{s}\|x\|^{-ns} \,\int_0^\infty 
    \phi'(r) \frac{1}{ns}r^{ns} \rd r.
    \end{align*}
In this case, we have $\lim_{s\to0+} s\,\alpha_s(x)=\frac{2\phi(0)}{n}>0$ by Lebesgue's dominated convergence theorem.  
When $F=\C$, we have
    \begin{align*}
    \alpha_s(x)&=|\det(x)|^{s} \int_{\C^\times} \phi(|z| \|x\|)|z|^{ns} \rd^\times z = 2\pi\, |\det(x)|^{s}\|x\|^{-ns} \, \int_0^\infty  \phi(r)\, r^{ns-1} \rd r    \\
    &=-2\pi\, |\det(x)|^{s}\|x\|^{-ns} \,\int_0^\infty 
    \phi'(r) \frac{1}{ns}r^{ns} \rd r.
    \end{align*}
In this case, we have $\lim_{s\to0+} s\,\alpha_s(x)=\frac{2\pi\phi(0)}{n}>0$.  
Hence, we obtain the next lemma.
\begin{lemma}\label{geom_gl}
There is a constant $c>0$ such that
    \[
    \lim_{s\to0+} s\cdot (M(s)\varphi_{s,  f}(1_{2n}))(1_n)
    =c\times \bJ^{\theta}(f)
    \]
for $f\in\Cc(G(F)/Z(F))$.
Here, $\bJ^{\theta}(f)\coloneqq\int_{F^\times\bs X'(F)}f(x,x^{-1})\rd^\times x/\rd^\times z$. 
\end{lemma}

\subsubsection{The spectral side}
Suppose that $\re(s)>0$ is sufficiently large so that the intertwining operator $(M(s)\varphi_{s,  f}(1_{2n}))(1_n)$ is given by the convergent integral
    \[
    (M(s)\varphi_{s,  f}(1_{2n}))(1_n)=    
    \int_{N^\sharp(F)} (\varphi_{s,  f}(w^{-1}u))(1_n) \rd u.
    \]
Applying the Harish-Chandra Plancherel formula to $\varphi_{s,  f}(w^{-1}u)\in\Cc(G(F)/Z(F))$,  we have
    \[
    (\varphi_{s,  f}(w^{-1}u))(1_n)=\int_{\Temp(G(F)/Z(F))} \tr(\pi(\varphi_{s,  f}(w^{-1}u)^\vee)) \rd\mu_{G(F)/Z(F)}(\pi),
    \]
where $\rd\mu_{G(F)/Z(F)}$ means the Plancherel measure on $G(F)/Z(F)$.

First we consider the case $F$ is non-archimedean.
Only those $\pi\in\Temp(G(F)/Z(F))$ having nonzero $K_{M^\sharp}$-invariant vector contributes to this integral.
Thus $\tr(\pi(\varphi_{s,  f}(w^{-1}u)^\vee))$ is supported on finite number of connected components of $\pi\in\Temp(G(F)/Z(F))$.

Suppose $\pi\in\Temp(G(F)/Z(F))$ satisfies $\pi^{K_{M^\sharp}}\neq0$. 
There exist representations $\pi_1$ and $\pi_2$ in $\Temp(\GL_n(F))$ so that $\pi=\pi_1\otimes\pi_2$. 
We may suppose $K_{M^\sharp}=K_0\times K_0$ for a compact subgroup $K_0$ in $\GL_n(F)$ without loss of generality. 
Take an orthonormal basis $v_{j1},  v_{j2},  \ldots,  v_{jm_j}$ of $\pi_j^{K_0}$ with respect to the invariant inner product $(\cdot ,  \cdot)_{\pi_j}$ on $\pi_j$, where $j=1$ or $2$. 
Note that $\{v_{1i_1}\otimes v_{2i_2} \}_{1\le i_1\le m_1,1\le i_2\le m_2}$ forms an orthonarmal basis of $\pi^{K_{M^\sharp}}=\pi_1^{K_0}\otimes \pi_2^{K_0}$. 
Write $(\cdot ,  \cdot)_{\pi}$ the invariant inner product on $\pi$.  
Then we have
\begin{multline}\label{eq:tracegl}
    \tr(\pi(\varphi_{s,  f}(w^{-1}u)^\vee))
    = \sum_{i_1=1}^{m_1}\sum_{i_2=1}^{m_2} (\pi(\varphi_{s,  f}(w^{-1}u)^\vee)v_{1i_1}\otimes v_{2i_2},  v_{1i_1}\otimes v_{2i_2})_\pi \\
    = \sum_{i_1=1}^{m_1}\sum_{i_2=1}^{m_2} (\varphi_{s,  \pi(f^\vee)v_{1i_1}\otimes v_{2i_2}}(w^{-1}u),  v_{1i_1}\otimes v_{2i_2})_\pi.    
\end{multline}
When $F$ is archimedean, we may suppose $\Gamma=\Gamma_0\times\Gamma_0$ for some finite subset $\Gamma_0$ of $\Irr_\ru(K_0)$, and only those $\pi\in\Temp(G(F)/Z(F))$ having $K$-types in $\Gamma_0$ contributes to the integral.
The equation \eqref{eq:tracegl} is valid if we let $v_{j1},  \ldots,  v_{jm_j}$ be an orthonormal basis of the space of $K_0$-type vectors in $\pi_j$ belonging to $\Gamma_0$.

Let $M_\pi(s)\,\colon I_\pi(s)\to I_{w(\pi)}(-s)$ be the standard intertwining operator.
Since we assumed $\re(s)$ is sufficiently large,  
    \[
     (M_\pi(s)(\varphi_{s,  \pi(f^\vee)v_{1i_1}\otimes v_{2i_2}})(1_{2n}),  v_{1i_1}\otimes v_{2i_2})_\pi
     =\int_{N^\sharp(F)}(\varphi_{s,  \pi(f^\vee)v_{1i_1}\otimes v_{2i_2}}(w^{-1}u),  v_{1i_1}\otimes v_{2i_2})_\pi \rd u.
    \]
    
Let $\mathrm{Std}\colon {}^L\GL_n \to \GL_n(\C)$ denote the standard representation. 
We obtain the gamma factor $\gamma(s,\pi_1\otimes\pi_2,\mathrm{Std}\otimes\mathrm{Std}^\vee,\psi)$ by \cite{Sha90}*{Theorem 3.5}. 
In particular, for any $\pi_0\in\Temp(\GL_n(F))$,
\[
\gamma(s,\pi_0\otimes\pi_0,\mathrm{Std}\otimes\mathrm{Std}^\vee,\psi)=\gamma(s,\pi_0,\mathrm{Ad},\psi). 
\]

\begin{lemma}\label{operator_gl}
Let $\pi_0\in\Temp(\GL_n(F))_\chi$. 
Set $\pi=\pi_0\otimes\pi_0(\in \Temp(G(F)/Z(F)))$ and $v_j=v_{1j}=v_{2j}$. 
Then there is a constant $c\in\C^\times$ such that
    \begin{align*}
    &\lim_{s\to0+}\gamma(s, \pi_0, \Ad, \psi)\, 
    (M_\pi(s)(\varphi_{s,  \pi(f^\vee)v_{1i_1}\otimes v_{2i_2}})(1_{2n}),  v_{i_1}\otimes v_{i_2})_\pi \\
    &\hspace{40pt} =c\,\chi(-1)^{n-1}\,((\varphi_{s, \pi(f^\vee)v_{i_1}\otimes v_{i_2}})(1_{2n}),  v_{i_2}\otimes v_{i_1})_\pi
    =c\,\chi(-1)^{n-1}\,(\pi(f^\vee)v_{i_1}\otimes v_{i_2}, v_{i_2}\otimes v_{i_1})_\pi.
    \end{align*}
Notice that $c$ does not depend on $\chi$ and $\pi$. 
Define $\pi(\theta)w\otimes w'=w'\otimes w$ where $w$, $w'\in\pi_0$. 
In particular,  we have
    \[
    \lim_{s\to0+}\gamma(s, \pi_0, \Ad, \psi)
    \int_{N^\sharp(F)} \tr(\sigma(\varphi_{s,  f}(w^{-1}u)^\vee)) \rd u
    =c\, \tr(\pi(\theta)\pi(f^\vee)).
    \]
\end{lemma}
\begin{proof}
This is \cite{HII08}*{Lemma 4.1}. 
By \cite{Sha90}*{Theorem 7.9} the normalized intertwining operator 
    \[
    \lim_{s\to0+}\pi(\theta)\, \gamma(s,\pi_0\otimes\pi_0,\mathrm{Std}\otimes\mathrm{Std}^\vee,\psi)\, M_\pi(s)
    \]
is a unitary endomorphism of $I_\pi(0)$.
It follows from \cite{Sha84}*{Theorem 5.1} and \cite{Sha85}*{Theorem 1} that there exists a constant $c\in\C^\times$ such that this endomorphism equals the scalar operator $c\, \chi(-1)^{n-1}\, \id$.
The factor $\omega_2(-1)^n$ in \cite{Sha84}*{Theorem 5.1} is replaced by $\chi(-1)^{n-1}$ in above, since our $w$ is different from that paper.  
For $\pi_1$, $\pi_2\in\Pi_2(\GL_n(F))$, we see that $\gamma(0,\pi_1\otimes\pi_2,\mathrm{Std}\otimes\mathrm{Std}^\vee,\psi)=0$ if and only if $\pi_1\simeq \pi_2$. 
\end{proof}

\begin{proposition}\label{spec_gl}
There is a constant $c\in\C^\times$ such that for $f\in\Cc(G(F)/Z(F))$, we have
    \begin{align*}
     &\lim_{s\to0+} s\cdot (M(s)\varphi_{s,  f}(1_{2n}))(1_n) \\
    & \qquad = c\, \int_{\pi_0 \in \EL(G(F),r)}  \tr(\pi(\theta)\pi(f^\vee)) \frac{\chi_{\pi_0}(-1)^{n-1} \, \gamma^\ast(0, \pi_0, \Ad, \psi)}{|S_{\pi_0}|} 
    \rd_{\EL(G(F),r)}(\pi_0).
    \end{align*}
    where we set $\pi=\pi_0\otimes\pi_0$ and $\chi_{\pi_0}$ denotes the central character of $\pi_0$. 
\end{proposition}

\begin{proof}
It known that the Plancherel measure for $\Temp(G(F))$ is given by
\[
\rd\mu_{G(F)}(\pi)= c'  \frac{\chi_{\pi_1}\chi_{\pi_2}(-1)^{n-1}\gamma^\ast(0, \pi_1, \Ad, \psi)\, \gamma^\ast(0, \pi_2, \Ad, \psi)}{|S_{\pi_1}|\, |S_{\pi_2}|} \rd_{\EL(G(F),r)}(\pi_1) \, \rd_{\EL(G(F),r)}(\pi_2)
\]
for some contant $c'>0$, see \cite{BP21a}*{Proof of Proposition 2.132}. 
By an argument similar to \cite{BP21a}*{Proof of Proposition 3.41} there exists a constant $c\in\C^\times$ such that
\begin{align*}
     &\lim_{s\to0+} s\int_{\Temp(G(F)/Z(F))} \Phi(\pi) \, \gamma(s, \pi, \mathrm{Std}\otimes\mathrm{Std}^\vee, \psi)^{-1} \, \rd\mu_{G(F)/Z(F)}(\pi) \\
    & \qquad = c\, \int_{\pi_0 \in \EL(G(F),r)}  \Phi(\pi_0\otimes\pi_0) 
    \, \gamma^\ast(0, \pi_0, \Ad, \psi) \,
    \frac{\rd_{\EL(G(F),r)}(\pi_0)}{|S_{\pi_0}|}
    \end{align*}
for all $\Phi\in\cS(\Temp(G(F)/Z(F)))$.
Here, we write $\cS(\Temp(G(F))/Z(F))$ for the space of Schwartz functions on $\Temp(G(F))/Z(F))$, see \cite{BP21a}*{p.191} for its definition.
The proposition follows from \cref{operator_gl} and this equation.
\end{proof}

\subsubsection{Conclusion}\label{sec:con1}

Combining \cref{geom_gl} and \cref{spec_gl}, we obtain the next theorem.

\begin{theorem}\label{eq_gl}
There is a constant $c\in\C^\times$  such that 
\[
\bJ^{\theta}(f)=c\,  \int_{\pi_0 \in \EL(G(F),r) } \tr(\pi(\theta)\pi(f^\vee)) 
    \frac{\chi_{\pi_0}(-1)^{n-1} \,\gamma^\ast(0, \pi_0, \Ad, \psi)}{|S_{\pi_0}|}
    \rd_{\EL(G(F),r)}(\pi_0)
\]
for all $f\in\Cc(G(F)/Z(F))$, where $\pi=\pi_0\otimes\pi_0$.  
\end{theorem}

Recall the twisted orbital integral $I^\theta(h,\trep)$ defined in \S\ref{sec:toihermi} for the case $E=F\times F$. 
Set 
\[
\theta'(g)\coloneqq J_n\t\!g^{-1}J_n^{-1} \quad (g\in\GL_n(F)) 
\]
and the involution $\theta$ on $G(F)$ is given by $\theta(g_1,g_2)=(\theta'(g_2),\theta'(g_1))$. 
The vector space $\HE_n(F)$ is given by
\[
\HE_n(F)=\{ (x_1,x_2)\in \M_n(F)\oplus\M_n(F) \mid  x_2=(-1)^{n-1} \t \! x_1 \}(\simeq \M_n(F)). 
\]
The group $G(F)$ acts on $\HE_n(F)$ as 
\[
(x,(-1)^{n-1}\t\!x)\cdot (g_1,g_2)=(\t\!g_1 xg_2,(-1)^{n-1} \t\!(\t\!g_1 xg_2))
\]
for $(g_1,g_2)\in G(F)$ and $(x,(-1)^{n-1}\t\!x)\in\HE_n(F)$. 
Take a function $h\in C_c^\infty (G(F))$ and define $f\in C_c^\infty (G(F)/Z(F))$ as
\[
f(g_1,g_2)=\int_{F^\times} h(ag_1,\theta'(a^{-1}g_2))\,  \rd^\times a .
\]
Then, there exist non-zero constants $c_1$ and $c_2$ such that we have $I^\theta(h,\trep)=c_1\times \bJ^\theta(f)$ and
\[
c_2\times \tr((\pi_0\otimes\pi_0)(\theta)(\pi_0\otimes\pi_0)(f^\vee))=\tr((\pi_0\otimes\pi_0\circ\theta')(\theta)(\pi_0\otimes\pi_0\circ\theta')(h^\vee))(=\hat{h}^\theta(\pi_0\otimes\pi_0\circ\theta'))
\]
where $(\pi_0\otimes\pi_0\circ\theta')(\theta)(v_1\otimes v_2)=v_2\otimes v_1$ $(v_1,v_2\in\pi_0)$.
Therefore, we arrive at
\begin{corollary}\label{cor:measureGL}
There exists a constant $c\neq 0$ such that
\[
I^\theta(h,\trep)=c \int_{\pi_0 \in \EL(G(F),r) } \hat{h}^\theta(\pi_0\otimes\pi_0\circ\theta') 
    \frac{\chi_{\pi_0}(-1)^{n-1} \,\gamma^\ast(0, \pi_0, \Ad, \psi)}{|S_{\pi_0}|}
    \rd_{\EL(G(F),r)}(\pi_0)
\]    
where $\chi_{\pi_0}$ means the central character of $\pi_0$. 
\end{corollary}

A measure $\mu_{\EL(G(F),r)}$ on $\EL(G(F),r)$ is defined by
\begin{equation}\label{eq:measureGL}
    \mu_{\EL(G(F),r)}(f)\coloneqq  \int_{\pi_0 \in \EL(G(F),r) } f(\pi_0\otimes\pi_0\circ\theta') 
    \frac{\chi_{\pi_0}(-1)^{n-1} \,\gamma^\ast(0, \pi_0, \Ad, \psi)}{|S_{\pi_0}|}
    \rd_{\EL(G(F),r)}(\pi_0), 
\end{equation}
where $f$ is a test function on $\EL(G(F),r)$. 
This was used to assert our main theorem (\cref{thm:maintheorem1}), and \cref{cor:measureGL} was used in the proof of \cref{thm:maintheorem1} along with \cref{thm:asym} {\it(3)}.

\subsection{Case (I\hspace{-.1em}I)}

Recall that $\chi$ is a unitary character of $E^\times$ which is trivial on $N_{E/F}(E^\times)$.
In particular, the character $\chi\circ\det$ on $G(E)=\GL_n(E)$ is $\theta$-invariant.
By abuse of notation, we write $\chi\circ\det$ as $\chi$.

\subsubsection{The geometric side}

For $f\in C^\infty_c(G(F)/N_{E/F}(E^\times))$ we have
\begin{align*}
(M(s, \chi)\varphi_{s,  f}(1_{2n}))(1_n)&=\int_{X'(F)} (\varphi_{s,  f}\left(
        \begin{pmatrix}
        0 & -1_n \\
        1_n & x
        \end{pmatrix}
    \right))(1_n) \rd x \\
    &=\int_{X'(F)} (\varphi_{s,  f}\left(
        \begin{pmatrix}
        x^{-1} & -1_n \\
        0 & x
        \end{pmatrix}
        \begin{pmatrix}
        1_n & 0 \\
        x^{-1} & 1_n
        \end{pmatrix}
    \right))(1_n) \rd x 
    \end{align*}
Since $\delta_{P^\sharp}\left(
        \begin{pmatrix}
        a & 0 \\
        0 & \theta(a)
        \end{pmatrix}
    \right)=|\det(a)|_E^n$ for $a\in G(E)$ and $\rd^\times x\coloneqq|\det(x)|^{-n} \rd x$ is a $G(E)$-invariant measure on $X'(F)$,  the last expression becomes    
    \begin{multline*}
    \int_{X'(F)}\mathbf{1}_L(x^{-1})\chi^{-1}(\det(x))|\det(x)|^{-s}f(x^{-1}) \rd^\times x  \\
    =\int_{X'(F)}\mathbf{1}_L(x)\chi(\det(x))|\det(x)|^s \, f(x) \rd^\times x 
    =\int_{N_{E/F}(E^\times)\bs X'(F)} \alpha_s(x) \, f(x) \, \frac{\rd^\times x}{\rd^\times z} .
    \end{multline*}
Here, $\rd^\times z$ is a Haar measure on $N_{E/F}(E^\times)$ and we set 
\[
\alpha_s(x)=\chi(\det(x))|\det(x)|^s \int_{E^\times} \mathbf{1}_L(N_{E/F}(z)x)|z|_E^{ns} \rd^\times z. 
\]

When $F$ is non-archimedean,  we can explicitly compute it to obtain 
    \[
    \alpha_s(x)=\chi(\det(x))|\det(x)|^s \cdot q_E^{[\frac{m}{2}]ns}(1-q_E^{-ns})^{-1}(1-q_E^{-1}),  
    \]
where $m\coloneqq\min(\ord_E(x_{ij}))$ for $x=(x_{ij})\in X'$.
When $F=\R$, we have
    \begin{align*}
    \alpha_s(x)&=\chi(\det(x))|\det(x)|^s \int_{\C^\times} \phi(|z|^2\|x\|)|z|^{ns} \rd^\times z    \\
    &=\chi(\det(x))|\det(x)|^s\|x\|^{-\frac{ns}{2}}\int_0^\infty \int_0^{2\pi} \phi(r^2)r^{ns-1} \rd r \rd\theta    \\
    &=-2\pi\chi(\det(x))|\det(x)|^s\|x\|^{-\frac{ns}{2}}\int_0^\infty 
    \phi'(r^2)\cdot 2r \frac{1}{ns}r^{ns} \rd r.
    \end{align*}
In both cases, we have $\lim_{s\to0+} s\,\alpha_s(x)=c\times\chi(\det(x))$ for some constant $c>0$ and we obtain the next lemma.

\begin{lemma}\label{geom_unitary}
There is a constant $c>0$ such that
    \[
    \lim_{s\to0+} s\cdot (M(s, \chi)\varphi_{s,  f}(1_{2n}))(1_n)
    =c\times \bJ^\theta(f,\chi)
    \]
for $f\in\Cc(G(E)/N_{E/F}(E^\times))$.
Here, 
    \[
    J^\theta(f,\chi)\coloneqq\int_{N_{E/F}(E^\times)\bs X'(F)} \chi(x) \, f(x) \, \frac{\rd^\times x}{\rd^\times z}
    \]
is the twisted orbital integral on $X'(F)$.
\end{lemma}

\subsubsection{The spectral side}
Suppose that $\re(s)>0$ is sufficiently large so that the intertwining operator $(M(s, \chi)\varphi_{s,  f}(1_{2n}))(1_n)$ is given by the convergent integral
    \[
    (M(s, \chi)\varphi_{s,  f}(1_{2n}))(1_n)=    
    \int_{N^\sharp(F)} (\varphi_{s,  f}(w^{-1}u))(1_n) \rd u.
    \]
Applying the Harish-Chandra Plancherel formula to $\varphi_{s,  f}(w^{-1}u)\in\Cc(G(E)/N_{E/F}(E^\times))$,  we have
    \[
    (\varphi_{s,  f}(w^{-1}u))(1_n)=\int_{\Temp(G(E)/N_{E/F}(E^\times))} \tr(\pi(\varphi_{s,  f}(w^{-1}u)^\vee)) \rd\mu_{G(E)/N_{E/F}(E^\times)}(\pi)
    \]
where $\rd \mu_{G(E)/N_{E/F}(E^\times)}$ denotes the Plancherel measure on $\Temp(G(E)/N_{E/F}(E^\times))$. 

First we consider the case $F$ is non-archimedean.
Only those $\pi\in\Temp(G(E)/N_{E/F}(E^\times))$ having nonzero $K_{M^\sharp}$-invariant vector contributes to this integral.
Thus $\tr(\pi(\varphi_{s,  f}(w^{-1}u)^\vee))$ is supported on finite number of connected components of $\pi\in\Temp(G(E)/N_{E/F}(E^\times))$.

Suppose $\pi\in\Temp(G(E)/N_{E/F}(E^\times))$ satisfies $\pi^{K_{M^\sharp}}\neq0$.
Take an orthonormal basis $v_1,  v_2,  \ldots,  v_r$ of $\pi^{K_{M^\sharp}}$ with respect to the invariant inner product $(\cdot ,  \cdot)_\pi$ on $\pi$.
Then we have
    \begin{equation}\label{eq:trace}
    \tr(\pi(\varphi_{s,  f}(w^{-1}u)^\vee))
    =\sum_{i=1}^r (\pi(\varphi_{s,  f}(w^{-1}u)^\vee)v_i,  v_i)_\sigma
    =\sum_{i=1}^r (\varphi_{s,  \pi(f^\vee)v_i}(w^{-1}u),  v_i)_\pi.
    \end{equation}

When $F$ is archimedean,  only those $\pi\in\Temp(G(E)/N_{E/F}(E^\times))$ having $K$-types in $\Gamma$ contributes to the integral.
The equation \eqref{eq:trace} is valid if we let $v_1,  \ldots,  v_r$ be an orthonormal basis of the space of $K$-type vectors in $\pi$ belonging to $\Gamma$.

Let $M_\pi(s, \chi)\,\colon I_\pi(s, \chi)\to I_{w(\pi)}(-s,  {}^w\chi)$ be the standard intertwining operator.
Since we assumed $\re(s)$ is sufficiently large,  
    \[
     (M_\pi(s, \chi)(\varphi_{s,  \pi(f^\vee)v_i})(1_{2n}),  v_i)_\pi
     =\int_{N^\sharp(F)}(\varphi_{s,  \pi(f^\vee)v_i}(w^{-1}u),  v_i)_\pi \rd u.
    \]

We set
\[
r_\chi\coloneqq \begin{cases}
    \As^+ & \text{when $\chi|_{F^\times}$ is trivial,} \\
    \As^- & \text{when $\chi|_{F^\times}$ is non-trivial.}    
\end{cases}
\]

\begin{lemma}\label{operator_unitary}
Suppose that $\pi$ belongs to $\EL(G(E),r_\chi)$. 
Then there is a constant $c_{\pi , \chi}\in\C^\times$ with absolute value $1$ such that
    \begin{align*}
    &\lim_{s\to0+}\gamma(s, \pi, r_\chi, \psi)
    (M_\pi(s, \chi)(\varphi_{s, \pi(f^\vee)v_i})(1_{2n}), v_i)_\pi \\
    &\hspace{40pt} =c_{\pi, \chi}\,((\varphi_{s, \pi(f^\vee)v_i})(1_{2n}),  \pi(\theta)v_i)_\pi
    =c_{\pi, \chi}\,(\pi(f^\vee)v_i, \pi(\theta)v_i)_\pi.
    \end{align*}
In the situation where $\chi$ is fixed, the map $\pi\mapsto c_{\pi, \chi}$ is continuous on $\EL(G(E),r_\chi)$.
In particular,  we have
    \[
    \lim_{s\to0+}\gamma(s, \pi, r_\chi, \psi)
    \int_{N^\sharp(F)} \tr(\sigma(\varphi_{s,  f}(w^{-1}u)^\vee)) \rd u
    =c_{\pi, \chi}\, \widehat{f}^\theta(\pi),
    \]
where $\widehat{f}^\theta(\pi)\coloneqq \tr(\pi(\theta)\pi(f^\vee))$.
\end{lemma}

\begin{proof}
This is \cite{HII08}*{Lemma 7.1}.
By \cite{Sha90}*{Theorem 7.9} the normalized intertwining operator 
    \[
    \lim_{s\to0+}\pi(\theta)\gamma(s, \pi, r_\chi, \psi    )\, M_\pi(s, \chi)
    \]
is a unitary endomorphism of $I_\pi(0, \chi)$.
From \cite{Sha90}*{Theorem 3.5 (2)} it follows that this endomorphism equals the scalar operator $c_{\pi, \chi}\id$ for some $c_{\pi, \chi}\in\C^\times$. See also \cite{Atobe24}*{p.48, (S1)}.
The continuity of the map $\pi\mapsto c_{\pi, \chi}$ follows from the explicit description of the constant $c_{\pi, \chi}$ using gamma factors, which is also given in \cite{Sha90}*{Theorem 3.5 (2)}.
\end{proof}

\begin{proposition}\label{spec_unitary}
There is a constant $c\in\C^\times$ such that for $f\in\Cc(G(E)/N_{E/F}(E^\times))$, we have
    \begin{align*}
     &\lim_{s\to0+} s\cdot (M(s, \chi)\varphi_{s,  f}(1_{2n}))(1_n) \\
    & \qquad = c\int_{ \pi \in \EL(G(E),r_\chi) }  \widehat{f}^\theta(\pi) \frac{\chi_\pi(-1)^{n-1}c_{\pi, \chi}}{|\fS_\pi|}
    \frac{\gamma^\ast(0, \pi, \Ad, \psi)}{\gamma^\ast(0, \pi,  r_\chi, \psi)} 
    \rd_{\EL(G(E),r_\chi)}(\pi),
    \end{align*}
where $\chi_\pi$ means the central character of $\pi$. 
\end{proposition}

\begin{proof}
From \cite{BP21a}*{Proposition 3.41} and the Harish-Chandra Plancherel formula (cf. \cite{BP21a}*{Proof of Proposition 2.132}), there exists a constant $c\in\C^\times$ such that
\begin{align*}
     &\lim_{s\to0+} s\int_{\Temp(G(E)/N_{E/F}(E^\times))} \Phi(\pi)\, \gamma(s, \pi, r_\chi, \psi)^{-1}\, \rd \mu_{G(E)/N_{E/F}(E^\times)}(\pi) \\
    & \qquad = c\,  \int_{ \pi \in \EL(G(E),r) }  \Phi(\pi) 
    \frac{\chi_\pi(-1)^{n-1}\gamma^\ast(0, \pi, \Ad, \psi)}{\gamma^\ast(0, \pi, r_\chi, \psi)} 
    \frac{\rd_{\EL(G(E),r_\chi)}(\pi)}{|\fS_{\pi}|}
    \end{align*}
for all $\Phi\in\cS(\Temp(G(E)/N_{E/F}(E^\times)))$.
The proposition follows from \cref{operator_unitary} and this equation.
\end{proof}

\subsubsection{Conclusion}\label{sec:con2}

Combining \cref{geom_unitary} and \cref{spec_unitary}, we obtain the next theorem.

\begin{theorem}\label{eq_unitary}
There is a constant $c\in\C^\times$  such that
\[
\bJ^\theta(f,\chi)  = c\, \int_{ \pi \in \EL(G(E),r_\chi)} \widehat{f}^\theta(\pi) \frac{\chi_\pi(-1)^{n-1}c_{\pi, \chi}}{|\fS_\pi|}
    \frac{\gamma^\ast(0, \pi, \Ad, \psi)}{\gamma^\ast(0, \pi,  r_\chi, \psi)} 
    \rd_{\EL(G(F),r_\chi)}(\pi)
\]
for all $f\in\Cc(G(E)/N_{E/F}(E^\times))$. 
\end{theorem}

\begin{corollary}\label{cor:llc2}
We can rewrite the above formula in \cref{eq_unitary} using the local Langlands correspondence (cf. \S\ref{sec:Langlands}). 
There is a constant $c\in\C^\times$  such that 
    \[
     \bJ^\theta(f,\chi) 
     = c\int_{\Temp(H(F))}  \widehat{f}^\theta(\el_{r_\chi}(\sigma))
    \frac{c_{\el_{r_\chi}(\sigma),\chi}\gamma^\ast(0, \sigma, \Ad, \psi)}{|S_\sigma|} \rd_{\Temp(H(F))}(\sigma)
    \]
for all $f\in\Cc(G(E)/N_{E/F}(E^\times))$, where $\rd_{\Temp(H(F))}(\sigma)$ means a spectral measure on $\Temp(H(F))$.  
\end{corollary}
\begin{proof}
This is a consequence of \cref{eq_unitary} and \eqref{eq:gamma}.
\end{proof}

The following corollary is necessary to relate Theorem \ref{eq_unitary} to the global result of Theorem \ref{thm:asym}.
\begin{corollary}\label{cor:measureU}
There exists a constant $c\in \C^\times$ such that
\[
I^\theta(f,\chi)= c  \int_{ \pi \in \EL(G(E),r_\chi)} \widehat{f}^\theta(\pi) \frac{\chi_\pi(-1)^{n-1}c_{\pi, \chi}}{|\fS_\pi|}
    \frac{\gamma^\ast(0, \pi, \Ad, \psi)}{\gamma^\ast(0, \pi,  r_\chi, \psi)} 
    \rd_{\EL(G(E),r_\chi)}(\pi) 
\]
for any $f\in C_c^\inf(G(E))$, where $\chi_\pi$ means the central character of $\pi$. 
See \S\ref{sec:toihermi} for the definition of $I^\theta(f,\chi)$.  
\end{corollary}
\begin{proof}
For each function $f\in C_c^\inf(G(E))$, we set
\[
\tilde{f}(g)\coloneqq \int_{N_{E/F}(E^\times)} f(ag) \, \rd^\times a\in C_c^\inf(G(E)/N_{E/F}(E^\times)).
\]
Then, we have $I^\theta(f,\chi)=c'\times \bJ^\theta(\tilde{f},\chi)$ for some positive constant $c'$. 
Hence, the assertion follows from \cref{eq_unitary}. 
\end{proof}

A measure $\mu_{\EL(G(E),r_\chi)}$ on $\EL(G(E),r_\chi)$ is defined by
\begin{equation}\label{eq:measureU}
    \mu_{\EL(G(E),r_\chi)}(f) \coloneqq \int_{ \pi \in \EL(G(E),r_\chi)} f(\pi) \frac{\chi_\pi(-1)^{n-1}c_{\pi, \chi}}{|\fS_\pi|}
    \frac{\gamma^\ast(0, \pi, \Ad, \psi)}{\gamma^\ast(0, \pi,  r_\chi, \psi)} 
    \rd_{\EL(G(E),r_\chi)}(\pi), 
\end{equation}
where $f$ is a test function on $\EL(G(E),r_\chi)$. 

\subsection{Case (I\hspace{-.1em}I\hspace{-.1em}I)}


In this case, any $\pi\in\End(G(F),r)$ has the trivial central character on $F^\times$. 

\subsubsection{The geometric side}

Note that $G(F)$ acts on $X'(F)$ transitively. 
We use the fact that 
    \[
    \delta_{P^\sharp}\left(
        \begin{pmatrix}
        a & 0 \\
        0 & \theta(a)
        \end{pmatrix}
    \right)=|\det(a)|^{2n-1}
    \]
for $a\in G(F)$ and that $\rd^\times x\coloneqq |\det(x)|^{\frac12-n} \rd x$ is a $G(F)$-invariant measure on $X'(F)$.   
Similarly as in the previous case, by them we have
    \[
    (M(s)\varphi_{s,  f}(1_{4n}))(1_{2n})=\int_{F^\times\bs X'(F)}\alpha_s(x)\, f(x) \frac{\rd^\times x}{\rd^\times z},
    \]
where $\rd^\times z$ is  a Haar measure on $F^\times$ and $\alpha_s(x)=|\det(x)|^{\frac{s}{2}} \int_{F^\times}\mathbf{1}_L(z^2x)|z|_F^{2ns} \rd^\times z$.

When $F$ is non-archimedean, since $\mathbf{1}_L$ is the characteristic function of $X\cap\M_{2n}(\fo_F)$, we have
    \[
    \alpha_s(x)=|\det(x)|^{\frac{s}{2}}\cdot q_F^{2[\frac{m}{2}]ns}(1-q_F^{-2ns})^{-1}(1-q_F^{-1}).
    \]
Here, we set $m=\min(\ord_F(x_{ij}))$ for $x=(x_{ij})\in X'$.
When $F=\R$,     
    \begin{align*}
    \alpha_s(x)&=|\det(x)|^{\frac{s}{2}} \int_{\R^\times} \phi(|z|^2\|x\|)|z|^{2ns} \rd^\times z   
    =2|\det(x)|^{\frac{s}{2}}\|x\|^{-ns}\int_0^\infty  \phi(r^2)r^{2ns-1} \rd r   \\
    &=-2|\det(x)|^{\frac{s}{2}}\|x\|^{-ns}\int_0^\infty 
    \phi'(r^2)\cdot 2r \frac{1}{2ns}r^{2ns} \rd r.
    \end{align*}
When $F=\C$,     
    \begin{align*}
    \alpha_s(x)&=|\det(x)|^{\frac{s}{2}} \int_{\C^\times} \phi(|z|^2\|x\|)|z|^{2ns} \rd^\times z    
    =2\pi|\det(x)|^{\frac{s}{2}}\|x\|^{-ns}\int_0^\infty  \phi(r^2)r^{2ns-1} \rd r    \\
    &=-2\pi|\det(x)|^{\frac{s}{2}}\|x\|^{-ns}\int_0^\infty 
    \phi'(r^2)\cdot 2r \frac{1}{2ns}r^{2ns} \rd r.
    \end{align*}
In all cases, we obtain $\lim_{s\to0+} s\,\alpha_s(x)=c$ for some constant $c>0$.

\begin{lemma}\label{geom_evenorth}
There is a constant $c>0$ such that
    \[
    \lim_{s\to0+} s\cdot (M(s)\varphi_{s,  f}(1_{4n}))(1_{2n})
    =c\,\bJ^\theta(f)
    \]
for $f\in\Cc(G(F)/Z_G(F))$.
Here, $\bJ^\theta(f)\coloneqq\int_{F^\times \bs X'(F)}f(x)\, \rd^\times x/\rd^\times z$. 
\end{lemma}

\subsubsection{The spectral side}
Suppose that $\re(s)>0$ is sufficiently large so that the intertwining operator $(M(s)\varphi_{s,  f}(1_{4n}))(1_{2n})$ is given by the convergent integral
    \[
    \int_{N^\sharp(F)} (\varphi_{s, f}(w^{-1}u))(1_{2n}) \rd u.
    \]
Applying the Harish-Chandra Plancherel formula to $\varphi_{s, f}(w^{-1}u)\in\Cc(G(F)/Z_G(F))$,  we have
    \[
    (\varphi_{s,  f}(w^{-1}u))(1_{2n})=\int_{\Temp(G(F)/Z_G(F))} \tr(\pi(\varphi_{s,  f}(w^{-1}u)^\vee)) \rd\mu_{G(F)/Z_G(F)}(\pi)
    \]

Suppose $\pi\in\Temp(G(F)/Z_G(F))$ and $v_1, v_2\in\pi$.
Let $M_\pi(s)\,\colon I_\pi(s)\to I_{w(\pi)}(-s)$ be the standard intertwining operator.
Since we assumed $\re(s)$ is sufficiently large,  
    \[
     (M_\pi(s)(\varphi_{s,  \pi(f^\vee)v_1})(1_{2n}),  v_2)_\pi
     =\int_{N^\sharp(F)}(\varphi_{s,  \pi(f^\vee)v_1}(w^{-1}u),  v_2)_\pi \rd u.
    \]
From the same argument as the previous section, we can show the following.
\begin{lemma}\label{operator_evenortho}
Suppose that  $\pi$ belongs to $\End(G(F),r)$, and $v_1, v_2\in\pi$. 
Then there is a constant $c_\pi\in\C^\times$ with absolute value $1$ such that
    \begin{align*}
    &\lim_{s\to0+}\gamma(s, \pi, \wedge^2, \psi)
    (M_\pi(s)(\varphi_{s,  \pi(f^\vee)v_1})(1_{2n}),  v_2)_\pi \\
    &\hspace{40pt} =c_\pi\,((\varphi_{s,  \pi(f^\vee)v_1})(1_{2n}),  \pi(\theta)v_2)_\pi
    =c_\pi\,(\pi(f^\vee)v_1,  \pi(\theta)v_2)_\pi.
    \end{align*}
Moreover, the map $\pi\mapsto c_{\pi}$ is continuous on $\EL(G(F),r)$, and 
    \[
    \lim_{s\to0+}\gamma(s, \pi, \wedge^2, \psi)
    \int_{N^\sharp(F)} \tr(\pi(\varphi_{s,  f}(w^{-1}u)^\vee)) \rd u
    =c_\pi\, \widehat{f}^\theta(\pi),
    \]
where $\widehat{f}^\theta(\pi)\coloneqq\tr(\pi(\theta)\pi(f^\vee))$.
\end{lemma}

\begin{proposition}\label{spec_evenortho}
There is a constant $c\in\C^\times$ such that for $f\in\Cc(G(F)/Z_G(F))$, we have
 \[
\lim_{s\to0+} s\cdot (M(s)\varphi_{s,  f}(1_{2n}))(1_n)  = c\int_{\pi \in \EL(G(F),r)} \widehat{f}^\theta(\pi) \frac{c_\pi}{|\fS_\pi|}
    \frac{\gamma^\ast(0, \pi, \Ad, \psi)}{\gamma^\ast(0, \pi,  \wedge^2, \psi)} 
    \rd_{\EL(G(F),r)}(\pi).
 \] 
\end{proposition}

\begin{proof}
Let $\rd \mu_{G(F)/Z_G(F)}$ denote the Plancherel measure on $\Temp(G(F)/Z_G(F))$. 
From \cite{Duh19}*{Proposition 3.2}, there is a constant $c\in\C^\times$ such that for all $\Phi\in\cS(\Temp(G(F)/Z_G(F)))$, we have
    \begin{align*}
     &\lim_{s\to0+} s\int_{\Temp(G(F)/Z_G(F))} \Phi(\pi)\gamma(s, \pi, \wedge^2, \psi)^{-1}\rd\mu_{G(F)/Z_G(F)}(\pi) \\
    & \qquad = c \int_{\pi \in \EL(G(F),r)}  \Phi(\pi) 
    \frac{\gamma^\ast(0, \pi, \Ad, \psi)}{\gamma^\ast(0, \pi,  \wedge^2, \psi)} 
    \frac{\rd_{\EL(G(F),r)}(\pi)}{|\fS_\pi|}.
    \end{align*}
The proposition follows from this equation and \cref{operator_evenortho}.
\end{proof}

\subsubsection{Conclusion}\label{sec:con3}

Combining \cref{geom_evenorth} and \cref{spec_evenortho}, we obtain the next theorem.

\begin{theorem}\label{eq_evenortho}
There is a constant $c\in\C^\times$  such that 
    \[
     \bJ^\theta(f) 
    = c\int_{ \pi \in \EL(G(F),r)} \widehat{f}^\theta(\pi) \frac{c_{\pi}}{|\fS_\pi|}
    \frac{\gamma^\ast(0, \pi, \Ad, \psi)}{\gamma^\ast(0, \pi,  \wedge^2,  \psi)} 
    \rd_{\EL(G(F),r)}(\pi)
    \]
for all $f\in\Cc(G(F)/Z_G(F))$.
\end{theorem}

\begin{corollary}\label{cor:llc3}
We can rewrite the above formula in \cref{eq_evenortho} using the local Langlands correspondence (cf. \S\ref{sec:Langlands}). 
There is a constant $c\in\C^\times$  such that 
    \[
     \bJ^\theta(f) 
     = c\int_{\Temp(H(F))}  \widehat{f}^\theta(\el_r(\sigma))
    \frac{c_{\el_r(\sigma)}\gamma^\ast(0, \sigma, \Ad, \psi)}{|S_\sigma|} \rd_{\Temp(H(F))}(\sigma)
    \]
for all $f\in\Cc(G(F)/Z_G(F))$.
\end{corollary}
\begin{proof}
This is a consequence of \cref{eq_evenortho} and \eqref{eq:gamma}.
\end{proof}

\begin{corollary}\label{cor:measureSOodd}
There exists a constant $c\in \C^\times$ such that
\[
\tilde{I}^\theta(f)= c \int_{ \pi \in \EL(G(F),r)} \widehat{f}^\theta(\pi) \frac{c_{\pi}}{|\fS_\pi|}
    \frac{\gamma^\ast(0, \pi, \Ad, \psi)}{\gamma^\ast(0, \pi,  \wedge^2,  \psi)} 
    \rd_{\EL(G(F),r)}(\pi)
\]
for any $f\in C_c^\inf(G(F))$. 
See \S\ref{sec:twistedAL} for the definition of $\tilde{I}^\theta(f)$.  
\end{corollary}
\begin{proof}
    Set $\tilde{f}(g)\coloneqq \int_{F^\times}f(ag)\,\rd^\times a\in C^\infty_c(G(F)/Z_G(F))$. 
    Then, for some constant $c>0$, we have $\tilde{I}^\theta(f)= c\times \bJ^\theta(\tilde{f})$. 
    Hence, from \cref{eq_evenortho}, we obtain the assertion. 
\end{proof}

A measure $\mu_{\EL(G(F),r)}$ on $\EL(G(F),r)$ is defined by
\begin{equation}\label{eq:measureSOodd}
    \mu_{\EL(G(F),r)}(f)\coloneqq  \int_{ \pi \in \EL(G(F),r)} f(\pi) \frac{c_{\pi}}{|\fS_\pi|}
    \frac{\gamma^\ast(0, \pi, \Ad, \psi)}{\gamma^\ast(0, \pi,  \wedge^2,  \psi)} 
    \rd_{\EL(G(F),r)}(\pi), 
\end{equation}
where $f$ is a test function on $\EL(G(F),r)$.

\subsection{Case (I\hspace{-.1em}V)}
Recall that $\chi$ is a quadratic character of $F^\times$.

\subsubsection{The geometric side}
We consider the case of $n$ even and odd separately.
First,  suppose that $n$ is odd.

\begin{lemma}
Suppose that $n$ is odd.
For $(x,  y)\in X'(F)$,  we have $H(-x,  y)^{-1}y=H(x,  y)^{-1}y$ and $\t yA_nH(x,  y)^{-1}y=-2$.
\end{lemma}

\begin{proof}
Set
    \[
    x_0=
        \begin{pmatrix}
        0 & 0 & w_{(n-1)/2} \\
        0 & 0 & 0 \\
        -w_{(n-1)/2} & 0 & 0
        \end{pmatrix},  \qquad
    y_0=\t (0,   \ldots,  0,  1,  0,  \ldots,  0),
    \]
where $y_0$ has $1$ at the $(n+1)/2$-th entry and other entries are zero. 
We see that $(x_0,  y_0)\in X'(F)$ and
    \[
    H(x_0,  y_0)=
        \begin{pmatrix}
        -1_{(n-1)/2} & 0 & 0 \\
        0 & -\frac12 & 0 \\
        0 & 0 & -1_{(n-1)/2}
        \end{pmatrix},  \qquad
    H(-x_0,  y_0)=
        \begin{pmatrix}
        1_{(n-1)/2} & 0 & 0 \\
        0 & -\frac12 & 0 \\
        0 & 0 & 1_{(n-1)/2}
        \end{pmatrix}.
    \]
Hence $H(-x_0,  y_0)^{-1}y_0=H(x_0,  y_0)^{-1}y_0=-2y_0$ and $\t y_0H(x_0,  y_0)^{-1}y_0=2$.

Suppose that $(x,  y)\in X'(F)$ is of the form $(x,  y)=(x_0,  y_0)\cdot g$ for some $g\in G(F)$.
Then we have $H(-x,  y)^{-1}y=\theta(g)^{-1}H(-x_0,  y_0)^{-1}y_0$ and $H(x,  y)^{-1}y=\theta(g)^{-1}H(x_0,  y_0)^{-1}y_0$.
Therefore we obtain $H(-x,  y)^{-1}y=H(x,  y)^{-1}y$.
Similarly we have  $\t yA_nH(x,  y)^{-1}y=\t y_0A_nH(x_0,  y_0)^{-1}y_0=-2$

Since the $G(F)$-orbit of $(x_0,  y_0)$ is dense in $X'(F)$,  they hold for all $(x,  y)\in X'(F)$.
\end{proof}

A direct computation using these equations shows the following corollary.

\begin{corollary}
Suppose that $n$ is odd.
For $(x,  y)\in X'(F)$,  we have
    \begin{align*}
    w^{-1}
        \begin{pmatrix}
        1_n & y & H(x,  y) \\
        0 & 1 & -\t yA_n \\
        0 & 0 & 1_n
        \end{pmatrix}
    &=
        \begin{pmatrix}
        -A_nH(-x,  y)^{-1} & 0 & 0 \\
        0 & 1 & 0 \\
        0 & 0 & -\t A_nH(x,  y)
        \end{pmatrix} \\
    &\qquad \times
        \begin{pmatrix}
        1_n & -y & H(-x,  y) \\
        0 & 1 & \t yA_n \\
        0 & 0 & 1_n
        \end{pmatrix}
        \begin{pmatrix}
        1_n & 0 & 0 \\
        -\t yA_nH(x,  y)^{-1} & 1 & 0 \\
        H(x,  y)^{-1} & H(x,  y)^{-1}y & 1_n
        \end{pmatrix}.
    \end{align*}
\end{corollary}

A similar equation holds when $n$ is even.
We omit the details and only sketch the proof.

\begin{lemma}
Suppose that $n$ is even.
For $(x,  y)\in X'(F)$,  we have $H(-x,  y)^{-1}y=H(x,  y)^{-1}y$ and $\t yA_nH(x,  y)^{-1}y=0$.
\end{lemma}

\begin{proof}
Set
    \[
    x_0=
        \begin{pmatrix}
        0 & w_{n/2} \\
        -w_{n/2} & 0 
        \end{pmatrix},  \qquad
    y_0=\t (1,  0,  \ldots,  0).
    \]
We see that $(x_0,  y_0)\in X'(F)$.
Similarly as in the $n$ odd case,  we can show the desired equations for $(x,  y)=(x_0,  y_0)$ by direct computation and the general case follows since the $G(F)$-orbit of $(x_0,  y_0)$ is dense in $X'(F)$.
\end{proof}

\begin{corollary}
Suppose that $n$ is even.
For $(x,  y)\in X'(F)$,  we have
    \begin{align*}
    w^{-1}
        \begin{pmatrix}
        1_n & y & H(x,  y) \\
        0 & 1 & -\t yA_n \\
        0 & 0 & 1_n
        \end{pmatrix}
    &=
        \begin{pmatrix}
        -A_nH(-x,  y)^{-1} & 0 & 0 \\
        0 & 1 & 0 \\
        0 & 0 & -\t A_nH(x,  y)
        \end{pmatrix} \\
    &\qquad \times
        \begin{pmatrix}
        1_n & y & H(-x,  y) \\
        0 & 1 & -\t yA_n \\
        0 & 0 & 1_n
        \end{pmatrix}
        \begin{pmatrix}
        1_n & 0 & 0 \\
        \t yA_nH(x,  y)^{-1} & 1 & 0 \\
        H(x,  y)^{-1} & -H(x,  y)^{-1}y & 1_n
        \end{pmatrix}.
    \end{align*}
\end{corollary}

Arguing as in the previous cases,  we get
    \begin{align*}
    &(M(s,  \chi)\varphi_{s,  f}(1_{2n+1}))(1_n)
    =\int_{X'(F)} (\varphi_{s,  f}\left(w^{-1}
        \begin{pmatrix}
        1_n & y & H(x,  y) \\
        0 & 1 & -\t yA_n \\
        0 & 0 & 1_n
        \end{pmatrix}
    \right))(1_n) \rd x \rd y \\
    &=\int_{X'(F)}1_L(H(x, y)^{-1})|\det(H(-x, y))|^{-\frac{s+n}{2}}f(-A_nH(-x, y)^{-1}) \rd (x, y) \\
    &=\int_{X'(F)}1_L(\t H(x, y)^{-1})|\det(H(x, y))|^{-\frac{s}{2}}f(-A_nH(x, y)^{-1}) \rd^\times (x, y) \\
    &=\int_{X'(F)}1_L(\t H(x, y))|\det(H(x, y))|^{\frac{s}{2}}f(-A_nH(x, y)) \rd^\times (x, y),    \end{align*}
where $\rd^\times (x, y)\coloneqq |\det(H(x, y))|^{-\frac{n}{2}} \rd (x, y)$ is a $G(F)$-invariant measure on $X'(F)$.
We used the fact that 
    \[
    \delta_{P^\sharp}\left(
        \begin{pmatrix}
        a & 0 & 0 \\
        0 & 1 & 0 \\
        0 & 0 & \theta(a)
        \end{pmatrix}
    \right)=|\det(a)|^n, \qquad a\in G(F).
    \]
The action of $F^\times$ on $X(F)$ is defined by $z\cdot (x,y)=(z^2x,zy)$, $z\in F^\times$, $(x,y)\in X(F)$. 
Hence we have
    \[
    (M_\chi(s)\varphi_{s,  f}(1_{2n+1}))(1_n)
    =   \int_{F^\times\bs X'(F)}
    \alpha_s(H(x,y)) \, f(-A_nH(x,y)) \frac{\rd^\times (x,y)}{\rd^\times z},
    \]
where $\rd^\times z$ is a Haar measure on $F^\times$ and $\alpha_s(h)=|\det(h)|^{\frac{s}{2}} \int_{F^\times}\mathbf{1}_L(z^2\t h)|z|^{ns} \rd^\times z$.  

In the non-archimedean case, the function $\alpha_s(h)$ is given as
    \[
    \alpha_s(h)=|\det(h)|^{\frac{s}{2}}
    q_F^{[\frac{m}{2}]ns}(1-q_F^{ns})^{-1}(1-q_F^{-1}),
    \]
where $m\coloneqq \min(\ord_F(x_{ij}), \ord_F(y_j))$ for $(x, y)\in X'(F)$ with $x=(x_{ij})$, $y=(y_j)$.

When $F=\R$,     
    \begin{align*}
    \alpha_s(h)&=|\det(h)|^{\frac{s}{2}} \int_{\R^\times} \phi(|z|^2\|h\|)|z|^{ns} \rd^\times z  
    =2|\det(h)|^{\frac{s}{2}}\|h\|^{-\frac{ns}{2}}\int_0^\infty \phi(r^2)r^{ns-1} \rd r \\
    &=-2|\det(h)|^{\frac{s}{2}}\|h\|^{-ns}\int_0^\infty 
    \phi'(r^2)\cdot 2r \frac{1}{ns}r^{ns} \rd r.
    \end{align*}
When $F=\C$,     
    \begin{align*}
    \alpha_s(h)&=|\det(h)|^{\frac{s}{2}} \int_{\C^\times} \phi(|z|^2\|h\|)|z|^{ns} \rd^\times z  
    =2\pi|\det(h)|^{\frac{s}{2}}\|h\|^{-\frac{ns}{2}}
    \int_0^\infty \phi(r^2)r^{ns-1} \rd r    \\
    &=-2\pi|\det(h)|^{\frac{s}{2}}\|h\|^{-\frac{ns}{2}}
    \int_0^\infty \phi'(r^2)\cdot 2r \frac{1}{ns}r^{ns} \rd r.
    \end{align*}
In all cases, we obtain $\lim_{s\to0+} s\,\alpha_s(x)=c$ for some constant $c>0$.

\begin{lemma}\label{geom_oddortho}
There is a constant $c>0$ such that
    \[
    \lim_{s\to0+} s\cdot (M_\chi(s)\varphi_{s, f}(1_{2n+1}))(1_n)
    =c\times \bJ^\theta(f)
    \]
for $f\in\Cc(G(F)/Z_G(F))_{\chi}$.
Here, $\bJ^\theta(f)\coloneqq \int_{F^\times\bs X'(F)}  f(-A_nH(x,y)) \rd^\times(x,y)/\rd^\times z$. 
\end{lemma}

\subsubsection{The spectral side}
Suppose that $\re(s)>0$ is sufficiently large so that the intertwining operator $(M_\chi(s)\varphi_{s, f}(1_{2n+1}))(1_n)$ is given by the convergent integral
    \[
    \int_{N^\sharp(F)} (\varphi_{s, f}(w^{-1}u))(1_n) \rd u.
    \]
For each $\pi\in \Temp(G(F))$, we write $\chi_\pi$ for the central character of $\pi$. 
Let $\Temp(G(F))_\chi$ be the subset of representations $\pi\in\Temp(G(F))$ such that $\chi_\pi=\chi$, and denote by $\rd\mu_{G(F), \chi}(\pi)$ the Plancherel measure on $\Temp(G(F))_\chi$.
Applying the Harish-Chandra Plancherel formula to $\varphi_{s,  f}(w^{-1}u)\in\Cc(G(F)/Z_G(F))_{\chi}$, we have
    \[
    (\varphi_{s, f}(w^{-1}u))(1_n)=\int_{\Temp(G(F))_{\chi}} \tr(\pi(\varphi_{s,  f}(w^{-1}u)^\vee)) \rd\mu_{G(F), \chi}(\pi).
    \] 
We have defined the subspace $\EL(G(F),r)_\chi$ by $\EL(G(F),r)_\chi\coloneqq\{ \pi\in \EL(G(F),r) \mid \chi_\pi=\chi \}$. 
By the same argument as in the other cases, we obtain the following.
\begin{lemma}\label{operator_oddortho}
Suppose that  $\pi$ belongs to $\End(G(F),r)_\chi$. 
Then there is a constant $c_{\pi, \chi}\in\C^\times$ with absolute value $1$ such that
    \[
    \lim_{s\to0+}\gamma(s, \pi, \Sym^2, \psi)
    \int_{N^\sharp(F)} \tr(\pi(\varphi_{s,  f}(w^{-1}u)^\vee)) \rd u
    =c_{\pi, \chi}\, \widehat{f}^\theta(\pi),
    \]
where $\widehat{f}^\theta(\pi)\coloneqq \tr(\pi(\theta)\pi(f^\vee))$.
Moreover, the map $\pi\mapsto c_{\pi, \chi}$ is continuous on $\EL(G(F),r)_\chi$.
\end{lemma}

\begin{proposition}\label{spec_oddortho}
There is a constant $c\in\C^\times$ such that for $f\in\Cc(G(F)/Z_G(F))_{\chi}$, we have
    \begin{align*}
     &\lim_{s\to0+} s\cdot (M_\chi(s)\varphi_{s,  f}(1_{2n}))(1_n) \\
    & \qquad = c\int_{ \pi \in \EL(G(F),r)_\chi}  \widehat{f}^\theta(\pi) \frac{\chi(-1)^{n-1}\, c_{\pi, \chi}}{|\fS_\pi|}
    \frac{\gamma^\ast(0, \pi, \Ad, \psi)}{\gamma^\ast(0, \pi,  \Sym^2, \psi)} 
    \rd_{\EL(G(F),r)}(\pi).
    \end{align*}
\end{proposition}

\begin{proof}
We can show that there is a constant $c\in\C^\times$ such that for all $\Phi\in\cS(\Temp(G(F))_{\chi})$, 
    \begin{align*}
     &\lim_{s\to0+} s\int_{\Temp(G(F))_{\chi}} \Phi(\pi)\gamma(s, \pi, \Sym^2, \psi)^{-1}\rd\mu_{G(F),\chi}(\pi) \\
    & \qquad = c \int_{ \pi \in \EL(G(F),r)_\chi }  \Phi(\pi) 
    \frac{\gamma^\ast(0, \pi, \Ad, \psi)}{\gamma^\ast(0, \pi,  \Sym^2, \psi)} 
    \frac{\rd_{\EL(G(F),r)}(\pi)}{|\fS_\pi|}.
    \end{align*}
The proof is almost verbatim to \cite{BP21a}*{Proposition 3.41}, and we omit the detail.
The proposition follows from \cref{operator_oddortho} and this equation.
\end{proof}

\subsubsection{Conclusion}\label{sec:con4}
Combining \cref{geom_oddortho} and \cref{spec_oddortho}, we obtain the next theorem.

\begin{theorem}\label{eq_oddortho}
There is a constant $c\in\C^\times$  such that 
\[
\bJ^\theta(f)=c\int_{\pi \in \EL(G(F),r)_\chi} \widehat{f}^\theta(\pi) \frac{\chi(-1)^{n-1}\, c_{\pi, \chi}}{|\fS_\pi|}
    \frac{\gamma^\ast(0, \pi, \Ad, \psi)}{\gamma^\ast(0, \pi,  \Sym^2, \psi)} 
    \rd_{\EL(G(F),r)}(\pi)
\]
for all $f\in\Cc(G(F)/Z_G(F))_{\chi}$.
\end{theorem}

\begin{corollary}\label{cor:llc4}
We can rewrite the above formula in \cref{eq_oddortho} using the local Langlands correspondence (cf. \S\ref{sec:Langlands}). 
There is a constant $c\in\C^\times$  such that 
    \[
     \bJ^\theta(f) 
     = c\int_{\Temp(H(F))}  
     \widehat{f}^\theta(\el_r(\sigma))
    \frac{\chi(-1)^{n-1}\, c_{\el_r(\sigma), \chi} \, \gamma^\ast(0, \sigma, \Ad, \psi)}{|S_\sigma|} \rd_{\Temp(H(F))}(\sigma)
    \]
for all $f\in\Cc(G(F)/Z_G(F))_{\chi}$.
\end{corollary}
\begin{proof}
This is a consequence of \cref{eq_oddortho} and \eqref{eq:gamma}.
\end{proof}

\begin{corollary}\label{cor:measureSpSOeven}
There exists a constant $c\in \C^\times$ such that
\[
I^\theta(f,\chi)= c \int_{ \pi \in \EL(G(F),r)_\chi} \widehat{f}^\theta(\pi) \frac{\chi(-1)^{n-1}\,c_{\pi,\chi}}{|\fS_\pi|}
    \frac{\gamma^\ast(0, \pi, \Ad, \psi)}{\gamma^\ast(0, \pi,  \Sym^2,  \psi)} 
    \rd_{\EL(G(F),r)}(\pi)
\]
for any $f\in C_c^\inf(G(F))$. 
See \S\ref{sec:vecsp} for the definition of $I^\theta(f,\chi)$.  
\end{corollary}
\begin{proof}
Set $f_\chi(g)\coloneqq \int_{F^\times}\chi(a) f(ag)\,\rd^\times a\in C^\infty_c(G(F)/Z_G(F))_\chi$. 
Let $H'$ denote the mapping $H\colon V\to \M_n$ given in \S\ref{sec:vecsp}. 
Then, we have the relation $H(x,y)=2\, H'(2^{-1}x, 2^{-1}\t y)\, A_n$. 
Hence, for some constant $c>0$, we have $I^\theta(f,\chi)= c\times \bJ^\theta(f_\chi)$. 
Therefore, from \cref{eq_oddortho}, we obtain the assertion. 
\end{proof}

A measure $\mu_{\EL(G(F),r)}$ on $\EL(G(F),r)$ is defined by
\begin{equation}\label{eq:measureSpSOeven}
    \mu_{\EL(G(F),r)}(f)\coloneqq  \sum_\chi \int_{ \pi \in \EL(G(F),r)_\chi} f(\pi) \frac{\chi(-1)^{n-1}\,c_{\pi,\chi}}{|\fS_\pi|}
    \frac{\gamma^\ast(0, \pi, \Ad, \psi)}{\gamma^\ast(0, \pi,  \Sym^2,  \psi)} 
    \rd_{\EL(G(F),r)}(\pi), 
\end{equation}
where $\chi$ runs through all quadratic characters on $F^\times$ and $f$ is a test function on $\EL(G(F),r)$. 
In \cref{thm:maintheorem1}, the measure obtained by restricting $\mu_{\EL(G(F),r)}$ to $\EL(G(F),r)_\chi$ is also written $\mu_{\EL(G(F),r)}$.

\newpage
\part{Density theorems}\label{part:3}

In this part, we will prove a twisted version of Sauvageot's density theorem \cite{Sau97}*{Proposition 7.1}. 
Our density theorem is used to prove the automorphic density theorem of \cref{thm:density}. 
We only need to focus on the generic representations, so our proof will be ad hoc and depend on the classification of the representations of $\GL_n$. 

\section{A variation of the Stone-Weierstrass theorem}

In this section, we will reprove some variant of the Stone-Weierstrass theorem as in \cite{Sau97} using the language of the theory of the Lebesgue integration.

Let $X$ be a locally compact Hausdorff space. 
Write $C_0(X)$ for the space of complex-valued functions on $X$ which tend to $0$ at $\infty$. 
The following is the usual Stone-Weierstrass theorem.
\begin{lemma}[\cite{Rud91}*{5.7}]\label{lem:Stone-Weierstrass}
    Let $A$ be a subalgebra of $C_0(X)$ such that
     \begin{enumerate}
         \item $A$ separates the points of $X$.
         \item For every $x \in X$, there exists $f \in A$ such that $f(x) \neq 0$.
         \item $A$ is closed under the operation taking pointwise complex conjugation.
     \end{enumerate}
    Then, $A$ is dense in $C_0(X)$ with respect to the topology of uniform convergence.
\end{lemma}

Write $C_c(X)$ for the space of continuous compactly-supported functions on $X$.  
Let $\mu$ be a positive Radon measure, and let $\|\; \|_\infty$ denote the sup norm of functions on $X$. 
Note that the topology of uniform convergence agrees with the topology obtained from the supremum metric. 
\begin{lemma}\label{approx-lemma}
    Let $\phi$ be a measurable bounded compactly-supported function.
    Then, for arbitrary positive number $\epsilon > 0$, there exists a functions $g \in C_c(X)$ such that
    \[
         \|g \|_{\infty} \leq \|\phi \|_{\infty} \quad \text{and}
         \quad \mu( |\phi - g| ) \leq \epsilon.
    \]
\end{lemma}
\begin{proof}
    This follows easily from Lusin's theorem \cite{Rud87}*{2.24}.
\end{proof}

Let $f$ be a function on $X$. 
We say that $f$ is Riemann integrable if $f$ is compactly supported and bounded, and $f$ is continuous almost everywhere with respect to $\mu$.    
\begin{lemma}\label{bound-lemma}
    Let $\psi$ be a Riemann integrable function, and suppose $\psi\ge 0$.
    We also take a natural number $n(\psi)$ and a positive number $\epsilon(\psi)$ such that $\|\psi\|_{\infty} \leq n(\psi)$ and $\mu(\psi) \leq \epsilon(\psi)$.
    Then, there exists an element $h \in C_c(X)$ such that $\psi(x) \leq h(x) \leq 4\, n(\psi)$ for almost everywhere $x \in X$. 
    Hence, we have $\mu(h) \leq 4\, n(\psi)\, \epsilon(\psi)$. 
\end{lemma}
\begin{proof}
    See the proof of \cite{Bou04}*{Ch. IV, \S5, n${}^\circ$12, Lemma 5}.
\end{proof}

We deduce the following proposition using these lemmas.
\begin{proposition}\label{Stone-Weierstrass}
    Let $X$ be a locally compact Hausdorff space and let $\mu$ be a positive Radon measure on $X$.
    Let $A$ be a subalgebra of $C_0(X)$ such that
     \begin{enumerate}
         \item[{\rm (1)}] $A$ separates the points of $X$.
         \item[{\rm (2)}] For every $x \in X$, there exists $f \in A$ such that $f(x) \neq 0$.
         \item[{\rm (3)}] $A$ is closed under the operation taking pointwise complex conjugation.
         \item[{\rm (4)}] Every element of $A$ is $\mu$-integrable.
     \end{enumerate}
     Then, for every Riemann integrable function $\phi$ and every real number $\epsilon>0$, there exist elements $g$, $h \in A$ such that 
     \[
     |\phi(x)-g(x)| \leq h(x) \;\; \text{(a.e. $x \in X$)} \quad \text{and} \quad \mu(h) \leq \epsilon.
     \]
\end{proposition}
\begin{proof}
    Let $K$ be the support of $\phi$ and take a relatively compact open neighborhood $V$ of $K$. 
    By the assumptions on $A$, there exists an element $h_0 \in A$ such that
    \[
      h_0\ge 0, \quad \text{and} \quad   h_0(x) \geq 1 \;\; \text{for all $x \in V$}.
    \]
    Take a positive number $\epsilon_1$ such that 
     \begin{align*}
        \epsilon_{1} (\|h_0\|_{\infty} + 3\mu(h_0)) \leq \epsilon. 
     \end{align*}
    In addition, there exists a Riemann integrable function $\phi_1$ such that 
     \begin{align*}
        \phi = \phi_{1} h_0.
     \end{align*}
    By applying Lemma \ref{approx-lemma} to $\phi$, for any $\epsilon_2>0$, we can take a function $g_1\in C_c(X)$ so that $\|g_1\|_\infty \le \|\phi\|_\infty$ and $\mu(|\phi-g_1|)<\epsilon_2$. 
    By applying Lemma \ref{bound-lemma} to $\psi\coloneqq|\phi-g_1|$ and choosing $\epsilon_2$ so that $\epsilon_1=4\, n(\psi)\, \epsilon_2$, we obtain a function $h_1 \in C_c(X)$ such that
    \[
         |\phi(x) - g_1(x)| \leq h_1(x) \;\; \text{(a.e. $x \in X$)} \quad \text{and} \quad 
         \mu(h_1) \leq \epsilon_1.
    \]
    By the usual Stone-Weierstrass theorem (\cref{lem:Stone-Weierstrass}), there exist functions $g'$, $h' \in A$ such that
    \[
         \|g' - g_1\|_{\infty} \leq \epsilon_{1} \quad \text{and} \quad 
         \|h' - h_1\|_{\infty} \leq \epsilon_{1}.
    \]
    We can assume that $h'$ is an everywhere non-negative function.
    Then, putting $g \coloneqq g'h_0$, we have
     \begin{align*}
         |\phi(x) - g(x)| &\leq |\phi_1(x)-g'(x)|h_0(x) \\
                          &\leq |\phi_1(x)-g_1(x)|h_0(x) + |g_1(x)-g'(x)|h_0(x) \\
                          &\leq h_1 h_0(x)+ \epsilon_1 h_0(x) \\
                          &= (h_1h_0(x) - h'h_0(x))+h'h_0(x) +\epsilon_1h_0(x) \\
                          &\leq h'h_0(x) + 2\epsilon_1h_0(x).
     \end{align*}
     Hence, if we put $h \coloneqq h'h_0 + 2\epsilon_1h_0 \in A$, we obtain the result because 
      \[
          \mu(h) \leq \epsilon_1( \|h_0\|_{\infty} + 3 \mu(h_0) )    \leq \epsilon. \qedhere
      \]
\end{proof}

\section{Generic unitary self-dual and conjugate self-dual representations}\label{sec:sau}

\subsection{Notations and fundamental facts}

We collect some notations used in this section and recall some fundamental facts.
First, we briefly recall the same setting as in \S\ref{sec:intro1}. 
Let $F$ be a number field, $S$ a finite set of places of $F$, and set $F_S \coloneqq \prod_{v \in S}F_{v}$. 
Let $E$ be an extension of $F$ with $[E:F]\le 2$.
Denote by $\iota$ the Galois conjugation when $E\neq F$, and the identity mapping when $E=F$. 
Consider the general linear group $G\coloneqq \Res_{E/F}\GL_n$ over the semi-local ring $F_S$ and $\theta$ be the alternating self-dual involution defined in \S\ref{sec:intro1}.
Let $\cH(G(F_S))$ denote the Hecke algebra for $G(F_S)$, which consists of compactly supported smooth $K_S$-finite functions on $G(F_S)$. 
The involution $\theta$ of $G$ over $F$ was defined in \eqref{eq:theta}.

Let $\Irr(G(F_S))$ be the set of isomorphism classes of irreducible representation of $G(F_S)$ and let $\Irr_{\ru}(G(F_S))$ (resp. $\Irr_{\gen}(G(F_S))$, $\Irr_{\gen, \ru}(G(F_S))$) be its subset consisted of unitary (resp. generic, generic unitary) representations. 
Also, let $\Temp(G(F_S))$ denote the subset consisted of tempered representations in $\Irr_\ru(G(F_S))$.
These spaces are endowed with the Fell topologies described in Section 1 of \cite{Tad88} and $\Temp(G(F_S))$ is closed in $\Irr_\ru(G(F_S))$.
It can be identified with the natural topology on $\Irr_{\gen, \ru}(G(F_S))$ induced from the topology of the space of unramified characters (see \cite{BP21a}*{p.245, l.32}).

For a finite (resp. archimedean) place $v$ of $F$, let $\Omega(G(F_v))$ be the set of conjugacy classes of cuspidal pair for $G(F_v)$, described in Section 1 of \cite{Tad88} (resp. the set of infinitesimal characters for $G(F_v)$).
These spaces have the natural algebraic varieties structure over $\C$, and hence, they are endowed with the natural locally compact Hausdorff topologies. 
Let $\Omega(G(F_S))$ be the product $\prod_{v \in S}\Omega(G(F_v))$ of these spaces.

By the superscript by $\theta$, we mean the $\theta$-fixed part of a set which $\theta$ acts on. 
Hence, we denote by $\Omega^{\theta}(G(F_S))$ the subset of $\Omega(G(F_S))$ consisted of $\theta$-fixed elements.
This is an algebraic subvariety of $\Omega(G(F_S))$.
Also, the hermitian involution $\pi \mapsto \pi^{*} \coloneqq \overline{\pi}^{\vee}$ induces the anti-holomorphic involution on $\Omega(G(F_S))$.
Let $\Omega_\ru(G(F_S))$ denote the subset consisting of unitary elements.

The following statement is well-known.
\begin{lemma}\label{lem:BZ}
Every irreducible generic representation of $G(F_S)$ is induced from an essentially square-integrable representation of a Levi subgroup of $G(F_S)$.    
\end{lemma}
\begin{proof}
In $p$-adic cases, it follows from the Bernstein-Zelevinsky classification \cite{Zel80}*{9.7, Theorem}.
In archimedean cases, it follows from \cite{Vog78}*{Theorem 6.2} and \cite{GS15}*{Theorem 2.6.1}.    
\end{proof}

For every place $v$, we have the map $\nu_{v}\colon \Irr(G(F_v)) \to \Omega(G(F_v))$ by taking the cuspidal datum of a representation or taking the infinitesimal character of a representation. 
Further, we obtain the map $\nu_S\colon \Irr(G(F_S)) \to \Omega(G(F_S))$. 
The following fact is crucial to our discussion. 
\begin{equation}\label{eq:inj} 
 \text{The restriction $\nu_{S,\gen}$ of $\nu_S$ to $\Irr_{\gen}(G(F_S))$ is injective.} 
\end{equation}
Hence, $\nu_{S,\gen}$ gives an isomorphism onto its image. 
This fact follows from \cref{lem:BZ}. 

\subsection{Image of twisted traces}

In this subsection, we abbreviate $G(F_S)$ as $G$ for simplicity. 
Let $\mathcal{X}(G)$ be the space of unramified characters of $G$ and let $\mathcal{X}^{\theta}_\ru(G)$ be the subset made of $\theta$-fixed and unitary elements.
Let $W$ be the Weyl group of $G$ and let $W^{\theta}$ be the subset of $\theta$-fixed elements.
We call a pair ($M$,$\sigma$) made of a $\theta$-stable standard Levi subgroup $M$ of $G$ over $F_v$ and a $\theta$-elliptic tempered representation $\sigma$, an elliptic datum (see \cite{MW18}*{\S2.12 in Ch.1} for the definition of $\theta$-elliptic representation).
Two elliptic data $(M_1$, $\sigma_{1})$, $(M_2$, $\sigma_{2})$ are said to be equivalent if there exists an element $\chi \in \mathcal{X}^{\theta}_{\ru}(M_2)$ and an element $w \in W^{\theta}$ such that $w(M_1$, $\sigma_{1}) = (M_2, \sigma_{2} \otimes \chi)$.

\subsubsection{The non-archemedean case}

Assume that $S = \{v\}$ for a finite place $v$. 
Denote by $\cA(\Omega^{\theta}(G))$ the space of algebraic functions defined in \cite{Sau97}*{p.165, l.16}). 
This space is $\C$-algebra by pointwise multiplication. 
Let $\{ (M_i,\sigma_{i}) \}_{i \in I}$ be a set of representatives of the equivalence classes of the elliptic data over $F_v$.
Let $P_i$ be the standard $\theta$-stable parabolic subgroup over $F_v$, which has $M_i$ as a Levi component.
We now state a twisted version of \cite{Sau97}*{Th\'eor\`eme 4.2}. 

\begin{proposition}\label{Sauvageot-4.2}
    Let $f_i$ be algebraic functions on $\Omega^{\theta}(G)$.
    We assume that for all but finitely many $i$, the function $f_i$ is equal to $0$.
    We define $f_{i, j} \coloneqq \delta_{i, j}f_i$.
    Then, there exist functions $h_i \in \mathcal{H}(G)$ such that 
    $$ 
    \tr(\pi(\theta)\pi(h_i))
    =
    f_{i, j}(\nu_v(\pi))
    $$
    if $\pi = \ind^{G}_{P_j}(\sigma_{j}\otimes \chi)$ for some $\chi \in \mathcal{X}^{\theta}(M_j)$.
\end{proposition}

\begin{proof}
    The main input of \cite{Sau97}*{Th\'eor\`eme 4.2} is the Bruhat filtration of the Jacquet functors of induced representations. Hence, we can prove this lemma in a similar manner as the paper if we use the twisted version of the Bruhat filtration, proved in Lemma 8.1 of \cite{Rog88}.
\end{proof}

\subsubsection{The archimedean case}

Assume that $S = \{v\}$ for an archimedean place $v$. 
We use the same notations as above.
We write $\cA(\Omega^{\theta}(G))$ for the restriction of Paley-Wiener functions on $\Omega(G)$ to $\Omega^{\theta}(G)$.    
This space is also $\C$-algebra by pointwise multiplication.
In the archimedean case, we have the twisted Paley-Wiener theorem proved by \cite{DM08}.

\begin{proposition}[\cite{DM08}*{Theorem 4}]
    Let $f_i \in \cA(\Omega^{\theta}(G))$.
    We assume that for all but finitely many $i$, the function $f_i$ is equal to $0$.
    We define $f_{i, j} \coloneqq \delta_{i, j}f_i$.
    Then, there exist functions $h_i \in \mathcal{H}(G)$ such that 
    $$ 
    \tr(\pi(\theta)\pi(h_i))
    =
    f_{i, j}(\nu_v(\pi))
    $$
    if $\pi = \ind^{G}_{P_j}(\sigma_{j}\otimes \chi)$ for some $\chi \in \mathcal{X}^{\theta}(M_j)$.
\end{proposition}

\subsubsection{The general case}

From here, we consider the general $S$. 
We define $\cA(\Omega^{\theta}(G(F_S)))$ as the tensor product of the spaces $\cA(\Omega^{\theta}(G(F_v)))$ over $v\in S$.
Let $\{ (M_i,\sigma_{i}) \}_{i \in I}$ be a set of representatives of the equivalence classes of the elliptic data over $F_S$, which are the products of the elliptic data over $F_v$.
By taking the product of the functions in local cases, we obtain the result of the semi-local case.

\begin{theorem}\label{thm:itt}
    Let $f_i$ be elements of $\cA(\Omega^{\theta}(G))$.
    We assume that for all but finitely many $i$, the function $f_i$ is equal to $0$.
    We define $(f_{i, j})_S \coloneqq \delta_{i, j}f_i$.
    Then, there exist functions $h_i \in \mathcal{H}(G)$ such that 
    $$ 
    \tr(\pi(\theta)\pi(h_i))
    =
    (f_{i, j})_S(\nu_S(\pi))$$
    if $\pi = \ind^{G}_{P_j}(\sigma_{j}\otimes \chi)$ for some $\chi \in \mathcal{X}^{\theta}(M_j)$.
\end{theorem}

\subsection{Density theorem}

In \S\ref{sec:measure}, we have defined the subspace $\EL(G(F_v),r_v)$ of $\Irr_\ru^\theta(G(F_v))$. 
In \cref{thm:maintheorem1}, we have introduced the measure $\mu_0$ on $\prod_{v\in S_0} \EL(G(F_v),r_v)$. 
By replacing $S_0$ and $S$ we define the measure $\mu_S$ on $\prod_{v\in S}\EL(G(F_v),r_v)$ in the same way as $\mu_0$.
We only use the following properties of this measure: it is absolutely continuous with respect to the natural Euclidean measure on $\prod_{v\in S_0} \EL(G(F_v),r_v)$ and represented by the continuous function of polynomial growth.

\begin{lemma}\label{lem:swt}
    If we take 
     \begin{itemize}
         \item $X$ as $\Omega_{\ru}^{\theta}(G(F_S))$,
         \item $\mu$ as the total variation of the pushforward of $\mu_S$ to $\Omega_{\ru}^{\theta}(G(F_S))$,
         \item $A$ as the $\C$-algebra generated by the functions $(f_{i,j})_S$,
     \end{itemize}
     then these satisfy the hypothesis of \cref{Stone-Weierstrass}.
\end{lemma}

\begin{proof}
    It is easy to check the hypotheses (1) and (4). 
    The check of hypotheses (2) and (3) in archimedean cases is non-trivial, but it follows from the density of $C^{\infty}_c(\Omega(G))$ in $C_0(\Omega(G))$ and the Fourier-Laplace transform.
    In general case, the check of the hypothesis (2) and (3) follows from the archimedean and non-archimedean cases, as the proof of \cite{Sau97}*{Corollaire 6.2} using the functions constructed in  \cite{Del86}*{3.1 Lemme}.
\end{proof}

By using the assertion \eqref{eq:inj}, \cref{thm:itt}, and \cref{lem:swt} we obatin the following propositions.
\begin{proposition}\label{prop:sav1}
    For every Riemann integrable function $\phi$ on $\Omega^{\theta}_\ru(G(F_S))$ and for every positive number $\epsilon$, there exist functions $h_1, h_2 \in \cH(G(F_S))$ such that 
    \begin{itemize}
        \item  $| \phi(\nu_S(\pi)) - \hat h_1(\pi) | \leq \hat h_2(\pi)$ (a.e. $\pi \in \Irr^{\theta}_{\gen, \ru} (G(F_S))$),
        \item $| \mu_S |\left( \hat h_2 \right)   \leq \epsilon$. 
    \end{itemize}
    Here, $|\mu_S|$ denotes the the total variation of $\mu_S$ and $\mu_S$-integralble functions $\hat h_j$ on $\Irr^{\theta}_{\ru} (G(F_S))$ are defined by $\hat h_j(\pi)\coloneqq \tr(\pi(\theta)\, \pi(h_j))$, $\pi\in \Irr^{\theta}_{\ru} (G(F_S))$ $(j=1$, $2)$.  
\end{proposition}

\begin{proposition}\label{prop:sav2}
 \begin{enumerate}
     \item[{\rm (1)}] Let $\psi$ be a Riemann integrable function on $\Temp^{\theta}(G(F_S))$.
     Then, there exists a Riemann integrable function $\phi$ on $\Omega_{\ru}^{\theta}(G(F_S))$ such that 
     we have $\phi(\nu_S(\pi)) = \psi(\pi)$ for every $\pi \in \Temp^{\theta}(G(F_S))$ and $\phi(\nu_S(\pi))=0$ for every $\pi\in\Omega_{\ru}^{\theta}(G(F_S)) \setminus \nu_S(\Temp^{\theta}(G(F_S)))$.
     \item[{\rm (2)}] 
     Let $C$ be a compact set contained in $\Omega_{\ru}^{\theta}(G(F_S)) \setminus \nu_S(\Temp^{\theta}(G(F_S)))$.
     There exists a function $f \in C_c(\Omega_{\ru}^{\theta}(G(F_S)))$ such that the support of $f$ is contained in $\Omega_{\ru}^{\theta}(G(F_S)) \setminus \nu_S(\Temp^{\theta}(G(F_S)))$, $f\ge 0$ on $\Omega_{\ru}^{\theta}(G(F_S))$, and $f \geq 1$ on $C$.
 \end{enumerate}
\end{proposition}

\begin{proof}
The assertion (1) follows from the fact that the restriction of $\nu_S$ to $\Temp^{\theta}(G(F_S))$ is a proper map and an isomorphism onto its image.
From the proof of (1), the set $\nu_S(\Temp^{\theta}(G(F_S)))$ is closed in $\Omega_{\ru}^{\theta}(G(F_S))$.
Because the set $\Omega_{\ru}^{\theta}(G(F_S)) \setminus \nu_S(\Temp^{\theta}(G(F_S)))$ is open in the locally compact space $\Omega_{\ru}^{\theta}(G(F_S))$, hence we can take the function as in our statement (2).
\end{proof}

\newpage

\bibliography{bibliography}
\bibliographystyle{amsplain}

\end{document}